\documentclass{mcom-l}

\usepackage[normalem]{ulem} 

\usepackage{color}
\usepackage{amsmath}
\usepackage{amsfonts}
\usepackage{amssymb}%
\usepackage{latexsym}
\usepackage{graphicx}
\usepackage{psfrag}
\usepackage{multirow}

\newtheorem{theorem}{Theorem}[section]
\newtheorem{corollary}[theorem]{Corollary}
\newtheorem{definition}[theorem]{Definition}
\newtheorem{lemma}[theorem]{Lemma}
\newtheorem{proposition}[theorem]{Proposition}

\newtheorem{assumption}[theorem]{Assumption}
\newtheorem{example}[theorem]{Example}
\newtheorem{remark}[theorem]{Remark}

\numberwithin{equation}{section}

\usepackage[a4paper]{geometry}
\newcommand{\R}{\mathbb{R}}

\newcommand{\cC}{{\mathcal C}}

\newcommand{\cI}{{\mathcal I}}

\newcommand{\cT}{{\mathcal T}}
\newcommand{\cK}{{\mathcal K}}

\newcommand{\cO}{{\mathcal O}}
\newcommand{\bx}{\mathbf{x}}

\newcommand{\bbb}{\mathbf{b}}
\newcommand{\bC}{\mathbf{C}}

\newcommand{\bI}{\mathbf{I}}
\newcommand{\br}{\mathbf{r}}
\newcommand{\bV}{\mathbf{V}}
\newcommand{\bW}{\mathbf{W}}

\newcommand{\bu}{\mathbf{u}}
\newcommand{\be}{\mathbf{e}}
\newcommand{\bn}{\mathbf{n}}

\newcommand{\bz}{\mathbf{z}}

\newcommand{\bfr}{\mathbf{r}}
\newcommand{\bfd}{\mathbf{d}}
\newcommand{\bg}{\mathbf{g}}

\newcommand{\veps}{\vert \abs \vert}

\newcommand{\C}{\mathbb{C}}

\newcommand{\tOmega}{\widetilde{\Omega}}

\newcommand{\bphi}{\boldsymbol{\phi}}

\newcommand{\re}{{\rm e}}
\newcommand{\ri}{{\rm i}}

\newcommand{\beq}{\begin{equation}}
\newcommand{\eeq}{\end{equation}}
\newcommand{\beqs}{\begin{equation*}}
\newcommand{\eeqs}{\end{equation*}}
\newcommand{\bit}{\begin{itemize}}
\newcommand{\eit}{\end{itemize}}
\newcommand{\ben}{\begin{enumerate}}
\newcommand{\een}{\end{enumerate}}
\newcommand{\bal}{\begin{align}}
\newcommand{\eal}{\end{align}}
\newcommand{\bals}{\begin{align*}}
\newcommand{\eals}{\end{align*}}
\newcommand{\bse}{\begin{subequations}}
\newcommand{\ese}{\end{subequations}}
\newcommand{\bpr}{\begin{proposition}}
\newcommand{\epr}{\end{proposition}}
\newcommand{\bre}{\begin{remark}}
\newcommand{\ere}{\end{remark}}
\newcommand{\bpf}{\begin{proof}}
\newcommand{\epf}{\end{proof}}
\newcommand{\ble}{\begin{lemma}}
\newcommand{\ele}{\end{lemma}}
\newcommand{\bco}{\begin{corollary}}
\newcommand{\eco}{\end{corollary}}
\newcommand{\bex}{\begin{example}}
\newcommand{\eex}{\end{example}}
\newcommand{\bth}{\begin{theorem}}
\newcommand{\enth}{\end{theorem}}

\newcommand{\Com}{\mathbb{C}}

\newcommand{\supp}{\mathop{{\rm supp}}}

\newcommand{\po}{\partial \Omega}

\newcommand{\eps}{\varepsilon}

\newcommand{\nrme}[1]{\N{#1}_{\weight}}

\newcommand{\Vell}{\Ned^\ell}

\newcommand{\ksqeps}{\frac{\wn^2}{\vert\abs\vert}}

\newcommand{\LtO}{L^2(\Omega)}
\newcommand{\bLtO}{{{\bf L}^2(\Omega)}}

\newcommand{\tendi}{\rightarrow \infty}
\newcommand{\tendo}{\rightarrow 0}

\newcommand{\matrixA}{A}
\newcommand{\matrixAeps}{{A}_\abs}
\newcommand{\matrixAepsinv}{\matrixAeps^{-1}}

\newcommand{\matrixAepsell}{A_{\abs}^\ell}

\newcommand{\matrixAepszero}{A_{\abs}^0}

\newcommand{\bfzero}{\mathbf{0}}

\newcommand{\Qepsell}{\proj_{\abs}^{\ell}}

\newcommand{\aeps}{a_\abs}

\def\XXint#1#2#3{{\setbox0=\hbox{$#1{#2#3}{\int}$}
     \vcenter{\hbox{$#2#3$}}\kern-.5\wd0}}

\newcommand*{\N}[1]{\left\|#1\right\|}

\newcommand{\matrixB}{{ B}}
\newcommand{\matrixBepsinv}{{ B}_\abs^{-1}}

\newcommand{\matrixC}{{C}}
\newcommand{\matrixI}{{ I}}

\newcommand{\bfx}{\mathbf{x}}

\usepackage{hyperref}
\definecolor{myblue}{rgb}{0,0,0.6}
\hypersetup{colorlinks=true, linkcolor=myblue,citecolor=myblue,filecolor=myblue,urlcolor=myblue}

\allowdisplaybreaks[4]

\newcommand{\ton}{\text{ on }}
\newcommand{\tin}{\text{ in }}
\newcommand{\tfa}{\text{ for all }}
\newcommand{\tfor}{\text{ for }}
\newcommand{\tas}{\text{ as }}
\newcommand{\tand}{\text{ and }}

\newcommand{\kseb}{\left(\frac{\wn^2}{\veps}\right)}
\newcommand{\eksb}{\left(\frac{\veps}{\wn^2}\right)}

\newcommand{\Hsub}{H_{\rm{sub}}}

\newcommand{\bE}{\mathbf{E}}

\newcommand{\HocO}{{\bf H}_0(\curlt;\Omega)}
\newcommand{\HcO}{{\bf H}(\curlt;\Omega)}

\newcommand{\bv}{{\bf v}}
\newcommand{\bq}{{\bf q}}
\newcommand{\bw}{{\bf w}}
\newcommand{\bvb}{\overline{{\bf v}}}

\newcommand{\bze}{\mathbf{0}}
\newcommand{\curlt}{{\rm{\bf curl}}}
\newcommand{\curl}{{\rm \bf curl }\,}
\newcommand{\dive}{{\rm \bf div }\,}
\newcommand{\grad}{{\rm \bf grad }\,}
\newcommand{\noi}{\noindent}

\newcommand{\bF}{\mathbf{F}}
\newcommand{\bG}{\mathbf{G}}
\newcommand{\bS}{\mathbf{S}}
\newcommand{\bQ}{\mathbf{Q}}

\newcommand{\boldT}{\mathbf{T}}
\newcommand{\bPsi}{\mathbf{\Psi}}
\newcommand{\bPhi}{\mathbf{\Phi}}
\newcommand{\Hdiv}{{{\bf H}({\rm {\bf div}};\Omega)}}
\newcommand{\Hdivo}{{{\bf H}(\dive^0;\Omega)}}
\newcommand{\Hdivoo}{{{\bf H}_0(\dive;\Omega)}}

\newcommand{\Ccont}{}

\newcommand{\Cstab}{}
\newcommand{\Cdual}{}
\newcommand{\Ccs}{}
\newcommand{\Csub}{}
\newcommand{\Creg}{}
\newcommand{\Crego}{}
\newcommand{\Cregt}{}
\newcommand{\Hcs}{H}
\newcommand{\wn}{k}
\newcommand{\abs}{\xi}
\newcommand{\Ned}{{\rm {\bf Q}}}
\newcommand{\RT}{{\rm {\bf V}}}
\newcommand{\dual}{\mathbf{s}}

\newcommand{\normeps}{}
\newcommand{\weight}{\curlt, \wn}

\newcommand{\proj}{\boldT}

\newcommand{\Nsub}{N_{\rm{sub}}}

\newcommand{\ncs}{n_{\rm{cs}}}
\newcommand{\Nedelec}{N\'{e}d\'{e}lec }

\usepackage{soul}

\definecolor{amcol}{rgb}{0.8,0,0}

\newcommand{\es}[1]{{#1}}
\newcommand{\igg}[1]{{#1}}
\newcommand{\mb}[1]{{#1}}

\newcommand{\xiprob}{\xi_{\mathrm{prob}}}
\newcommand{\xiprec}{\xi_{\mathrm{prec}}}

\newcommand{\neweq}{\simeq}

\begin{document}

\title[Domain decomposition for high-frequency Maxwell with absorption]
{Domain decomposition  preconditioning  for  the 
high-frequency 
time-harmonic Maxwell equations with absorption }

\author{M. Bonazzoli}
\address{\mb{Sorbonne Universit\'e, Universit\'e Paris-Diderot SPC, CNRS, Inria, Laboratoire Jacques-Louis Lions, \'equipe Alpines, F-75005 Paris}}
      \email{\mb{marcella.bonazzoli@inria.fr}}
%
\author{\,\,V. Dolean}
\address{Department of Mathematics and Statistics, University of Strathclyde, Glasgow, G1 1XH, UK, and 
Universit\'e C\^ote d'Azur, CNRS, Laboratoire J-A Dieudonn\'e, France
}
\email{victorita.dolean@strath.ac.uk}
\thanks{}

\author{\,\,I.~G.~Graham}
\address{Department of Mathematical Sciences, University of Bath,
Bath BA2 7AY, UK.}
\curraddr{}
\email{I.G.Graham@bath.ac.uk}
\thanks{}

\author{\,\,E.~A.~Spence}
\address{Department of Mathematical Sciences, University of Bath,
Bath BA2 7AY, UK.}
\email{E.A.Spence@bath.ac.uk}
\thanks{} 

\author{\,\,P.-H. Tournier}
\address{\mb{Sorbonne Universit\'e, Universit\'e Paris-Diderot SPC, CNRS, Inria, Laboratoire Jacques-Louis Lions, \'equipe Alpines, F-75005 Paris}}
\email{tournier@ljll.math.upmc.fr}

\begin{abstract}
\igg{This paper rigorously analyses preconditioners for  the time-harmonic Maxwell equations with absorption, where the  PDE is discretised using curl-conforming finite-element methods of fixed, arbitrary order and the 
preconditioner is constructed  using Additive Schwarz domain decomposition methods.
The theory developed here shows that if the absorption is large enough, 
and if the subdomain and coarse mesh diameters and overlap  are chosen appropriately, 
then the classical two-level overlapping Additive Schwarz preconditioner 
(with PEC boundary conditions on the subdomains) performs optimally  -- in the sense that GMRES converges in a wavenumber-independent number of iterations -- for the problem with absorption.
An important feature of the theory is that it allows the 
coarse space to be built from low-order elements even if the PDE is discretised using high-order elements. It
also shows that additive methods with minimal overlap can be robust.
Numerical experiments are given that  illustrate the theory and its dependence on various parameters.
These experiments motivate some extensions of the preconditioners which have better robustness for problems with less absorption, including the propagative case. At the end of the paper we illustrate the performance of  these on     two 
substantial  applications; the first (a problem with absorption arising from medical imaging) shows the empirical  robustness of the preconditioner against heterogeneity, and the second (scattering by a COBRA cavity) shows good scalability of the preconditioner with up to 3,000 processors. }

\


\


\end{abstract}

\maketitle

\section{Introduction}\label{sec:intro}

The construction of fast iterative solvers for the indefinite time-harmonic Maxwell system at high-frequency is a problem of great current interest. Some of the difficulties that arise are similar to those encountered for the high-frequency Helmholtz equation, and in this paper we investigate how Domain Decomposition (DD) solvers recently proposed for the high-frequency Helmholtz equation work in the Maxwell case. These solvers are built from preconditioners for the discretised boundary value problem (BVP) with added absorption. In the Helmholtz context, the idea of preconditioning with absorption originated in \cite{ErVuOo:04}, and is often called ``shifted Laplacian preconditioning" (see, e.g., the recent collection of articles \cite{LaTaVu:17}). The paper \cite{GrSpVa:17} proposed a two-level domain decomposition preconditioner for the Helmholtz equation based on adding absorption, and, from the point of view of analysis, the present paper is roughly-speaking the Maxwell analogue of this Helmholtz work. We emphasise, however, that the Maxwell theory in the present paper entails additional technical difficulties and we need to use both the theory for the ``positive-definite" case in \cite{HiTo:00}, \cite{To:00}, \cite{PaZh:02} and the theory for the indefinite case (for small wavenumber) in \cite{GoPa:03}; in particular, the technical heart of the paper, \S\ref{sec:43}, contains wavenumber-explicit analogous of results in \cite{GoPa:03}. 

\subsection{The Maxwell boundary value problems}

The time-harmonic Maxwell equations, written in terms of the (time-harmonic) electric field $\bE$, are
\beq\label{eq:new1}
\curl \left(\frac{1}{\mu}
\curl \bE\right) -\left( \eps \omega^2 +\ri \sigma\omega\right) \bE = \bG,
\eeq
where $\bG$ is the source term, 
$\mu$ is the magnetic permeability, $\eps$ is the dielectric constant, $\sigma$ is the conductivity, 
and $\omega$ is the angular frequency; \eqref{eq:new1} arises from assuming in the time-dependent Maxwell equations that the fields and source depend on time via $\re^{-\ri \omega t}$.

In the case of propagation through a homogeneous medium,
$\mu=\mu_0$, $\eps=\eps_0$, and $\sigma=\sigma_0$, where  
$\mu_0$, $\eps_0$, and $\sigma_0$ are all positive constants. Then, with the wavenumber $k$ defined by $k:=\omega \sqrt{\eps_0\mu_0}$ and $\bF:= \mu_0\bG$, \eqref{eq:new1} becomes
\beq\label{eq:new2}
\curl (\curl \bE) - \left(k^2 + \ri k \sigma_0 \sqrt{\frac{\mu_0}{\eps_0}}\right)\bE = \bF.
\eeq
(The quantity $\sqrt{{\mu_0}/{\eps_0}}$ is often called the {\em vacuum impedance}.) 
We consider domain-decomposition preconditioning for finite-element discretizations of the more general partial differential equation (PDE)
\beq\label{eq:2}
\curl\left(
\curl \bE\right) - (\wn^2+ \ri \abs) \bE = \bF
\eeq
with $\abs \in \mathbb{R}\setminus\{0\}$. 
Our theory is for the PDE \eqref{eq:2} posed in a bounded  Lipschitz polyhedron $\Omega$ with the perfect electric conductor (PEC) boundary conditions
\beq
\label{eq:PECcond}
\bE\times \bn =\bze \quad \ton \po.
\eeq
We also give numerical experiments, however, for the boundary value problem (BVP) of \eqref{eq:2} supplemented with the impedance boundary conditions 
\beq\label{eq:imp}
(\curl\bE)\times \bn - \ri \ \text{sign}(\abs) k  (\bn \times \bE)\times \bn = \bze  \quad \ton \po;
\eeq
we discuss the prospects of extending our theory to this case in Remark \ref{rem:imp}.

In the case that the ``absorption parameter" $\abs$ equals $k\sigma_0 \sqrt{ \mu_0/\eps_0}$, \eqref{eq:2} becomes \eqref{eq:new2}.
Although one of our numerical experiments concerns a practical heterogenous problem from medical imaging where $\sigma>0$ (see \S\ref{subsec:medimax}),
our main motivation for considering discretisations of \eqref{eq:2} is as preconditioners for the ``indefinite" Maxwell problem
\beq\label{eq:1}
\curl \left(
\curl \bE\right) - k^2\bE = \bF,
\eeq
with the same boundary conditions as prescribed for \eqref{eq:2}. 

\subsection{Analysing preconditioning with absorption}\label{sec:analyse}

We now discuss the rationale of preconditioning discretisations of BVPs involving the PDE \eqref{eq:1} with discretisations of \eqref{eq:2}. 
Much of this discussion is independent of the boundary conditions in the BVP, but we highlight when they play a role. 
We denote the Galerkin matrix arising from discretising \eqref{eq:2} with Ned\'el\'ec elements (of arbitrary order) by $A_\abs$, and thus the corresponding Galerkin matrix of \eqref{eq:1} is $A_0$, which we also denote by $A$.
Solving the linear system $A\bx = \bbb$ is difficult when $k$ is large because (i) the dimension of the matrix $A$ must grow at least like $k^3$ as $k$ increases to resolve the oscillations of the solution, (ii) $A$ is indefinite when $\wn$ is sufficiently large, and (iii) $A$ is non-Hermitian, and in general non-normal, when the BVP contains an impedance boundary condition.
With system matrices like these, general iterative methods like preconditioned (F)GMRES have to be employed. However, analyzing the convergence of these methods is hard, since an analysis of the spectrum of the system matrix alone is not sufficient for any rigorous convergence estimates.

We therefore want to find a ``good" preconditioner for $A$, in the sense that we would like the number of iterations needed to solve the preconditioned system to be independent of $\wn$, and we would also like the preconditioner to be, roughly speaking, as parallelisable as possible. 

Our preconditioning strategy (written in left-preconditioning form) is to
iteratively solve
 \beqs
\matrixBepsinv \matrixA \bx = \matrixBepsinv\bbb, 
\eeqs
where $B_\abs^{-1}$ is  an  approximation   of  
$A_\abs^{-1}$ computed using DD. The idea is that as $|\abs|$ increases it becomes easier to calculate a good approximation of $A_\abs^{-1}$ (since the problem becomes less wave-like and more ``elliptic"\footnote{Statements like this frequently appear in the literature; one way of understanding them rigorously is via the coercivity result of Lemma \ref{lem:coer} below.}), but $|\abs|$ cannot increase too much, otherwise $A_\abs^{-1}$ is ``too far away" from $A^{-1}$.
To understand this better, we write
\beq\label{eq:key_intro}
\matrixBepsinv\matrixA  = \matrixBepsinv \matrixA_{\abs}
 - \matrixBepsinv \matrixA_{\abs} (\matrixI - \matrixA_{\abs}^{-1} 
\matrixA),
\eeq
and recall that a sufficient (but by no means necessary) condition
for GMRES applied to a matrix $C$ to converge is that the \emph{field of values} (also called the \emph{numerical range}) of $C$ is bounded away from the origin (see \S\ref{sec:52} below for more discussion on this).
It is therefore clear from \eqref{eq:key_intro} that sufficient conditions  for  
$B_\eps^{-1}$ to be a good preconditioner for $A$ are: 
\begin{quotation} 
(i) $\matrixAepsinv$ is  a good
preconditioner for $\matrixA$ in the sense that $\|I -  \matrixA_{\abs}^{-1}A\|$ is small, independently of $\wn$,
\end{quotation} 
and 
 \begin{quotation} 
(ii) $\matrixBepsinv$ is  a good
preconditioner for $\matrixAeps$ in the sense that both the norm and the distance of the field of values from the origin of $\matrixB_\abs^{-1}\matrixAeps$ are bounded independently of $\wn$.
 \end{quotation} 
From the discussion above, we expect (i) to be achieved when $|\abs|$ is sufficiently small, and (ii) to be achieved when $|\abs|$ is sufficiently large.
Furthermore, we expect the boundary conditions in the BVP to affect the answer to (i). 
 For the PEC BVP, one expects (i) to be achieved when $|\abs|$ is sufficiently small and $\wn$ is such that the BVP with $\abs=0$ has a unique solution. 
For the impedance BVP, one expects (i) to be achieved when $|\abs|$ is sufficiently small, for all $\wn>0$.

\subsection{Previous results on achieving (i) and (ii) for the Helmholtz equation}

In case of the Helmholtz equation $\Delta u + (\wn^2+ \ri \abs)u=f$, \cite{GaGrSp:15} showed that, for the interior impedance and truncated exterior Dirichlet BVPs (which both have unique solutions for all $\wn>0$), (i) holds when $|\abs|/\wn$ is sufficiently small. 
The message from the combination of \cite{CoGa:17}, \cite{GrSpVa:17}, and \cite{ErGa:12} is that one needs $|\abs|\sim\wn^2$ for standard ``coercive elliptic" technology such as multigrid and classical additive Schwarz DD (with Dirichlet boundary conditions on the subdomains) to work in a $\wn$-independent way on $A_\abs$; indeed, \cite{GrSpVa:17} shows that (ii) holds for classical additive Schwarz DD with $|\abs|\sim\wn^2$, and \cite{CoGa:17} and \cite{ErGa:12} show that one needs $|\abs|\sim\wn^2$ for multigrid
to converge in a $\wn$-independent number of iterations.

One advantage of DD over multigrid in this context is that ``wave-based" components such as impedance or PML boundary conditions on the subdomains can more-easily be incorporated into DD preconditioners (see \S\ref{sec:discussion} below). The numerical experiments in \cite{GrSpVa:17} and \cite{GrSpVa:17a} (following earlier experiments in \cite{KiSa:07} and \cite{KiSa:13}) show that additive Schwarz DD preconditioners for $A$ can perform well for $\wn \lesssim \abs \ll \wn^2$ if 
\emph{impedance} boundary conditions are used on the subdomain problems, instead of Dirichlet ones, and these experiments are backed up by analysis in
\cite{GrSpZo:17} that shows that Property (ii) can be satisfied in some situations with $|\abs| \sim \wn^{1+\beta}$ for $\beta$ small.

\subsection{The main contributions of this paper}\label{sec:contributions}

\

\noindent\textbf{Theory:} The theoretical heart of this paper are results giving sufficient conditions for property (ii) above to hold 
when the classical two-level overlapping additive Schwarz DD preconditioner, with PEC boundary conditions on the subdomains, is applied to the 
PDE \eqref{eq:2} posed in a convex polyhedron $\Omega$ with a PEC boundary condition on $\partial \Omega$; these results are given in \S\ref{sec:53} - see Theorems \ref{cor:final2} and \ref{cor:final4}. 
This theory shows that property (ii) is achieved 
(i.e.~GMRES converges in a $\wn$-independent number of iterations when applied to $B_\abs^{-1} A_\abs$)
when $|\abs|\sim\wn^2$ and when the subdomain and coarse-grid diameters $H \sim \wn^{-1}$ and the overlap $\sim H$.
Another case that is robust is when the fine and coarse grid each have a fixed number of degrees of freedom
per wavelength and the  overlap is ``minimal'', i.e. of the order of the fine mesh diameter $ h$.  
 Our theory also gives results for $0<|\abs|\ll \wn^2$, but  these results are not sharp (especially in their restrictions on the subdomain and coarse-grid diameters).

We do not prove any results directly about property (i). However we expect that, at least for the Maxwell interior impedance problem, property (i) above holds when $|\abs|/\wn$ is sufficiently small, just as in the Helmholtz case. Indeed, in Appendix \ref{app:1} we prove the PDE-analogue of this result, i.e.~that the relative error between the solutions of \eqref{eq:1} and \eqref{eq:2} (with the same impedance boundary condition) is bounded as $\wn\tendi$ when $|\abs|/\wn$ is sufficiently small. The Helmholtz analogue of this result is proved in \cite[Theorem 6.1]{GaGrSp:15}, leading us to expect that the conditions for property (i) to hold for Maxwell are the same as for Helmholtz. 

We note that, with these results, there is a ``gap" between the $\abs$ for which Property (ii) is proved to hold ($|\abs|\sim\wn^2$), and the $\abs$ for which we expect Property (i) to hold ($|\abs|\lesssim \wn$), but we expect this gap will decrease when the PEC boundary conditions on the subdomains are replaced by impedance boundary conditions, just as in the Helmholtz case \cite{GrSpZo:17}; this will be rigorously explored in future work.

\igg{\noindent\textbf{Numerical Experiments:} 
 Numerical experiments are given in \S\ref{sec:num}.}
\igg{In \S \ref{sec:6.2} we illustrate our theoretical results concerning  property (ii), and the effect of various
    choices of absorption,  
    coarse grid  diameter,  subdomain diameter and   overlap. We also 
investigate the effect of the
    choice of inner product on the convergence of GMRES.} 

\igg{The theory helps  us identify alternative preconditioners (not covered by the theory) that  have better robustness properties, especially in the propagative case $\abs = 0$. These are introduced in \S \ref{sec:6.3}.  In \S \S
 \ref{subsec:medimax} and \ref{subsec:cobra}, the alternative preconditioners are applied to two substantial practical examples.}
\igg{For 
    a heterogeneous  absorptive problem     
  (arising from medical imaging), 
  preconditioning $A_\abs$ with $B^{-1}_\abs$ is robust with respect to the variation of $\eps, \mu,$ and $\sigma$. In particular, the performance of the analogous one-level preconditioner deteriorates 
significantly when $\sigma=0$ in some areas of the domain (in comparison to the situation where $\sigma>0$ everywhere), but the performance of  $B^{-1}_\abs$ does not; see \S\ref{subsec:medimax}.}
 \igg{On the so-called ``COBRA cavity" test case of scattering by a 3-d waveguide-like cavity, 
    preconditioning $A$ with $B^{-1}_\abs$ exhibits good scalability with up to 3000 processors; see \S\ref{subsec:cobra}.}

\subsection{Discussion of the novelty of our results in the context of the DD literature}\label{sec:discussion}

\

\

\noindent\textbf{Classic overlapping Schwarz methods for ``indefinite" Maxwell and Helmholtz.} Cai and Widlund \cite{CaWi:92} proved results about overlapping additive Schwarz preconditioners (with Dirichlet boundary conditions on the subdomains) for non-symmetric indefinite  second-order linear elliptic PDEs that are ``close to" a symmetric coercive PDE, i.e.~this theory treats the Helmholtz equation as a perturbation of the Laplace equation.
The disadvantage of this theory applied to the Helmholtz equation is that it essentially requires quasi-optimality of the coarse-grid problem, and 
then because of the pollution effect this requires that the coarse grid diameter $\Hcs$ should satisfy $\Hcs \ll k^{-1}$
(recall that with impedance boundary conditions on $\po$, quasioptimality is proved with $\Hcs \lesssim k^{-2}$ \cite[Proposition 8.2.7]{Me:95}, and 
even in 1-d this is sharp; see, e.g., the literature review on \cite[Pages 182 and 183]{GrLoMeSp:15}).
The problem of giving a theory of DD for the Helmholtz equation with practical coarse grid sizes of $\Hcs\gtrsim 1/k$ is therefore still open.

The Maxwell analogue of \cite{CaWi:92} is the work by Gopalakrishnan and Pasciak \cite{GoPa:03}. Indeed, via a perturbative approach, these authors studied classical additive Schwarz DD preconditioners (involving PEC boundary conditions on the subdomains) for \eqref{eq:1} with sufficiently small $k$ and posed in a convex polyhedron $\Omega$ with a PEC boundary condition on $\Omega$ (a similar approach was then applied to multigrid in \cite{GoPaDe:04}).

Our theory is the Maxwell analogue of \cite{GrSpVa:17}. 
The Helmholtz theory in \cite{GrSpVa:17} drew inspiration from the work of Cai and Widlund \cite{CaWi:92}, but did not use any specific technical results from that work.
The Maxwell theory in the present paper, however, uses in an essential way some of the technical results in \cite{GoPa:03}, modifying them appropriately using results about the Maxwell equations with absorption (see \S\ref{sec:43} below).

\

\noindent\textbf{Better boundary conditions on the subdomains: impedance, PML, optimised Schwarz.} 
Both overlapping and non-overlapping Schwarz methods usually perform much better for high-frequency Helmholtz and Maxwell problems if the Dirichlet/PEC boundary conditions on the subdomains are changed. 
In the Helmholtz case, impedance boundary conditions on the subdomains were proposed for non-overlapping methods in \cite{De:91, BeDe:97}, and
were then used in the overlapping case by
\cite{CaCaElWi:98}, and more recently in \cite{KiSa:07, KiSa:13} (with the latter
 explicitly using absorption) 
Perfectly matched layers on the subdomains were used in the overlapping case in \cite{To:98}, and in the non-overlapping case in \cite{ScZs:07}.

The optimal boundary conditions on the subdomains  involve the Dirichlet-to-Neumann operator \cite{EnZh:98}
and \emph{Optimised Schwarz methods}
then use various approximations of this operator to produce good, practical methods; 
see
\cite{DoGaGe:09, DoGaLaLePe:15, BoDoGaLa:12} for 
optimised Schwarz methods for Maxwell (and the latter about Maxwell with absorption).

One thing to emphasise from the analysis point of view is that all analysis of optimised Schwarz methods is essentially \emph{either} for the case of two half-planes, using the Fourier transform (see, e.g., \cite{GaMaNa:02}, \cite{DoGaGe:09}, \cite{GaZh:16a}), \emph{or} for the case of a circle embedded in the plane, using Fourier series (see, e.g., \cite{BoAnGe:12}). \igg{Furthermore, in the setting of frequency-domain wave propagation, 
there does not yet exist a framework for combining the analysis of optimised subdomain boundary conditions with coarse grid operators, or for giving a convergence theory explicit in subdomain or coarse-grid size.   }

\

\noindent\textbf{The search for good coarse spaces for wave problems.}
Whereas the design of good coarse spaces is relatively well-understood for homogeneous coercive self-adjoint problems, the design of practical coarse spaces for the Helmholtz and Maxwell equations is a largely open problem; the reason for this is that, as described above, conventional piece-wise polynomial coarse spaces require $\Hcs \sim k^{-2}$ for quasi-optimality of the coarse-space problem.
One approach to obtain practical coarse spaces is to use coarse spaces with oscillatory basis functions, such as
the plane-wave coarse spaces of Farhat and collaborators \cite{FaMaLe:00}, \cite{FaLeLePiRi:01}, and coarse spaces based on solutions of eigenproblems involving the Dirichlet-to-Neumann operator on the subdomain interfaces \cite{CoDoKrNa:14} (and thus applicable also in the case of variable wavenumber). 
\igg{This paper considers a second approach, initiated in \cite{GrSpVa:17} for the Helmholtz equation, 
namely applying conventional piecewise-polynomial coarse spaces to the problem with absorption added; these two approaches were compared in the Helmholtz case in \cite{BoDoGrSpTo:17}.}

\section{The variational formulation, Galerkin method, and other preliminary results}
\label{sec:Var} 

\subsection{Variational formulation}\label{sec:var1}

Let $\Omega$ be a bounded, simply connected Lipschitz domain in 
$\mathbb{R}^3$; 
the vast majority of our results will be for the particular case that 
$\Omega$ is a Lipschitz polyhedron, but we indicate below when we make this assumption.
Let
\beqs
\HcO := \big\{ \bv \in {\bf L}^2(\Omega) : \curl \bv \in {\bf L}^2(\Omega)\big\}.
\eeqs
Let $\bn$ denote the outward-pointing unit normal vector to $\Omega$ (which is defined almost everywhere on $\po$), and recall that the tangential trace $\bv\times\bn$ is well-defined for $\bv \in \HcO$ (see, e.g, \cite[Theorem 3.29]{Mo:03}). 

The theory in this paper concerns the PDE \eqref{eq:2} in $\Omega$ with the PEC boundary condition $\bE \times \bn=\bze$ on $\po$. For this theory, we therefore 
work in the space $\HocO$ defined by 
\beqs
\HocO := \{ \bv \in {\bf L}^2(\Omega) : \curl \bv \in {\bf L}^2(\Omega), \bv\times \bn=\bze\}.
\eeqs
We define the weighted inner product and norm on $\HocO$ by 
\beq\label{eq:weight}
(\bv,\bw)_{\weight}  =  (\curl \bv, \curl \bw)_{{\bf L}^2(\Omega)} + \wn^2 (\bv,\bw)_{{\bf L}^2(\Omega)} \quad  \text{and}  \quad 
\N{ \bv}_{\weight} = (\bv,\bv)_{\weight}^{1/2}.  
\eeq
The standard variational formulation of the BVP 
\begin{equation}
\curl\left(
\curl\bE\right) - (\wn^2 + \ri \abs) \bE = \bF \,\,\tin \Omega,
\qquad \bE \times \bn = \bze \,\,\ton \po
\label{eq:BVP}
\eeq
is: 
given $\bF \in {\bf L}^2(\Omega)$, $\abs\in \mathbb{R}$ and $\wn>0$, 
\beq\label{eq:vf_intro}
\text{find}\,\, \bE \in \HocO\,\, \text{ such that } \,\,a_\abs(\bE,\bv) = F(\bv) \,\,\text{ for all }\,  \bv \in \HocO,
\eeq
where 
\beq\label{eq:Helmholtzvf_intro}
a_\abs(\bE,\bv) := \int_\Omega 
\curl  \bE\cdot \overline{\curl \bv}  - (\wn^2 + \ri \abs)\int_\Omega  \bE\cdot\bvb
\quad \tand\quad 
F(\bv) := \int_\Omega \bF\cdot\bvb.
\eeq
When $\abs = 0$ and the PDE is \eqref{eq:1},
we simply write 
$a(\cdot,\cdot)$ instead of $a_\abs(\cdot,\cdot)$. 
When $|\abs|>0$, the solution of \eqref{eq:vf_intro} exists for all $\wn>0$ by the continuity and coercivity results in \S\ref{sec:prop} below. When $\abs=0$, the solution of \eqref{eq:vf_intro} exists for all but a countable set of $\wn$; see, e.g, \cite[Corollary 4.19]{Mo:03}.

The adjoint of the sesquilinear form $a_\abs(\cdot,\cdot)$, denoted by $a_\abs^*(\cdot,\cdot)$ is defined by 
\beq\label{eq:adj_ses}
a_\abs^*(\bE,\bv)=\overline{a_\abs(\bv,\bE)};
\eeq
see, e.g., \cite[Equation 2.27]{SaSc:11}. Thus, $a_\abs^*(\cdot,\cdot)$ is given by
\beq\label{eq:adjoint_form}
a^*_\abs(\bE,\bv) := \int_\Omega 
\curl  \bE\cdot \overline{\curl \bv}  - (\wn^2 - \ri \abs)\int_\Omega \bE\cdot\bvb,
\eeq
and one can check that the variational problem \eqref{eq:vf_intro} with $a_\abs(\cdot,\cdot)$ replaced by $a_\abs^*(\cdot,\cdot)$ is equivalent to the BVP
\begin{equation}
\curl\left(
\curl\bE\right) - (\wn^2 - \ri \abs) \bE = \bF \,\,\tin \Omega, \qquad
\bE \times \bn = \bze \,\,\ton \po; \label{eq:BVPadj}
\eeq
we refer to this BVP as the \emph{adjoint BVP}.

Finally we note that the BVPs \eqref{eq:BVP} and \eqref{eq:BVPadj} are usually posed with the source term $\bF$ in $\Hdivo$, where
\beqs
\Hdivo:= \big\{ \bu \in {\bf L}^2(\Omega) : \dive \bu =0\big\}.
\eeqs
The results in this paper concern the matrix arising from the Galerkin discretisation of the variational problem \eqref{eq:vf_intro}; these results are therefore independent of $\bF$ in \eqref{eq:BVP}/\eqref{eq:Helmholtzvf_intro} and so we make no assumption on $\bF$ in  in \eqref{eq:BVP}/\eqref{eq:Helmholtzvf_intro} other than that it is in $\bLtO$.

\bre\textbf{\emph{(Alternative convention for adding absorption)}}
Here we have chosen to add absorption in the form \eqref{eq:2}; our results can easily be translated into the case when one preconditions discretisations of the PDE \eqref{eq:1} with discretisations of 
$\curl(
\curl \bE) - (\wn+ \ri \eta
)^2 \bE = \bF$
for a different absorption parameter $\eta\geq0$. 
\ere

\subsection{Discretisation of the variational problem by edge finite elements}\label{sec:dis}

With $\Omega$ a Lipschitz polyhedron, let 
 $\cT^h$ be  a family of
conforming tetrahedral meshes 
that are shape-regular as the  
mesh diameter  $h \rightarrow 0$.  
A typical element of $\cT^h$ is  denoted $\tau \in \cT^h$ (a closed subset of
$\overline{\Omega}$). 
We define our approximation space $\Ned_h\subset \HocO$ as the N\'ed\'elec curl-conforming finite element space, of some fixed order $m$, on the mesh $\mathcal{T}^h$ with functions whose tangential trace is zero on $\po$; see \cite{Ne:80}, \cite[Chapter 5]{Mo:03}.
N\'ed\'elec curl-conforming finite elements are often termed {edge elements} because at the lowest order basis functions and degrees of freedom are associated with the edges of the mesh.
At higher order the geometrical identification of basis functions and degrees of freedom is more complicated.
Here, as degrees of freedom, we adopt suitable integrals on edges, faces and volumes of the mesh (see \cite[Definition 6, Propositions 3 and 4]{BoRa:17}).
Let $\cI^h$ be a set of indices for the degrees of freedom defined on the tetrahedra of $\mathcal{T}^h$.
Note that in order to define the coefficients of the interpolation operator onto $\Ned_h$ as the degrees of freedom applied to the function to be interpolated, the following duality property between basis functions and degrees of freedom needs to be satisfied:
\[
\psi_i (\bw_j) = \delta_{ij}, \quad i,j \in \cI^h,
\]  
where $\psi_i$ is the $i$-th degree of freedom and $\bw_j$ is the $j$-th basis function. Since this property is not automatically granted for high-order edge elements, we use the technique introduced in \cite{BoRa:17} to restore it.

The Galerkin method applied to the variational problem \eqref{eq:vf_intro} is
\beqs
\text{find}\,\, \bE_h \in \Ned_h\,\, \text{ such that } \,\,a_\abs(\bE_h,\bv_h) = F(\bv_h) \,\,\text{ for all }\,  \bv_h \in \Ned_h.
\eeqs
The Galerkin matrix $A_\abs$ is defined by
\beqs
(A_\abs)_{i j} := a_\abs(\bw_{j}, \bw_{i}), \quad i,j \in \cI^h,
\eeqs
and finding the Galerkin solution is then equivalent to solving the linear system 
$A_\abs \mathbf{u} =\mathbf{f}$,
where $f_i := F(\bw_{i})$, $i \in \cI^h$. 

\subsection{Properties of the sesquilinear form $a_\abs(\cdot,\cdot)$}\label{sec:prop}

In this section we briefly provide the key theoretical properties of the
sesquilinear form defined in \eqref{eq:Helmholtzvf_intro}. This
form depends on both  parameters $\abs$ and $\wn$, but only the first of these is reflected in the notation.  We will assume throughout that
\begin{equation} \label{eq:alpha}
\vert \abs\vert  \lesssim  \wn^2.
\end{equation} 
Throughout the paper we use the notation $A\lesssim B$ (equivalently $B \gtrsim A$) when $A/B$ is bounded above by a constant independent of $\wn$, $\abs$, and mesh diameters $h, \Hsub, \Hcs$ (the latter two introduced below).  We write $A \sim B$ when  $A\lesssim B$ and  $B \lesssim  A$. 

\ble\textbf{\emph{(Continuity of $a_\abs(\cdot,\cdot)$ and $a^*_\abs(\cdot,\cdot)$)}}\label{lem:cont}
\beq
\max\Big\{ |a_{\abs}(\bv,\bw)|, |a^*_{\abs}(\bv,\bw)|\Big\} \lesssim \Ccont \N{\bv}_{\weight} \N{\bw}_{\weight}
\label{eq:cont} 
\eeq
for all $\wn>0$ and $\bv, \bw \in \HocO$. 
\end{lemma}

\bpf
This follows from the Cauchy-Schwarz inequality.
\epf

Before stating the next result, about the coercivity of $a_\abs(\cdot,\cdot)$, we need to define $\sqrt{\wn^2 + \ri \abs}$.
When doing this,  we need to consider both positive and negative $\abs$ since, whichever choice we make for the problem \eqref{eq:BVP}, the other forms the adjoint problem, and we need estimates on the solutions and sesquilinear forms for both problems (in particular, this is essential for analysing both left and right preconditioning - see Theorem \ref{thm:final3} and Theorem \ref{cor:final4} below).

\begin{definition}\label{def:z}
$z(\wn,\abs):= \sqrt{\wn^2 + \ri \abs}$ where the square root is defined with the branch cut on the positive real axis. 
\end{definition}

\noi Note that this definition implies that, when $\abs\neq 0$,
\beq
\label{eq:propsqrt}
\Im(z) > 0, \quad \mathrm{sign}(\abs) \Re(z)>  0, \quad \text{ and }\quad z(\wn,-\abs) = - \overline{z(\wn,\abs)}.
\end{equation}

\begin{lemma}{\textbf{\emph{(\cite[Proposition 2.3]{GrSpVa:17})}}}\label{prop:elem}
With $z(\wn,\abs)$ defined above, for all $\wn>0$,
\begin{equation}\label{eq:ests4lemma}
\vert z \vert \sim \wn \quad \text{and} \quad  \frac{\Im(z)}{\vert z\vert} \sim \frac{\vert \abs\vert }{\wn^2}.
\end{equation}
\end{lemma}

\begin{lemma}\textbf{\emph{(Coercivity of $a_\abs(\cdot,\cdot)$ and $a^*_\abs(\cdot,\cdot)$)}}\label{lem:coer}
There exists
a constant $\rho>0$ independent of $\wn$ and $\abs$ such that 
\beq\label{eq:coercivity}
|a_{\abs}(\bv,\bv)|=|a^*_{\abs}(\bv,\bv)|
 \geq \ \Im\left( \Theta a_\abs(\bv,\bv) \right)\ \geq \
\rho\,\frac{\vert \abs\vert }{\wn^2}\,  \N{\bv}_{\weight}^2
\eeq
for all $\wn>0$ and $v\in \HocO$,  where $\Theta = -\overline{z}/\vert z \vert$.\end{lemma}

\bre\textbf{\emph{(Discussion of coercivity of $a_\abs(\cdot,\cdot)$)}}
We make two remarks.

\noi

\ben
\item Once one has established that 
\beq\label{eq:coer1a}
| a_{\abs}(\bv,\bv)|\geq \ \rho \, \frac{\vert \abs\vert }{\wn^2}\,  \N{\bv}_{\weight}^2
\eeq
for all $v\in \HocO$, then the existence of a $\Theta(\wn,\abs)$ with $|\Theta(\wn,\abs)|=1$ such that 
\beq\label{eq:coer2}
\Im\left( \Theta(\wn,\abs) a_\abs(\bv,\bv) \right)\ \geq \ \rho \, \frac{\vert \abs\vert }{\wn^2}\,  \N{\bv}_{\weight}^2
\eeq
for all $v\in \HocO$ follows from the convexity of the numerical range (see, e.g., \cite[Theorem 1.1-2]{GuRa:97}). Nevertheless, our proof of Lemma \ref{lem:coer} establishes \eqref{eq:coer2} directly and obtains \eqref{eq:coer1a} as a consequence.

\item The coercivity of $a_{\abs}(\cdot,\cdot)$ is discussed in the literature (e.g.~in \cite[Page 244]{Mo:92}, \cite[Page 1]{GoPa:03}) and an explicit statement 
for the heterogeneous interior impedance problem
appeared recently in \cite[Theorem 2.2]{KiUrZe:16}. The analogous result for the homogeneous Helmholtz Dirichlet problem appears in abstract form in \cite[\S2.4]{Le:85} and the result 
for the homogeneous Helmholtz interior impedance problem appears in \cite[Lemma 2.4]{GrSpVa:17} in a form similar to Lemma \ref{lem:coer}.
\een
\ere

\begin{proof}[Proof of Lemma \ref{lem:coer}]
With $z$ defined by Definition \eqref{def:z}, writing $z = p+iq$ and using 
the definition of $a_\abs(\cdot,\cdot)$,  we have 
\beqs
 a_{\abs}(\bv,\bv) =\big\|\curl \bv \big\|_{{\bf L}^2(\Omega)}^2 - (p+\ri q )^2 \big\|\bv \big\|_
{{\bf L}^2(\Omega)}^2 .
\eeqs
Therefore
\begin{equation*}
\Im \left[ -(p-\ri q)  a_{\abs}(\bv,\bv)\right] = q \big\|\curl \bv\big\|_{{\bf L}^2(\Omega)}^2 +  
q (p^2+ q^2) \big\|\bv \big\|_
{{\bf L}^2(\Omega)}^2.
\end{equation*}
Hence, dividing through by $\vert z \vert = \sqrt{p^2 + q^2} $ and setting 
$\Theta = - \overline{z} /\vert z \vert$,  we have 
\begin{equation*}
\Im \left[ \Theta   a_{\abs}(\bv,\bv)\right] \  =\  \frac{\Im(z)}{\vert z\vert} \left[ \big\|\curl\bv \big\|_{{\bf L}^2(\Omega)}^2 +  
 \vert z \vert^2  \big\|\bv \big\|_
{{\bf L}^2(\Omega)}^2\right].
\end{equation*}
The inequality \eqref{eq:coercivity} then follows from the two estimates in \eqref{eq:ests4lemma}. 

The result about the adjoint form $a^*_\abs(\cdot,\cdot)$ follows immediately after noticing that the third equation in \eqref{eq:propsqrt} implies that $\Im z(k,-\abs)= \Im z(\wn,\abs)$.
\end{proof} 

\begin{corollary}\textbf{\emph{(Bound on the solutions of \eqref{eq:BVP} and \eqref{eq:BVPadj} via Lax--Milgram)}}\label{cor:LM}
The solution of the variational problem \eqref{eq:vf_intro} exists, is unique, and satisfies the bound.
\beq\label{eq:LM}
\N{\bE}_{\weight} \lesssim  
\left(
\frac{\wn}{|\abs|}\right)
\big\|\bF\big\|_{{\bf L}^2(\Omega)}
\eeq
for all $\wn>0$.
The same is true if the sesquilinear form $a_\abs(\cdot,\cdot)$ in \eqref{eq:vf_intro} is replaced by $a^*_\abs(\cdot,\cdot)$ given by \eqref{eq:adjoint_form}.
\end{corollary}

\bpf
The Lax--Milgram theorem, the continuity result of Lemma \ref{lem:cont}, and the coercivity result of Lemma \ref{lem:coer} imply that the solution of the variational problem \eqref{eq:vf_intro} satisfies
\beqs
\N{\bE}_{\weight} \lesssim 
\left(\frac{\wn^2}{|\abs|}\right)\N{F}_{(\curlt,\wn)'},
\eeqs
where $\|\cdot\|_{(\curlt,\wn)'}$ denotes the norm on the dual-space of $\HocO$ defined by 
\beqs
\N{F}_{(\curlt,\wn)'} := \sup_{\bv\in \HocO\setminus\{\bze\}}\frac{|F(\bv)|}{\N{\bv}_{\weight}}.
\eeqs
From the definition of $F(\cdot)$ in \eqref{eq:Helmholtzvf_intro},
\beqs
|F(\bv)| \leq \big\|\bF\big\|_{{\bf L}^2(\Omega)} \big\|\bv\big\|_{{\bf L}^2(\Omega)}\leq \frac{1}{\wn}\big\|\bF\big\|_{{\bf L}^2(\Omega)}\N{\bv}_{\weight}
\eeqs
and the inequality \eqref{eq:LM} follows. The result about the solution to the adjoint problem follows in a similar way.
\epf
\subsection{Regularity of the BVP and its adjoint}

In order to estimate the approximation properties of the coarse grid operator in Lemma \ref{lem:approx} below, we use a duality argument and need ${\bf H}^1$-regularity of both $\bE$ and $\curl \bE$.
We first formulate this regularity as an assumption (Assumption \ref{ass:reg}) and then show that this assumption is satisfied when $\Omega$ is \emph{either} $C^{1,1}$ \emph{or} a convex polyedron
(Lemma \ref{lem:reg1}). 

\begin{assumption}\textbf{\emph{($\wn$- and $\abs$-explicit ${\bf H}^1$ regularity)}}\label{ass:reg}
$\Omega$ is such that, given $\bF\in {\bf L}^2(\Omega)\cap \Hdivo$,
$\abs\in\mathbb{R}\setminus\{0\}$, and $\wn>0$, the solution $\bE$ of the BVP \eqref{eq:BVP} has  
$\curl \bE\in {\bf H}^1(\Omega)$ and $\bE \in {\bf H}^1(\Omega)$.
Moreover, if $\abs$ satisfies \eqref{eq:alpha} then, given $\wn_0>0$, 
\beq\label{eq:reg_result}
\N{\curl \bE}_{{\bf H}^1(\Omega)} + \wn\N{\bE}_{{\bf H}^1(\Omega)} \lesssim \Crego \wn\bigg( \big\|\curl \bE\big\|_{{\bf L}^2(\Omega)} + \wn\big\|\bE\big\|_{{\bf L}^2(\Omega)}\bigg) + \Cregt\N{\bF}_{{\bf L}^2(\Omega)},
\eeq
 for all $\wn\geq \wn_0$.
\end{assumption}

\noi Note that:
(i) since $\abs\in\mathbb{R}\setminus\{0\}$, Assumption \ref{ass:reg} concerns the solution of both the BVP \eqref{eq:BVP} and the adjoint BVP \eqref{eq:BVPadj}, and
(ii) the bound \eqref{eq:reg_result} can be viewed as a rigorous expression of the idea that taking a derivative of a solution of the PDE \eqref{eq:1}  incurs a power of $\wn$.

If Assumption \ref{ass:reg} holds, the bound \eqref{eq:LM} from the Lax--Milgram theorem immediately implies the following corollary.

\begin{corollary}
If Assumption \ref{ass:reg} holds, and $\bE$ is the solution to \emph{either} the BVP \eqref{eq:BVP} \emph{or} the adjoint BVP \eqref{eq:BVPadj} with 
$\bF \in {\bf L}^2(\Omega)\cap \Hdivo$, 
$\abs\in\mathbb{R}\setminus\{0\}$ satisfying \eqref{eq:alpha}, and $\wn>0$, then, given $\wn_0>0$,
\beq\label{eq:reg_result2}
\N{\curl\bE}_{{\bf H}^1(\Omega)} + \wn\N{\bE}_{{\bf H}^1(\Omega)} \lesssim \Creg \left(
\frac{\wn^2}{|\abs|}\right)
 \N{\bF}_{{\bf L}^2(\Omega)}, 
\eeq
for all $\wn\geq \wn_0$.
\end{corollary}

We now describe two situations in which Assumption \ref{ass:reg} holds. 

\ble
\label{lem:reg1}
If $\Omega$ is \emph{either} a bounded $C^{1,1}$ domain \emph{or} a convex polyhedron, 
then Assumption \ref{ass:reg} holds. 
\ele

Since the proof of Lemma \ref{lem:reg1} is quite long and involved, we relegate it to Appendix \ref{app:2}.

\bre\textbf{\emph{(Relation of Lemma \ref{lem:reg1} to other results in the literature)}}
The analogue of Lemma \ref{lem:reg1} for the interior impedance problem with $\abs=0$  
is contained in 
\cite[Lemma 5.2.2 and Remark 5.5.8]{Mo:11},
with an analogous result for a bounded Lipschitz polyhedron (with less regularity of $\bE$) contained in \cite[Theorem 5.5.5]{Mo:11}. 
A similar result to Lemma \ref{lem:reg1} for the PEC case with $\abs=0$ is contained in \cite[Theorem 4.5.3. Page 162]{CoDaNi:10}; this result proves that if
 $\po$ is $C^2$, $\bF \in {\bf L}^2(\Omega)$, and $\dive \bF \in H^1(\Omega)$, then $\curl \bE\in {\bf H}^1(\Omega)$ and $\bE \in {\bf H}^2(\Omega)$. 
\es{Observe that in this last result, the condition $\dive\bF =0$ is replaced by a regularity condition on $\dive \bF$. We see a similar feature in our proof of Lemma \ref{lem:reg1}; indeed, the proof shows that Lemma \ref{lem:reg1} actually holds with $\dive\bF =0$ replaced by $\dive\bF \in {\bf L}^2(\Omega)$ (and then with 
a $\|\dive\bF\|_{{\bf L}^2(\Omega)}$ term on the right-hand side of \eqref{eq:reg_result} -- see \eqref{eq:reg_result_app}).
}

\ere

\section{Domain decomposition} 

\label{sec:DD}

\subsection{Definition and properties of the subdomains}\label{sec:DD1}

To define appropriate subspaces of the edge finite-element space
$\Ned_h\subset \HocO$,  
we start with a collection of open subsets    $\{ \tOmega_\ell: \ell =
1, \ldots, N\}$ of $\R^d$  that form an overlapping cover of $\overline{\Omega}$, 
and we set $\Omega_\ell =
\tOmega_\ell \cap \overline{\Omega}$. 
Each $\overline{\Omega}_\ell$  is assumed to be non-empty 
and to  consist of a union of
elements of the mesh $\cT_h$. Then, for each $\ell = 1, \ldots, N$,  we set 
\beq\label{eq:Vell}
\Ned^\ell_h := \Ned_h \cap {\bf H}_0(\curlt; \Omega_\ell);
\eeq
i.e.~the tangential 
traces of elements of $\Ned^\ell_h$ vanish on the
internal boundary $\partial \Omega_\ell \backslash \po$ (as well as on $\partial \Omega_\ell \cap \po$).
In writing the definition of $\Ned^\ell_h$, we are using the fact that ${\bf H}_0(\curlt;\Omega_\ell)$ can be considered as a subset of $\HocO$ by extending functions in ${\bf H}_0(\curlt;\Omega_\ell)$ by zero (such extensions are in $\HocO$ by, e.g., \cite[Lemma 5.3]{Mo:03}).

Let $\cI^h(\Omega_{\ell}) \subset \cI^h$ be the set of indices of the degrees of freedom whose support (edge, face or volume) is contained in $\Omega_{\ell}$. 
We then have that $\cI^h= \bigcup_{\ell=1}^N \cI^h(\Omega_\ell)$. 
For $i\in \cI^h({\Omega_\ell})$ and $j \in \cI^h$, we define the restriction matrices  
\beq\label{eq:restriction1}
(R^\ell)_{ij} := \delta_{ij}.
\eeq
We make the following assumptions on the subdomains:
\ben
\item \emph{Shape regularity:} the subdomains are shape-regular
Lipschitz polyhedra of diameter $H_\ell$, in the sense that the volume is of order $ H_\ell^3$ and surface
  area of order 
  $ H_\ell^{2}$ (with omitted constants independent of all parameters). We then let $\Hsub:= \max_\ell H_\ell$.
\item \emph{Uniform overlap of order $\delta$:} For each $\ell = 1, \ldots , N$, let $\mathring{\Omega}_\ell$ denote the
part of $\Omega_\ell$ that is not overlapped by any other subdomains,
and for $\mu>0$ let $\Omega_{\ell, \mu}$ denote the set of points in
$\Omega_\ell$ that are a distance no more than $\mu$ from the
boundary $\partial \Omega_\ell$. Then we assume that for some
$\delta>0$ and some $0<c<1$ fixed, 
\beqs
\Omega_{\ell, c \delta}\subset \Omega_\ell\backslash
  \mathring{\Omega}_\ell \subset \Omega_{\ell, 
\delta};
\eeqs
the case $\delta \sim  H_\ell$ is called \emph{generous overlap}.
\item {\em Finite
  overlap assumption:} as $h, \Hsub\tendo$,
\begin{equation} 
\#
  \Lambda(\ell) \lesssim 1, \quad \text{where} \quad   
\Lambda(\ell) = \big\{ \ell' :  \Omega_\ell \cap \Omega_{\ell'} \not =
\emptyset \big\}.   \label{eq:finoverlap}\end{equation}       
\een

\subsection{Definition of the coarse space}\label{sec:DD2}
Let  $\{\cT^{H}\}$ be  a sequence of shape-regular, tetrahedral meshes on 
$\overline{\Omega}$, with mesh diameter $\Hcs$. We assume that each element of $\cT^H$ consists of the union of a set of fine grid elements.
Let $\cI^H$ be a set of indices for the degrees of freedom defined on the tetrahedra of the coarse mesh $\mathcal{T}^H$.
The coarse basis functions $\{{\bf w}^{H}_p\}_{p \in \cI^H}$ are taken to be the curl-conforming basis functions on $\cT^H$ with zero tangential traces on $\po$;
importantly, we allow the coarse space basis functions to be a different order than the fine-grid basis functions in $\Ned_h$.
We define the coarse finite element space
$\Ned_H :=  \mathrm{span} \{{\bf w}^H_{p}  : p \in \cI^H
\} \subset \HocO$,
and we define the ``restriction'' matrix
\begin{equation}\label{eq:restriction}
(R^0)_{pj} := \psi^h_j(\bw^H_{p}),  \quad p \in \cI^{H}, \quad j \in \cI^h,
\end{equation}
where $\psi^h_j$ are the degrees of freedom on the fine mesh.
Note that $(R^0)^T$ is the interpolation matrix from $\Ned_H$ onto $\Ned_h$: its entries $((R^0)^T)_{jp} = (R^0)_{pj} = \psi^h_j(\bw^H_{p})$ are integrals (on edges, faces or volumes) which should be computed with sufficiently accurate quadrature formulas; we return to this in \S\ref{sec:num}.

\subsection{Two-level Additive Schwarz preconditioners}\label{sec:3.3}
With the restriction matrices $(R^\ell)_{\ell=0}^{N}$ defined by \eqref{eq:restriction1} and \eqref{eq:restriction} above, 
we define
 \beq\label{eq:minor}
 A_{\abs}^\ell \ := \ R^\ell  A_\abs (R^\ell)^T.
 \eeq
 For $\ell=1,\ldots, N$, the matrix $A_{\abs}^\ell$ 
is then  just the minor  of 
$A_\abs$ corresponding to rows
and columns taken from $\cI^h({\Omega_\ell})$ and the matrix $A_{\abs}^0$ is the Galerkin matrix for the variational problem \eqref{eq:vf_intro} 
discretised in $\Ned_H$. 
The matrix $A_{\abs}^\ell$ is therefore the Galerkin matrix corresponding to solving the Maxwell problem \eqref{eq:vf_intro}  in the domain $\Omega_\ell$ with PEC boundary conditions on $\partial \Omega_\ell$.
The coercivity 
result Lemma \ref{lem:coer} implies that the matrices $\matrixA_{\abs}^\ell$, $\ell=0,\ldots, N$, are invertible for all mesh 
sizes $h$ and all choices of  $\abs \not = 0$.
Indeed,  if  
$\matrixAepszero \bV
= \mathbf{0}$, where $\bV$ is a vector such that $\bv_H:= \sum_{p \in \cI^H}V_p \bw_{p}^H \in \Ned_H$,
then
$0 \ = \ \bV^*\matrixAepszero \bV \ = \ a_\abs(\bv_H, \bv_H) $.
Therefore,
$$0 \ = \ \vert a_\abs(\bv_H, \bv_H) \vert  \ \geq \ \rho \,\frac{\vert \abs\vert }{\wn^2}\, \Vert \bv_H\Vert_{\weight}^2, $$
which immediately implies $\bv_H = \bfzero$, and thus $\bV =
\bfzero$. Similar arguments apply to $\matrixAepsell$ and to the
adjoints $(A_{\abs}^\ell)^*$, $\ell = 0, \ldots, N$.  
The theory in this paper considers the classical
two-level \emph{Additive Schwarz} preconditioner for $A_\abs$ defined by
\beq\label{eq:defAS} 
\matrixB_{\abs, \text{AS}}^{-1}: =  \sum_{\ell=0}^N (R^\ell)^T (A_{\abs}^\ell)^{-1} R^\ell.
\eeq

\subsection{Discrete Helmholtz decomposition of $\Ned_h$ and associated results}\label{sec:orthog}
Recall that $\Hdiv$ is defined by
\beqs
\Hdiv:= \big\{ \bv \in {\bf L}^2(\Omega) : \dive \bv \in L^2(\Omega)\big\}
\eeqs
and $\Hdivoo$ by 
\beqs
\Hdivoo:= \big\{ \bv \in {\bf L}^2(\Omega) : \dive \bv \in L^2(\Omega) \, \tand \, \bv\cdot \bn =0 \ton \partial \Omega\big\}
\eeqs
(recall that the normal trace $\bv\cdot\bn$ is well-defined on $\Hdiv$ by, e.g., \cite[Thereom 3.24]{Mo:03}).

Let $\RT_h$ denote the Raviart-Thomas finite element subspaces of $\Hdivoo$
of index $m$ based on the fine mesh $\cT^h$.
Let $W_h$ denote the subspace of $H_0^1(\Omega)$ consisting of continuous piecewise polynomials of degree $m+1$, also on the fine mesh $\cT^h$.
We then have the discrete Helmholtz decomposition 
\beq\label{eq:orthog}
\Ned_h = \curl_h \RT_h \oplus \grad W_h,
\eeq
see, e.g., \cite[\S2, in particular Remark 2.1]{ArFaWi:00}
\cite[\S7.2.1, in particular Lemma 7.4]{Mo:03} \cite[Equation 2.3]{GoPa:03}, where $\curl_h$ is the ${\bf L}^2$-adjoint of the map $\curl : \Ned_h \rightarrow \RT_h$, i.e.~
\beqs
\big( \curl_h \bv_h, \bq_h\big)_{{\bf L}^2(\Omega)} = \big(  \bv_h, \curl \bq_h\big)_{{\bf L}^2(\Omega)} \quad\tfor \bq_h \in \bQ_h \tand \bv_h \in \bV_h,
\eeqs
and the decomposition \eqref{eq:orthog} is orthogonal both in $\bLtO$ and in $\HcO$.

We define $\RT_H$ and $W_H$ in the same way as $\RT_h$ and $W_h$, but using the coarse mesh $\{\cT^{H}\}$. We also set
\beq\label{eq:Whell}
\RT^\ell_h := \RT_h \cap {\bf H}_0(\dive; \Omega^\ell) \quad\tand\quad W^\ell_h := W_h \cap H_0^1(\Omega^\ell)\quad\tfor \ell=1,\ldots,N,
\eeq
where fields on $\Omega^\ell$ are identified as fields on $\Omega$ via extending them by zero.
We then have the analogue of the decomposition \eqref{eq:orthog}:
\beq\label{eq:orthog2}
\Ned_h^\ell = \curl_h^\ell \RT_h^\ell \oplus \grad W_h^\ell,\quad\ell=1,\ldots,N,
\eeq
where $\curl_h^\ell$ is the ${\bf L}^2$-adjoint of the map $\curlt: \Ned_h^\ell \rightarrow \RT_h^\ell$.

For fields in $\curlt_h^\ell \RT_h^\ell$ we have the following Poincar\'e--Friedrichs inequality from \cite[Lemma 4.1]{GoPa:03}.

\ble\emph{\textbf{(Poincar\'e--Friedrichs type-inequality \cite[Lemma 4.1]{GoPa:03})}}\label{lem:PF} If $\Omega_\ell$ is polyhedral,
\beq\label{eq:PF1}
\N{\bq}_{{\bf L}^2(\Omega_\ell)} \lesssim \Hsub \N{\curl \bq}_{{\bf L}^2(\Omega_\ell)}\quad \tfa\, \bq \in \curl_h^{\ell} \RT^\ell_h,
\eeq
where the omitted constant is independent of $h$ and $\Hsub$.
\ele

\noi Note:
\ben
\item In \cite{GoPa:03} the subdomains are related to the coarse grid elements, but the result \eqref{eq:PF1} is independent of the coarse mesh and thus holds in our more general setting where the subdomains can be unrelated to the coarse grid.
\item The result \eqref{eq:PF1} is proved in \cite{GoPa:03}  for $\Omega_\ell$ convex, using the Poincar\'e--Friedrichs type-inequality for convex domains in \cite[Chapter III, Prop.~5.1]{GiRa:86}. This inequality, however, holds for polyhedral domains by \cite[Prop.~4.6]{AmBeDaGi:98}, and thus \cite[Lemma 4.1]{GoPa:03} holds in this more-general situation (see \cite[Remark 2.2]{GoPa:03}).
\een

Finally, we recall the following result from, e.g., \cite[Theorem 52]{Sc:09}, \cite[Equation 4.10]{GoPa:03} (with a similar result in \cite[Lemma 5.2]{ArFaWi:00}).

\ble\label{lem:magic}
\igg{If $\Omega$ is a convex polyhedron}  
then,
given $\bq_h \in \curl_h \RT_h$, there exists a unique field in ${\bf H}_0(\curl; \Omega)$, which we denote by $\bS \bq_h$, such that
$\curl (\bS \bq_h) = \curl \bq_h$ and 
$\dive \bS \bq_h=0.$ 
Furthermore,
\beq\label{eq:magic}
\N{\bq_h - \bS \bq_h}_{{\bf L}^2(\Omega)} \lesssim h \N{\curl \bq_h}_{{\bf L}^2(\Omega)}.
\eeq
\ele

The key point from Lemma \ref{lem:magic} is that although $\bq_h$ is not divergence-free, $\bS \bq_h$ provides an approximation to $\bq_h$ that is divergence-free and has the same curl. The assumption on the geometry of $\Omega$ comes from the fact that we need $H^2$ regularity of solutions of Laplace's equation (see the references in Appendix \ref{app:2} where we also use this).

\section{Theory  of Additive Schwarz methods}
\label{sec:Convergence}
The following theory establishes rigorously that Property (ii) of \S\ref{sec:analyse} holds 
for the preconditioner \eqref{eq:defAS} applied to $A_\abs$ if $\vert \abs\vert $ is   sufficiently  large, $\Hsub$,
$\Hcs$ are sufficiently small and the overlap is generous.

The theory is split into four sections; Sections \ref{sec:41}, \ref{sec:42}, and \ref{sec:44} are very similar to the Helmholtz theory in \cite{GrSpVa:17}.
\S \ref{sec:43} is very different to the Helmholtz theory in that we need to use and adapt the arguments of \cite{GoPa:03}; see the discussion at the beginning of \S \ref{sec:43}.

\subsection{Stable splitting and associated results}\label{sec:41}

\begin{lemma}\textbf{\emph{(Stable splitting in the $\curlt, \wn$-norm)}}\label{lem:stable_splitting}
For all $\bv_h \in \Ned_h$, there exist $\bv^{\ell} \in \Ned^\ell$ 
for each $\ell = 0, \ldots, N$ such that 
\begin{equation}
\label{eq:stable_splitting}
\bv_h = \sum_{\ell=0}^N \bv^{\ell} \quad \text{and } \quad 
\sum_{\ell=0}^N \N{\bv^{\ell}}_{\weight}^2 \ \lesssim\ \Cstab\bigg( 1 + \left(\frac{\Hcs}{\delta}\right)^2 \bigg) \Vert \bv_h \Vert_{\weight}^2.
\end{equation}  
\end{lemma}
\begin{proof}
When $\delta \sim \Hcs$, \eqref{eq:stable_splitting} is proved in \cite[Lemma 4.1]{PaZh:02}, noting that their stability constant is independent of their constant $\alpha$ (which corresponds to our $\wn^2$). The proof of \cite[Lemma 4.1]{PaZh:02} can easily be repeated in the case that $\delta$ is independent of $\Hcs$: $\Hcs^{-1}$ on the right-hand side of \cite[Equation 4.1]{PaZh:02} should be replaced by $\delta^{-1}$, and this change propagated through the proof; the result is \eqref{eq:stable_splitting}.

The result \eqref{eq:stable_splitting} with $\delta$ independent of $\Hcs$, but with the additional assumption that $\Omega$ is convex, is proved in \cite[Theorem 4.4]{To:00} (noting again that the stability constant in that theorem is independent of the parameters $\eta_1$ and $\eta_2$ in that paper, which for us correspond to $\wn^2$ and $1$ respectively).
\end{proof}

Note that in the Helmholtz case, the factor $(\Hcs/\delta)^2$ in \eqref{eq:stable_splitting} can be replaced by $\Hcs/\delta$ -- see \cite[Lemma 4.1]{GrSpVa:17} -- but this feature has so far not been established for the Maxwell case. For this reason, the dependence of our GMRES bounds in \S\ref{sec:53} on $\Hcs/\delta$ is worse than that in the analogous Helmholtz results in \cite{GrSpVa:17}.

The next  lemma is a kind of converse to Lemma \ref{lem:stable_splitting}. 
Here the energy of a sum of components is estimated above by the sum of the energies. 

\begin{lemma}\label{lem:norm_of_sum}
For all choices of $\bv^{\ell} \in \Ned_{\ell}$ ,  
 $\ell = 0, \cdots, N$, we have  
\begin{equation}
\label{eq:norm_of_sum}
\quad \bigg\Vert \, \sum_{\ell=0}^N  \bv^{\ell} \, \bigg\Vert_{\weight}^2  \lesssim
\sum_{\ell=0}^N \N{\bv^{\ell}}_{\weight}^2\ .
\end{equation}  
\end{lemma}

\noi The proof of this result is essentially identical to the Helmholtz analogue \cite[Lemma 4.2]{GrSpVa:17}. We include it, nevertheless, since it is key to the rest of the paper.

\begin{proof}[Proof of Lemma \ref{lem:norm_of_sum}]
Let  $\sum_\ell$ denote  the sum from $\ell = 1$ to $N$ and  
and recall the notation   $\Lambda(\ell)$ introduced in
\eqref{eq:finoverlap}.  
Several applications of the Cauchy-Schwarz inequality imply that
\begin{align}
\bigg\Vert \sum_{\ell} \bv^\ell \bigg\Vert_{\weight}^2  \hspace{-1ex}= & 
\bigg(\sum_\ell \bv^\ell ,  \sum_{\ell'} \bv^{\ell'}\bigg)_{\weight}  
\hspace{-1ex}=  \sum_\ell \sum_{\ell' \in \Lambda(\ell)}  (\bv^{\ell} ,
\bv^{\ell'})_{\weight}   
 \leq \sum_\ell \N{\bv^\ell}_{\weight} \bigg(\sum_{\ell' \in \Lambda(\ell)}  \Vert \bv^{\ell'}\Vert_{\weight}\bigg)\nonumber \\
\leq  & \bigg(\sum_\ell \N{\bv^\ell}_{\weight}^2\bigg)^{1/2}  
\bigg(\sum_\ell \bigg(\sum_{\ell' \in \Lambda(\ell)}  
\big\|\bv^{\ell'}\big\|_{\weight}
\bigg)^2 \bigg)^{1/2} \nonumber \\
\  \leq \  & 
\bigg(\sum_\ell \N{\bv^\ell }_{\weight}^2\bigg)^{1/2}  
\bigg(\sum_\ell  \# \Lambda(\ell) \sum_{\ell' \in \Lambda(\ell)}  
\big\| \bv^{\ell'}\big\|_{\weight}^2 
 \bigg)^{1/2} 
\  \lesssim\     
\sum_\ell \N{ \bv^\ell }_{\weight}^2 \ , \label{eq:sum_from_1} 
\end{align}
where we have also used the finite overlap assumption \eqref{eq:finoverlap}. 
To obtain   \eqref{eq:norm_of_sum}, we
write    
\begin{align}
\bigg\Vert \sum_{\ell=0}^N \bv^\ell \bigg\Vert_{\weight}^2    =  & 
\bigg(\sum_{\ell=0}^N \bv^\ell ,  \sum_{\ell=0}^N \bv^\ell\bigg)_{\weight} \label{eq:rheq}
 =  \big\Vert \bv^0\big\Vert_{\weight}^2 + 2 \bigg(\bv^0, \sum_{\ell} \bv^\ell\bigg)_{\weight}
 +  \bigg(\sum_\ell \bv^\ell ,  \sum_\ell \bv^\ell\bigg)_{\weight}. 
\end{align}
Using  the  Cauchy-Schwarz and the  arithmetic-geometric mean
inequalities on the middle term,  we can   estimate \eqref{eq:rheq}
from   above by 
$$ \lesssim \ \big\Vert \bv^0\big\Vert_{\weight}^2 + \bigg\Vert \sum_\ell \bv^\ell
\bigg\Vert_{\weight}^2 ,  $$
and the result follows from \eqref{eq:sum_from_1}. 
 \end{proof}
 
\bre\textbf{\emph{(Impedance boundary conditions)}}\label{rem:imp}
The main obstacle to extending the theory in this paper to the BVP \eqref{eq:BVP} with the PEC boundary condition replaced by the impedance boundary condition \eqref{eq:imp} is that the Hilbert space for the variational formulation of the impedance problem is \emph{not} $\HcO$, but ${\bf H}_{\rm imp}(\curl;\Omega)$ (see, e.g., \cite[\S3.8]{Mo:03}). The first step towards extending the theory in the paper to the impedance BVP would be to establish a stable-splitting in the norm on ${\bf H}_{\rm imp}(\curl;\Omega)$.
\ere 
 
 \subsection{Definition of the projection operators and the path towards bounds on the field of values}\label{sec:42}

For each $\ell = 1, \ldots, N$,  we define  linear operators $\proj_{\abs}^{\ell} : 
\HocO\rightarrow \Ned^{\ell}$ as follows. 
Given $\bv \in \HocO$, $\proj_{\abs}^{\ell} \bv$ is defined to be the unique solution of 
the equation 
\beq 
\label{eq:defQi} 
\aeps (\proj_{\abs}^{\ell} \bv , \bw_h^\ell) \ =\  \aeps (\bv , \bw_h^\ell),  \quad \bw_h^\ell \in \Vell. 
\eeq 
Recall from the discussion underneath \eqref{eq:Vell} that $\Ned^\ell_h$ can be considered as a subspace of $\HocO$ by extension by zero, and we can therefore consider $\proj_{\abs}^{\ell}\bv$ as an element of $\HocO$ with support on $\Omega_\ell$. 
Analogously, given $\bv \in \HocO$, $\proj_{\abs}^{0} \bv$ is defined to be the unique solution of 
the equation 
\beq 
\label{eq:defQi0} 
\aeps (\proj_{\abs}^{0} \bv , \bw_H) \ =\  \aeps (\bv , \bw_H),  \quad \bw_H \in \bQ_H,
\eeq 
and thus $\proj_{\abs}^{0} : 
\HocO\rightarrow \Ned_H \subset \HocO.$
We then define 
$$
\proj_{\abs} = \sum_{\ell=0}^N \proj_{\abs}^{\ell};$$ 
we show in  Theorem
\ref{thm:repQ} below that the   matrix
representation of $\proj_{\abs}$ corresponds to the action of the preconditioner \eqref{eq:defAS} on the matrix    $A_\abs$.

\bre\textbf{\emph{(Notation for the projection operators $\proj_\abs^\ell$)}}
The Helmholtz theory in \cite{GrSpVa:17} used the letter $Q$ for the projection operators, following the notation in \cite{CaWi:92} 
Here we use the letter $T$ for the projection operators, following \cite{GoPa:03}, and also allowing us to use $Q$ for spaces of N\'ed\'elec elements
(as in, e.g., \cite{GoPa:03}, \cite{ArFaWi:00}, \cite{Sc:09}, \cite{Sc:01}).
\ere

\noi The goals of this section (\S\ref{sec:Convergence}) are to (i) bound $\|\proj_\abs\|_{\weight}$ from above, and (ii) bound the field of values 
of $\proj_\abs$ away from the origin, where the field of values is the set of complex numbers
\beq\label{eq:fov}
\frac{\left(\bv_h, \proj_\abs \bv_h\right)_{\weight}}{\|\bv_h\|_{\weight}^2} \quad\text{for } \bv_h \in \Ned_h\setminus \{0\};
\eeq
note that the field of values is computed with respect 
to the wavenumber-dependent $(\cdot, \cdot)_{\weight} $ inner product.
A bound from above on $\|\proj_\abs\|_{\weight}$ follows immediately from the following.

\begin{theorem}\textbf{\emph{(Upper bound on the norm of $\proj_{\abs}$)}}
\label{thm:boundQ}
\beqs
\nrme{\proj_{\abs} \bv_h} \ \lesssim \ 
\left(\ksqeps\right) \nrme{\bv_h}
\quad \text{for all}\quad  \bv_h \in \Ned_h.
\eeqs
\end{theorem}

\begin{proof}
By the definition of $\proj_{\abs}$ and Lemma \ref{lem:norm_of_sum}, we have 
\begin{equation}\nrme{\proj_{\abs}  \bv_h}^2 \ = \ \left\Vert\sum_{\ell=0}^N
    \proj_{\abs}^{\ell}  \bv_h\right\Vert_{\weight}^2 \ \lesssim\ 
  \sum_{\ell=0}^N \N{\proj_{\abs}^{\ell} \bv_h}_{\curlt,\wn}^2\ . \label{eq:641}\end{equation}
Furthermore, by applying Lemma \ref{lem:coer} and the definitions 
\eqref{eq:defQi}, \eqref{eq:defQi0} we have 
\begin{align*}
\sum_{\ell=0}^N \nrme{\Qepsell \bv_h}^2 \ \lesssim \ &
\left(\ksqeps\right) \sum_{\ell=0}^N \Im \big[\Theta a_\abs(\Qepsell
\bv_h,\Qepsell \bv_h)\big]
\ =  \ 
\left(\ksqeps\right)  \Im \left[\Theta \sum_{\ell=0}^N
  a_{\abs}(\bv_h,\Qepsell \bv_h)\right]\\
 \ = \ & 
\left(\ksqeps\right)  \Im \left[\Theta  a_{\abs}\left(\bv_h,\sum_{\ell=0}^N \Qepsell \bv_h\right)\right]
\ \leq  \ 
\left(\ksqeps\right)  \left\vert  a_{\abs}\left(\bv_h,\sum_{\ell=0}^N \Qepsell \bv_h\right)\right\vert\nonumber
\end{align*}
(recalling that  $|\Theta|=1$).
Then, using  Lemma \ref{lem:cont}, and then 
Lemma \ref{lem:norm_of_sum},  
we have 
\begin{align}
\sum_{\ell=0}^N \nrme{\Qepsell \bv_h}^2  \lesssim  
\left(\ksqeps\right)  \nrme{\bv_h} 
\left\Vert \sum_{\ell=0}^N \Qepsell \bv_h\right\Vert_{\weight}
 \lesssim 
\left(\ksqeps\right)  \nrme{\bv_h} \bigg(\sum_{\ell=0}^N \nrme{\Qepsell \bv_h}^2\bigg)^{1/2}, \label{eq:642}
 \end{align}
and the result follows on combining \eqref{eq:642} with \eqref{eq:641}.
\end{proof}

The next two results are two of the three ingredients we use to bound the field of values away from the origin (the third ingredient is provided in \S\ref{sec:43}, and the bound is proved in \S\ref{sec:44}).

\begin{lemma}\label{lem:lowrestQi}
\beq\label{eq:45bound}
  \sum_{\ell=0}^N \Vert \proj_{\abs}^{\ell} \bv_h \Vert_{\weight}^2 
   \gtrsim 
\bigg( 1 + \left(\frac{\Hcs}{\delta}\right)^2 \bigg)^{-1}
     \left( \frac{\vert \abs \vert }{\wn^2} \right)^2 \Vert \bv_h \Vert_{\weight}^2  \quad \tfa \bv_h \in \Ned_h.
\eeq
\end{lemma}

\begin{proof}
Using Lemma \ref{lem:coer}, the decomposition of $\bv_h$ given in Lemma
\ref{lem:stable_splitting},
the definition of $\proj_{\abs}^{\ell}$, and
Lemma \ref{lem:cont},
we obtain
\begin{align*}
\left(\frac{\vert \abs \vert }{\wn^2}\right) \Vert \bv_h \Vert_{\weight}^2 \ \lesssim& \ \Im \big[\Theta a_\abs (\bv_h, \bv_h)\big] \ = \ \sum_{l=0}^N \Im \big[\Theta a_\abs (\bv_h, \bv^\ell)\big] \\
 =  & \ \sum_{l=0}^N \Im \big[\Theta a_\abs (\Qepsell \bv_h, \bv^\ell)\big] \ \lesssim\  \sum_{l=0}^N \Vert \Qepsell \bv_h \Vert_{\weight} \Vert \bv^\ell\Vert_{\weight} \ . 
\end{align*}
Using the Cauchy-Schwarz inequality and Lemma \ref{lem:stable_splitting}, we find
\begin{align*}
\left(\frac{\vert \abs\vert }{\wn^2} \right)\Vert \bv_h \Vert_{\weight}^2 &\lesssim 
\bigg( \sum_{\ell=0}^N \Vert \proj_{\abs}^{\ell} \bv_h \Vert_{\weight}^2\bigg)^{1/2}  \bigg( \sum_{\ell=0}^N 
\Vert \bv^\ell\Vert_{\weight}^2 \bigg)^{1/2} \\
& \lesssim     \bigg(
  \sum_{\ell=0}^N \Vert \proj_{\abs}^{\ell} \bv_h \Vert_{\weight}^2\bigg)^{1/2}  
\bigg( 1 + \left(\frac{\Hcs}{\delta}\right)^2 \bigg)^{1/2} 
\Vert \bv_h\Vert_{\weight} ,
\end{align*}
and the result follows.
\end{proof}

The bound \eqref{eq:45bound} 
shows that to bound the field of values \eqref{eq:fov} away from the origin it is sufficient to bound 
$(\bv_h, \proj_\abs \bv_h )_{\weight}$ from below by 
$\sum_{\ell = 0}^N
\Vert \proj_{\abs}^{\ell} \bv_h \Vert_{\weight}^2$.
The next lemma expresses $(\bv_h, \proj_\abs \bv_h )_{\weight}$ in terms of 
$\sum_{\ell = 0}^N
\Vert \proj_{\abs}^{\ell} \bv_h \Vert_{\weight}^2$ plus a sum of ``remainder" terms $R_{\abs}^\ell (\bv_h)$.

\begin{lemma}\label{lem:combined}  For $\ell = 0, \ldots, N$, set
\igg{\beqs
R_{\abs}^\ell (\bv_h)  : =  \big((\bI - \proj_{\abs}^{\ell})\bv_h, \proj_{\abs}^{\ell} \bv_h \big)_{\weight}.
\eeqs}
Then
\begin{align}
 (\bv_h, \proj_\abs \bv_h )_{\weight} 
=  \sum_{\ell = 0}^N
\Vert \proj_{\abs}^{\ell} \bv_h \Vert_{\weight}^2 + \sum_{\ell = 0}^NR_{\abs}^\ell (\bv_h). \label{eq:E4ii}
\end{align}
and 
\beq\label{eq:boundE}
R_{\abs}^\ell(\bv_h) =(2\wn^2+ \ri \abs)\big(  (\bI-\proj_{\abs}^{\ell})\bv_h, \proj_{\abs}^{\ell}\bv_h \big)_{{\bf L}^2(\Omega_\ell)}.
\eeq
\end{lemma} 

\begin{proof}
By the definition of $\proj_{\abs}$,
\beqs
(\bv_h, \proj_\abs \bv_h )_{\weight} = \sum_{\ell = 0}^N 
\left(\bv_h, \proj_{\abs}^{\ell} \bv_h \right)_{\weight}= \sum_{\ell = 0}^N
\left\{ \Vert \proj_{\abs}^{\ell} \bv_h \Vert_{\weight}^2 + \left((\bI - \proj_{\abs}^{\ell})\bv_h, \proj_{\abs}^{\ell} \bv_h \right)_{\weight} \right\}.
\eeqs
Observe that
\beqs
(\bu,\bv)_{\weight} = a_\abs(\bu,\bv) + (2\wn^2+ \ri \abs)\big( \bu,\bv\big)_{{\bf L}^2(\Omega)},
\eeqs
and, since $\proj_{\abs}^{\ell}\bv_h \in \Ned^\ell$, the definition of $\proj_{\abs}^{\ell}$ \eqref{eq:defQi}/\eqref{eq:defQi0} implies that 
\beqs
a_\abs((\bI - \proj_{\abs}^{\ell})\bv_h, \proj_{\abs}^{\ell} \bv_h)=0.
\eeqs
The results \eqref{eq:E4ii} and \eqref{eq:boundE} then follow from combining the last three equations and using the fact that $\supp \proj_{\abs}^{\ell}\bv_h \subseteq \Omega_\ell$.
\end{proof}

\subsection{The key results about the projection operators $\proj_{\abs}^{\ell}$ 
}\label{sec:43}

Lemma \ref{lem:combined} above shows that we can bound the field of values of $\proj_{\abs}$ away from the origin, provided that we can get good estimates for the ``remainder terms" $R_{\abs}^\ell(\bv_h)$, $\ell=0, \ldots, N$, given by \eqref{eq:boundE}.
It is at this point that our Maxwell theory deviates substantially from the Helmholtz theory in \cite{GrSpVa:17}. There, the Cauchy-Schwarz inequality was used on \eqref{eq:boundE}, 
\bit
\item[(a)] the analogue of $\|(\bI-\proj_{\abs}^0)\bv_h\|_{{\bf L}^2(\Omega_\ell)}$ was estimated using a duality argument (using $H^2$-regularity of the Laplacian in convex polyhedra), and
\item[(b)] the analogue of $\|\proj_{\abs}^{\ell}\bv_h\|_{{\bf L}^2(\Omega_\ell)}$, $\ell =1,\ldots,N$, was estimated using the standard scalar Poincar\'e--Friedrichs inequality (taking advantage of the fact that the analogue of $\proj_{\abs}^{\ell}\bv_h$ satisfies zero Dirichlet boundary conditions on at least part of $\partial \Omega_\ell$).
\eit
This approach does not immediately carry over to the Maxwell case because (a) duality arguments for Maxwell's equations are more complicated than for the Helmholtz equation, 
and (b) the appropriate analogue of the Poincar\'e--Friedrichs inequality does not hold for gradient fields (see Lemma \ref{lem:PF} above).

Nevertheless, one of the main technical results obtained by Gopalakrishnan and Pasciak, \cite[Lemma 4.3]{GoPa:03}, involves estimating
$( (\bI-\proj_{\abs}^{\ell})\bv_h, \bw_h^\ell)_{{\bf L}^2(\Omega_\ell)}$ 
for $\bw_h^\ell \in \Ned^\ell_h$ and $\ell =0, \ldots, N$. Lemma \ref{lem:jay} below is essentially this result, adapted to our situation by (i) making everything explicit in $\wn$ and $\abs$, and (ii) using coercivity of $a_\abs(\cdot,\cdot)$, instead of an error estimate on the Galerkin solution, in the duality argument.
(Additionally, as noted in \S\ref{sec:orthog}, the subdomains  in \cite{GoPa:03} are related to the coarse grid elements, but here we allow the subdomains to be unrelated to the coarse grid.)

Before stating this key result, we need to prove the following result about approximability of the adjoint problem on the coarse grid. 

\begin{lemma}\textbf{\emph{(Coarse-grid approximability of the adjoint problem)}}
\label{lem:approx}
If Assumption \ref{ass:reg} holds and $\bE$ is the solution of the adjoint problem \eqref{eq:BVPadj} with $\bF \in {\bf L}^2(\Omega)$ and $\dive\bF=0$,
then
\beq\label{eq:ass1}
 \inf_{\bphi_H\in \Ned_H} \N{\bE-\bphi_H}_{\weight} \lesssim \Hcs
  \left(\frac{\wn^2}{|\abs|}\right) \N{\bF}_{{\bf L}^2(\Omega)}.
\eeq
\end{lemma}

\bpf
When $\bQ_H$ is the space of lowest order N\'ed\'elec elements (i.e.~$m=0$ in \S\ref{sec:dis}), the inequality 
\beq\label{eq:JS}
 \inf_{\bphi_H\in \Ned_H} \left(\N{\curl(\bE-\bphi_H)}_{\bLtO} + \wn \N{\bE-\bphi_H}_{\bLtO}\right) \lesssim \Hcs
  \bigg(|\curl\bE|_{{\bf H}^1(\Omega)} + \wn |\bE|_{{\bf H}^1(\Omega)}\bigg)
\eeq
holds by \cite[Theorem 5 and Corollary 6]{Sc:01}. The fact that $\bQ_H$ for $m=0$ is contained in $\bQ_H$ for $m= 1, 2, \ldots$ means that \eqref{eq:JS} holds for $m=0, 1, 2,\ldots$, 
and then the result follows from using 
the $\wn$- and $\abs$-explicit bound \eqref{eq:reg_result2}.
\epf

We now state the key result that allows us to estimate the remainder terms $R_{\abs}^\ell (\bv_h)$.

\begin{lemma}\textbf{\emph{($\wn$- and $\abs$-explicit analogue of \cite[Lemma 4.3]{GoPa:03})}}\label{lem:jay}

\noi (i) For any $\bv_h\in \Ned_h$ and $\bw_h^\ell\in \Ned^\ell_h$,
\beq\label{eq:jay1}
\left|\big((\bI-\proj_{\abs}^{\ell})\bv_h, \bw_h^\ell \big)_{{\bf L}^2(\Omega_\ell)}\right| \lesssim \Hsub \| (\bI-\proj_{\abs}^{\ell}) \bv_h\big\|_{{\bf L}^2(\Omega_\ell)} \big\|\curl \bw_h^\ell\big\|_{{\bf L}^2(\Omega_\ell)}
\eeq
for $\ell= 1,\ldots,N$,

\noi 
(ii) If Assumption \ref{ass:reg} holds, then, given $\wn_0>0$, and for any $\bv_h\in \Ned_h$ and $\bw_H\in \Ned_H$,
\beq\label{eq:jay2}
\left|\big( (\bI-\proj_{\abs}^0)\bv_h, \bw_H \big)_{{\bf L}^2(\Omega)}\right| \lesssim \Hcs \left( \frac{\wn}{| \abs|}\right) 
\Cdual
\big\| 
(\bI-\proj_{\abs}^0) \bv_h\big\|_{\weight} \big\|
\bw_H\big\|_{\weight},
\eeq
for all $\wn\geq \wn_0$.
\end{lemma}

The proof of (i) is the same as in \cite[Lemma 4.3]{GoPa:03}, but since it is so short we include it here.
The proof of (ii) begins in the same way as in \cite[Lemma 4.3]{GoPa:03}, but then deviates substantially.

\

\bpf[Proof of Lemma \ref{lem:jay} (i)]
Given $\bv_h\in \Ned_h$, let $\be_{h,\ell}:= (\bI-\proj_{\abs}^{\ell})\bv_h$ for $\ell=1,\ldots,N$ (we use a subscript $\ell$, rather than a superscript, on $\be_{h,\ell}$ since it is not necessarily in $\Ned^\ell_h$).
We first show that
\beq\label{eq:EJ1}
\left(\be_{h,\ell}, \grad \phi^\ell_h \right)_{{\bf L}^2(\Omega^\ell)} =0 \quad\tfa \phi^\ell_h \in W^\ell_h \,\,\tfor \ell=1,\ldots,N,
\eeq
where $W_h^\ell$ is defined by \eqref{eq:Whell}.
Indeed, from the definition of $\proj_{\abs}^{\ell}$ \eqref{eq:defQi}, 
\beq\label{eq:EJ2}
a_\abs( \be_{h,\ell}, \bw_h^\ell) = 0 \quad\tfa \bw_h^\ell \in \Ned^\ell_h.
\eeq
Since 
\beqs
a_\abs( \be_{h,\ell}, \grad \phi^\ell_h) = -(\wn^2+ \ri \abs)\left( \be_{h,\ell}, \grad \phi^\ell_h\right)_{{\bf L}^2(\Omega^\ell)},
\eeqs
the result \eqref{eq:EJ1} follows from \eqref{eq:EJ2} with $\bw_h^\ell = \grad \phi^\ell_{h}$ (which is indeed in $\Ned^\ell_h$ by the decomposition \eqref{eq:orthog2}).

Given $\bw_{h}^\ell \in \Ned_h^\ell$, by \eqref{eq:orthog2} there exist  $\bz_h^\ell \in \RT^\ell_h$ and $\psi^\ell_h \in W_h^\ell$ such that
\beq\label{eq:EJ3}
\bw_h^\ell=\curl_h^\ell \bz_h^\ell + \grad \psi_h^\ell.
\eeq
Then, by \eqref{eq:EJ1}, 
\begin{align*}
\big|\left( \be_{h,\ell}, \bw_h^\ell\right)_{{\bf L}^2(\Omega_\ell)}\big| = \big|\left( \be_{h,\ell}, \curl_h^\ell \bz_h^\ell\right)_{{\bf L}^2(\Omega_\ell)} \big|
& \leq \,\big\|\be_{h,\ell}\big\|_{{\bf L}^2(\Omega_\ell)}\big\|\curl_h^\ell \bz_h^\ell\big\|_{{\bf L}^2(\Omega_\ell)}
\quad\text{ by Cauchy--Schwarz},
\\
&\lesssim \Hsub\,\big\|\be_{h,\ell}\big\|_{{\bf L}^2(\Omega_\ell)}\N{\curl \curl_h^\ell \bz_h^\ell}_{{\bf L}^2(\Omega_\ell)},
\end{align*}
by the Poincar\'e inequality \eqref{eq:PF1} (since $ \curl_h^\ell \bz_h^\ell \in \curl_h \RT^\ell_h$). The result \eqref{eq:jay1} then follows from 
by observing that 
$\curl \bw_h^\ell = \curl \curl_h^\ell \bz_h^\ell$ from \eqref{eq:EJ3}. 
\epf

\bpf[Proof Lemma \ref{lem:jay} (ii)]
Given $\bv_h\in \Ned_h$, let $\be_{h,0}:= (\bI-\proj_{\abs}^{0})\bv_h$. Exactly as in the proof of (i) we have that 
\beq\label{eq:EJ4}
\left(\be_{h,0}, \grad \phi_H \right)_{{\bf L}^2(\Omega)} =0 \quad\tfa \phi_H \in W_H.
\eeq
From the decomposition \eqref{eq:orthog}, there exist  $\br_h\in \RT_h$ and $\phi_h\in W_h$ such that
\beq\label{eq:EJ5}
\be_{h,0} = \curl_h \br_h + \grad \phi_h.
\eeq
Similarly, given $\bw_H\in \Ned_H$, there exist $\bz_H \in \RT_H$ and $\psi_H\in W_H$ such that
\beq\label{eq:EJ6}
\bw_H = \curl_H \bz_H + \grad \psi_H.
\eeq
Then, by \eqref{eq:EJ4}, we have 
\begin{align}\nonumber
\left( \be_{h,0}, \bw_H \right)_{\bLtO} &= \left(\be_{h,0} , \curl_H \bz_H\right)_\bLtO ,\\
&= \left( \curl_h \br_h, \curl_H \bz_H\right)_\bLtO + \left(\grad \phi_h, \curl_H \bz_H\right)_\bLtO=: I_1 + I_2.\label{eq:EJ6a}
\end{align}
We estimate $I_1$ and $I_2$ separately; we begin by estimating $I_2$, since this is slightly easier.

In Lemma \ref{lem:magic} we recalled the properties of $\bS : \curl_h \RT_h \rightarrow \HocO$, but analogous properties hold for $\bS : \curl_H \RT_H \rightarrow \HocO$.
We let $\bq= \bS(\curl_H \bz_H)$, and we seek to introduce $\curl_H \bz_H - \bq$ in the inner product defining $I_2$ so that we can obtain a power of $\Hcs$ via \eqref{eq:magic}. Indeed, 
\begin{align}\nonumber
I_2 &=\big(  \grad \phi_h, \curl_H\bz_H - \bq\big)_\bLtO + \big( \grad \phi_h, \bq\big)_\bLtO,\\ \nonumber
&= \big(  \grad \phi_h, \curl_H\bz_H - \bq\big)_\bLtO - \big( \phi_h, \dive \bq\big)_{\LtO} \quad\text{since } \phi_h\in H_0^1(\Omega),\\ \nonumber
&= \big(  \grad \phi_h, \curl_H\bz_H - \bq\big)_\bLtO 
\qquad\qquad\quad\text{since } \dive \bq =0 \text{ by Lemma \ref{lem:magic}},\\ \nonumber
&\leq\N{\grad \phi_h}_\bLtO \N{\curl_H\bz_H -\bq}_\bLtO \quad\text{ by Cauchy--Schwarz},\\ \nonumber
& \leq \normeps\N{\be_{h,0}}_{\bLtO}\N{\curl_H\bz_H -\bq}_\bLtO \quad\qquad\text{ by the orthogonal decomposition \eqref{eq:EJ5}},\\ \nonumber
& \lesssim \Hcs \normeps\N{\be_{h,0}}_{\bLtO} \N{\curl \curl_H\bz_H}_{\bLtO} \quad\text{ by \eqref{eq:magic}},\\
& = \Hcs \normeps\N{\be_{h,0}}_{\bLtO} \N{\curl \bw_{H}}_{\bLtO}  \,\,\,\quad\qquad\text{ by the orthogonal decomposition \eqref{eq:EJ6}.}\label{eq:EJ7}
\end{align}
We now estimate $I_1$. 
The main work is in using a duality argument to show that
\beq\label{eq:EJ8}
\N{\curl_h \br_h}_\bLtO \lesssim\Hcs \kseb 
\N{\be_{h,0}}_{\weight}.
\eeq
(Recalling \eqref{eq:EJ5}, we see that \eqref{eq:EJ8} bounds one component of the ``error" $\be_{h,0}$ in ${\bf L}^2(\Omega)$ by the $\curlt, \wn$-norm of the whole ``error" multiplied by $\Hcs k^2/|\abs|$).
Assuming we have \eqref{eq:EJ8}, we find that
\begin{align}\nonumber
I_1 &\lesssim \Hcs \left( \frac{\wn^2}{|\abs|}\right) \N{\be_{h,0}}_{\weight} \N{\curl_H \bz_H}_{\bLtO},\\
& \lesssim \Hcs\left( \frac{\wn^2}{|\abs|}\right) \, \,\N{\be_{h,0}}_{\weight} \N{\bw_H}_{\bLtO}
\quad\text{ by the orthogonal decomposition \eqref{eq:EJ6}}.\label{eq:EJ9}
\end{align}
Then, by using \eqref{eq:EJ7} and \eqref{eq:EJ9} in \eqref{eq:EJ6a}, and also by recalling that $|\abs|\lesssim \wn^2$, we have
\begin{align*}
\big|\left( \be_{h,0}, \bw_H \right)_{\bLtO}\big| &\lesssim 
\Hcs\left( \frac{\wn^2}{|\abs|}\right) \,
 \,
\N{\be_{h,0}}_{\weight} \N{\bw_H}_{\bLtO}
+ \Hcs \N{\be_{h,0}}_{\bLtO} \N{\curl \bw_{H}}_{\bLtO}\\
& \lesssim\Hcs\left( \frac{\wn}{|\abs|}\right)
\N{\be_{h,0}}_{\weight} \N{\bw_{H}}_{\weight}
\end{align*}
which is the result \eqref{eq:jay2}. 

To finish the proof, we only need to establish \eqref{eq:EJ8}.
We use the result \eqref{eq:ass1} about coarse-grid approximability of the adjoint problem, and so we need to consider a BVP with divergence-free source. Let $\dual: = \bS(\curl_h \br_h)$. Then
\beq\label{eq:EJ9a}
\N{\dual- \curl_h \br_h}_{\bLtO} \lesssim h \N{\curl \curl_h \br_h}_\bLtO = h \N{\curl \be_{h,0}}_{\bLtO}
\eeq
by \eqref{eq:magic} and \eqref{eq:EJ5}. Furthermore  $\dive \dual =0$ and
\beq\label{eq:EJ10}
\curl \dual = \curl \curl_h \br_h.
\eeq
We now set up the standard duality argument for $\dual$: let $\bu\in \HocO$ be the solution of the variational problem
\beq\label{eq:vpduality}
a_\abs(\bv, \bu) = (\bv, \dual)_\bLtO \quad\tfa \bv \in \HocO;
\eeq
i.e.~from \eqref{eq:adj_ses} $\bu$ is the solution of the adjoint variational problem \eqref{eq:adjoint_form} with source term $\dual$. Therefore,
\beq\label{eq:EJ10a}
\N{\dual}^2_\bLtO = a_\abs(\dual, \bu).
\eeq
We now want to use the Galerkin orthogonality property of $\be_{h,0}$:
\beq\label{eq:EJ10b}
a_\abs( \be_{h,0}, \bphi_H) = 0 \quad \tfa \bphi_H \in \Ned_H.
\eeq
Our goal is therefore to create an $a_\abs(\be_{h,0}, \bu)$ on the right-hand side of \eqref{eq:EJ10a}, so that we can replace it by 
$a_\abs(\be_{h,0}, \bu-\bphi_H)$ for an arbitrary $\bphi_H^0\in \Ned_H$, use continuity of $a_\abs(\cdot,\cdot)$ \eqref{eq:cont}, and then use the \eqref{eq:ass1}.
We claim that
\beq\label{eq:EJ11}
a_\abs(\dual,\bu) = a_\abs(\be_{h,0}, \bu) - (\wn^2+ \ri\abs)\left(\big(\dual - \curl_h \br_h\big), \bu\right)_{\bLtO}.
\eeq
Indeed, from the definition of $a_\abs(\cdot,\cdot)$ \eqref{eq:Helmholtzvf_intro}, the equation \eqref{eq:EJ10}, and the decomposition \eqref{eq:EJ5}, 
we have
\begin{align}\nonumber
a_\abs(\dual,\bu) &= \big( \curl \curl_h \br_h, \curl \bu\big)_\bLtO  - (\wn^2 + \ri \abs)\left(\dual, \bu\right)_\bLtO,\\ \nonumber
&= \big( \curl \be_{h,0}, \curl \bu\big)_\bLtO  -(\wn^2 + \ri \abs)\left( \be_{h,0}, \bu\right)_\bLtO
- (\wn^2 + \ri \abs)\left( \big(\dual-\be_{h,0}), \bu\right)_\bLtO,\\
&= a_\abs(\be_{h,0}, \bu) - (\wn^2 + \ri \abs)\left(\big(\dual - \curl \br_h - \grad \phi_h\big), \bu\right)_\bLtO.\label{eq:EJ12}
\end{align}
Using \eqref{eq:vpduality} with $\bv= \grad \phi_h$, recalling that $\phi_h\in W_h\subset H_0^1(\Omega)$ and $\dive \dual=0$ (and thus $(\grad \phi_h , \dual)_\bLtO=0$), we have
\beqs
\left( \grad \phi_h, \bu\right)_\bLtO=0.
\eeqs
Using this in \eqref{eq:EJ12} we find \eqref{eq:EJ11}.

Therefore, combining \eqref{eq:EJ10a}, \eqref{eq:EJ10b}, and \eqref{eq:EJ11}, and then using continuity of $a_\abs(\cdot,\cdot)$ \eqref{eq:cont},  we get
\begin{align}\nonumber
\N{\dual}^2_\bLtO &= a_\abs(\be_{h,0}, \bu- \bphi_H) -(\wn^2+ \ri\abs)\left(\big(\dual - \curl_h \br_h\big), \bu\right)_{\bLtO},\\
&\lesssim \Ccont \N{\be_{h,0}}_{\weight} \N{\bu-\bphi_H}_{\weight} + \wn^2\N{\dual - \curl_h \br_h}_{\bLtO}\big\|\bu\big\|_{\bLtO}.\label{eq:EJ13}
\end{align}
From the result of the Lax--Milgram theorem \eqref{eq:LM} applied to the variational problem \eqref{eq:vpduality}, we have
\beqs
\big\|\bu\big\|_{\bLtO}\lesssim 
\left(\frac{1}{|\abs|}\right)\N{\dual}_{\bLtO}.
\eeqs
Using this last estimate in \eqref{eq:EJ13}, along with \eqref{eq:EJ9a} and \eqref{eq:ass1}, we have 
\begin{align}
\N{\dual}_\bLtO
\lesssim\kseb\bigg[\Hcs
 \N{\be_{h,0}}_{\weight}
+ h
 \N{\curl\be_{h,0}}_{\bLtO}\bigg]\lesssim\kseb\left[\Hcs
+ h
\right] \N{\be_{h,0}}_{\weight}.
\label{eq:EJ14}
\end{align}
Now, by \eqref{eq:EJ9a}, 
\beqs
\N{\curl_h\br_h}_{\bLtO} \leq \N{\dual}_{\bLtO} + h \N{\be_{h,0}}_{\weight}.
\eeqs
Combining this with \eqref{eq:EJ14}, and using the crude estimate $h\leq \Hcs$, we obtain \eqref{eq:EJ8} and the proof is complete.
\epf

\subsection{Bounding the field of values away from the origin}\label{sec:44}
  
We now use the results from \S\ref{sec:43} in conjunction with Lemma \ref{lem:combined} to bound the field of values of $\proj_{\abs}$ away from the origin.

\begin{lemma}\textbf{\emph{(Bound on $R_{\abs}^0 (\bv_h)$)}}\label{lem:coarse}
Under Assumption \ref{ass:reg}, 
given $\wn_0>0$,
\beq\label{eq:E5}
\big|R_{\abs}^0 (\bv_h)\big| \lesssim \wn \Hcs \kseb\Cdual \left[ \kseb^\gamma \N{\proj_{\abs}^0\bv_h}^2_{\weight} + \kseb^{-\gamma}\N{\bv_h}^2_{\weight}\right],
\eeq
for all $\wn\geq \wn_0$, for all $\gamma\geq 0$, and for all $\bv_h\in \Ned_h$.
\end{lemma}

\bpf
Using (i) the bound \eqref{eq:jay2} in the expression for $R_{\abs}^0 (\bv_h)$ \eqref{eq:boundE}, (ii) the triangle inequality, (iii) 
the inequality 
\beq\label{eq:Cauchy}
2ab \leq \epsilon a^2 + \frac{1}{\epsilon}b^2 \quad\tfor a, b, \epsilon>0,
\eeq
and (iv) the fact that $|\abs|\lesssim \wn^2$, 
we obtain
\begin{align*}\nonumber
\big|R_{\abs}^0 (\bv_h)\big|& \lesssim  \wn \Hcs \kseb\Cdual\N{(\bI-\proj_{\abs}^0)\bv_h}_{\weight} \N{\proj_{\abs}^0\bv_h}_{\weight},\\
&\lesssim  \wn \Hcs \kseb \Cdual\left[
\N{\proj_{\abs}^0\bv_h}^2_{\weight} + \N{\proj_{\abs}^0\bv_h}_{\weight}\N{\bv_h}_{\weight}\right],\\
&\lesssim  \wn \Hcs \kseb \Cdual\left[ \N{\proj_{\abs}^0\bv_h}^2_{\weight} + \kseb^{\gamma} \N{\proj_{\abs}^0\bv_h}_{\weight}^2 + \kseb^{-\gamma}\N{\bv_h}^2_{\weight}\right]\ , 
\end{align*}
for any $\gamma \geq 0 $. Since  $\veps\lesssim \wn^2$,  \eqref{eq:E5} follows.  
\epf

\begin{lemma}\textbf{\emph{(Bound on $\sum R_{\abs}^\ell (\bv_h)$)}}\label{lem:local}
For any $\gamma'\geq 0$ and any $\bv_h\in \Ned_h$,
\beq\label{eq:E7}
\sum_{\ell= 1}^N \big|R_{\abs}^\ell (\bv_h)\big| \lesssim  \wn \Hsub  \left[
\kseb^{\gamma'} \sum_{\ell=1}^N \N{\proj_{\abs}^\ell\bv_h}^2_{\weight} + \kseb^{-\gamma'} \N{\bv_h}^2_{\weight}\right].
\eeq
\end{lemma}
\begin{proof} 
Using the bound \eqref{eq:jay1} and the triangle inequality, we find
\begin{align*}
\big| R_{\abs}^\ell(\bv_h)\big|&\lesssim  \wn^2\Hsub  \big\|(\bI-\proj_{\abs}^{\ell})\bv_h\big\|_{{\bf L}^2(\Omega_\ell)} \big\|\curl (\proj_{\abs}^{\ell}\bv_h)\big\|_{{\bf L}^2(\Omega_l)},\\
&\lesssim \wn^2\Hsub   \left[ \big\|\bv_h\big\|_{{\bf L}^2(\Omega_\ell)}\big\|\curl (\proj_{\abs}^{\ell}\bv_h)\big\|_{{\bf L}^2(\Omega_{\ell})}
+\big\|\proj_{\abs}^{\ell}\bv_h\big\|_{{\bf L}^2(\Omega_l)} \big\|\curl(\proj_{\abs}^{\ell}\bv_h)\big\|_{{\bf L}^2(\Omega_l)} \right],\\
&  \lesssim \wn \Hsub  \left[ \wn\big\|\bv_h\big\|_{{\bf L}^2(\Omega_\ell)} \big\|\proj_{\abs}^{\ell}\bv_h\big\|_{\weight}+\N{\proj_{\abs}^{\ell}\bv_h}_{\weight}^2 \right]\ 
\end{align*}
(where the $\curlt,\wn$-norm is over  the support of 
$\proj_{\abs}^{\ell}\bv_h$, which is $\Omega_\ell$).
Using the inequality \eqref{eq:Cauchy}
 and the fact that $|\abs|\lesssim \wn^2$, 
we obtain
\beqs
\big|R_{\abs}^\ell(\bv_h)\big|\lesssim \wn \Hsub  \left[ \kseb^{\gamma'} \N{\proj_{\abs}^{\ell}\bv_h}^2_{\weight} + \wn^2 \kseb^{-\gamma'} \N{\bv_h}^2_{{\bf L}^2(\Omega_\ell)}\right], 
\eeqs
with $\gamma'\geq0$. Summing from $\ell=1$ to $N$, 
and using the finite-overlap property \eqref{eq:finoverlap},  gives \eqref{eq:E7}.
\end{proof}

We now obtain the bound on the field of values of $\proj_{\abs}$. 

\begin{theorem}\textbf{\emph{(Bounding the field of values away from the origin)}} \label{thm:final} 
If Assumption \ref{ass:reg} holds, given $\wn_0>0$, there exists a  constant $\cC_1>0$ 
such that if $\wn\geq\wn_0$ and 
\beq\label{eq:E20}
\max\left\{ (\wn\Hsub)\Csub,\  (\wn \Hcs)  \kseb \Ccs \right\}  \ \leq \  
\cC_1 \bigg(1 + \left(\frac{\Hcs}{\delta}\right)^2\bigg)^{-1}  \left(\frac{\vert \abs \vert }{\wn^2}\right), 
\eeq
then
\beq\label{eq:E12}
\frac{\vert (\bv_h, \proj_\abs \bv_h )_{\weight} \vert }
{\Vert \bv_h \Vert_{\weight}^2 }\  \gtrsim \  \bigg(1+ \left(\frac{\Hcs}{\delta}\right)^2\bigg)^{-1}\left(\frac{\vert \abs\vert }{\wn^2} \right)^{2}\quad \text{for all }\,  \bv_h \in \Ned_h.
\eeq
\end{theorem}

\

\bpf
By Lemma \ref{lem:combined},
\beqs
\left|\left(\bv_h, \proj_\abs \bv_h\right)_{\weight}\right| \ \geq\  \sum_{\ell=0}^N \N{\proj_{\abs}^{\ell}\bv_h}_{\weight}^2 - \sum_{\ell=0}^N \big|R_{\abs}^\ell(\bv_h)\big|.
\eeqs
Then, using the bounds \eqref{eq:E5} and \eqref{eq:E7} 
we have
\begin{align*}
&\left|\left(\bv_h, \proj_\abs \bv_h\right)_{\weight}\right|\  \gtrsim \ \sum_{\ell=0}^N \N{\proj_{\abs}^{\ell}\bv_h}_{\weight}^2 
- (\wn \Hcs)\kseb\Cdual \left[ \kseb^{\gamma} \N{\proj_{\abs}^0\bv_h}_{\weight}^2 + \kseb^{-\gamma} \N{\bv_h}^2_{\weight}\right] \\
&\hspace{4cm}-(\wn\Hsub)  \left[ \kseb^{\gamma'} \sum_{\ell=1}^N\N{\proj_{\abs}^{\ell}\bv_h}_{\weight}^2 + \kseb^{-\gamma'} \N{\bv_h}^2_{\weight}\right] 
\end{align*}
for $\gamma,\gamma'\geq 0$.
Therefore, there exist $C_1, C_2>0$ (sufficiently small but independent of all parameters)
such that
\beq\label{eq:N1}
(\wn \Hcs) \kseb^{1+\gamma}\Cdual  \leq  C_1 \quad\tand\quad
(\wn\Hsub) \kseb^{\gamma'} \leq  C_2
\eeq
ensure that
\begin{align}
\left|\left(\bv_h, \proj_\abs \bv_h\right)_{\weight}\right| \ \gtrsim \ & \sum_{\ell=0}^N \N{\proj_{\abs}^{\ell}\bv_h}_{\weight}^2
-(\wn \Hcs) \kseb^{1-\gamma} \N{\bv_h}^2_{\weight}
- (\wn\Hsub) \kseb^{-\gamma'} \N{\bv_h}^2_{\weight}.
\label{eq:N3}
\end{align}
Using  the bound \eqref{eq:45bound} from Lemma \ref{lem:lowrestQi} in \eqref{eq:N3}, we obtain 
\begin{align*}
\left|\left(\bv_h, \proj_\abs \bv_h\right)_{\weight}\right|
\gtrsim  &
\bigg(1+ \left(\frac{\Hcs}{\delta}\right)^2\bigg)^{-1} \eksb^2 \N{\bv_h}^2_{\weight} -(\wn \Hcs) \kseb^{1-\gamma} \N{\bv_h}^2_{\weight}\\
&\qquad-(\wn\Hsub) \kseb^{-\gamma'} \N{\bv_h}^2_{\weight}.
\end{align*}
Therefore, there exist $C_3, C_4>0$ (sufficiently small but independent of all parameters) so that the conditions
\beq\label{eq:E13}
(\wn \Hcs) \kseb^{1-\gamma} \leq C_3 \bigg(1+ \bigg(\frac{\Hcs}{\delta}\bigg)^2\bigg)^{-1} \eksb^2
\tand
(\wn\Hsub) \kseb^{-\gamma'} \leq  C_4  \bigg(1+ \bigg(\frac{\Hcs}{\delta}\bigg)^2\bigg)^{-1} \eksb^2,
\eeq
together with \eqref{eq:N1} 
ensure that the result \eqref{eq:E12} holds.
The conditions in \eqref{eq:E13} 
can then be rewritten as 
\beq\label{eq:E13alt}
(\wn \Hcs) \kseb^{3-\gamma} 
\bigg(1+ \bigg(\frac{\Hcs}{\delta}\bigg)^2\bigg)
\Cstab
\leq C_3 
\quad\tand\quad
(\wn\Hsub) \kseb^{2-\gamma'}
\bigg(1+ \bigg(\frac{\Hcs}{\delta}\bigg)^2\bigg)
\Cstab
 \leq  C_4.
\eeq

In summary, we have shown that there exist $C_1, C_2, C_3, C_4>0$ such that the required 
result \eqref{eq:E12} holds if \eqref{eq:N1} 
and \eqref{eq:E13alt} hold. 

The optimal choice of $\gamma$ to balance the exponents in the first equations in \eqref{eq:N1} and \eqref{eq:E13alt} (ignoring the factor $(1+ (H/\delta)^2)$) is $\gamma=1$, and the optimal choice of $\gamma'$ to balance the exponents in the second equations in \eqref{eq:N1} and \eqref{eq:E13alt} (again ignoring $(1+ (H/\delta)^2)$) is also $\gamma'=1$. With these values of $\gamma$ and $\gamma'$, 
the four conditions above are ensured by the  condition \eqref{eq:E20} and the proof is complete. 
\epf

\section{Matrices and convergence of GMRES}
\label{sec:Matrices}

In this section we convert the results of Theorems \ref{thm:boundQ} and  \ref{thm:final} into results about matrices, giving results about preconditioning $A_\abs$ (Theorems \ref{thm:final1}-\ref{cor:final4} below).

\subsection{From projection operators to matrices}\label{sec:51}

We first interpret the operators $\proj_{\abs}^{\ell}$ defined in \eqref{eq:defQi}, \eqref{eq:defQi0} in terms of matrices.

\begin{lemma} \label{thm:repQ}
Let $\bv_h = \sum_{j\in \cI^h} V_j \phi_j \ \in \ \Ned_h$. Then 
\begin{align*}
(i)  \quad \proj_{\abs}^{\ell} \bv_h & = \sum_{j \in \cI^h({\Omega_\ell})} \left((R^\ell)^T(A_{\abs}^\ell)^{-1} R^\ell A_\abs \mathbf{V}\right)_j \phi_j \ ,   \quad \ell = 1, \ldots, N, \\
(ii)  \quad \proj_{\abs}^0 \bv_h & = \sum_{p \in \cI^H} \left(R_0^T (A_{\abs}^0)^{-1} R_0 A_\abs 
\mathbf{V}\right)_p \Phi_p,  
\end{align*} 
where $A_{\abs}^\ell\ , \ell = 0, \ldots, N$ is defined in \eqref{eq:minor}.
\end{lemma}

\noi We omit the proof, since it is essentially identical to the proof of \cite[Theorem 5.4]{GrSpVa:17}.

\

The main results of the previous section - Theorems  \ref{thm:boundQ} 
and \ref{thm:final} -  give estimates for the norm and the field of values of the operator $\proj_{\abs}$ on the space $\Ned_h$, with respect to the inner product $(\cdot, \cdot)_{\weight}$ and its associated  norm. 

The next lemma shows that, in order to translate these results into norm and 
field of values estimates for the preconditioned matrix $B_{\abs, AS}^{-1}A_\abs$, we need to work in the weighted 
inner product $ \langle \cdot, \cdot\rangle_{D_\wn}$ defined such that 
if $\bv_h, \bw_h \in \Ned_h$ with coefficient vectors $\bV, \bW$ then 
\begin{equation}\label{eq:ip}
(\bv_h, \bw_h )_{\weight}\  =\ \langle \bV, \bW\rangle_{D_\wn};
\end{equation}    
note that the definition of $(\cdot, \cdot)_{\weight}$ \eqref{eq:weight} means that $\langle \bV, \bW\rangle_{D_\wn}$ depends on the wavenumber  $\wn$.
 
\begin{lemma}\label{ipnormQ}
Let $\bv_h = \sum_{j\in \cI^h} V_j \phi_h \ \in \ \Ned_h$. Then
\beqs
(i) \quad (\bv_h, \proj_\abs \bv_h)_{\weight}  = \langle \bV, B_{\abs, AS} ^{-1} A_\abs \bV\rangle_{D_\wn}, \quad\tand\quad
(ii) \quad\Vert \proj_\abs \bv_h\Vert _{\weight}   = \Vert B_{\abs, AS}^{-1} A_\abs \bV\Vert_{D_\wn}.
\eeqs
\end{lemma}
\begin{proof}
(i) and (ii) follow from combining Lemma \ref{thm:repQ}, Equation \eqref{eq:ip}, and the definition of $B_{\abs, AS}^{-1}$.
\eqref{eq:defAS}
\end{proof}

\subsection{Recap of Elman-type estimates for convergence of GMRES}\label{sec:52}

We consider the abstract  linear system 
 \begin{equation*}
\matrixC \bfx = \bfd 
\end{equation*}
in $\mathbb{C}^n$, where $\matrixC$ is an  $n\times n$ nonsingular complex matrix.   
Given an initial guess $\bfx^0$, we introduce the residual $\bfr^0 = \bfd- C \bfx^0$ and 
the usual Krylov spaces:  
$$  \cK^m(C, \bfr^0) := \mathrm{span}\{\matrixC^j \bfr^0 : j = 0, \ldots, m-1\} \ .$$
Let $\langle \cdot , \cdot \rangle_D$ denote the inner product on $\C^n$ 
induced by some Hermitian positive definite  matrix $D$, i.e.  
\begin{equation*}
\langle \bV, \bW\rangle_D := \bW^*D\bV
\end{equation*} 
with induced norm $\Vert \cdot \Vert_D$, where $^*$ denotes Hermitian transpose. For $m \geq 1$, define   $\bfx^m$  to be  the unique element of $\cK^m$ satisfying  the  
 minimal residual  property: 
$$ \ \Vert \bfr^m \Vert_D := \Vert \bfd - \matrixC \bfx^m \Vert_D \ = \ \min_{\bfx \in \cK^m(C, \br^0)} \Vert {\bfd} - {\matrixC} {\bfx} \Vert_D. $$
When $D = I$ this is just the usual GMRES algorithm, and we write use $\Vert \cdot \Vert$ to denote $\Vert \cdot \Vert_I$, but for  more general  $D$ it 
is the weighted GMRES method \cite{Es:98} in which case  
its implementation requires the application of the weighted Arnoldi process \cite{GuPe:14}.

The following theorem is a simple generalisation to the weighted setting of the GMRES convergence result of Beckermann 
Goreinov and Tyrtyshnikov \cite[Theorem 2.1]{BeGoTy:06}. This result is an improvement of the so-called ``Elman estimate", originally due to Elman \cite{El:82}; see also \cite{EiElSc:83}, \cite[Theorem 3.2]{St:97}, \cite[Corollary 6.2]{EiEr:01}, \cite{LiTi:12}, and the review \cite[\S6]{SiSz:07}.

\begin{theorem}\textbf{\emph{(Elman-type estimate for weighted GMRES)}} \label{thm:Elman}  
Let $C$ be a matrix with $0\notin W_D(C)$, where 
\beqs
W_D(C):= \big\{ \langle C \bv, \bv\rangle_D : \bv \in \Com^N, \|\bv\|_D=1\big\}
\eeqs
is the \emph{field of values}, also called the \emph{numerical range} of $C$ with respect to the inner product $\langle\cdot,\cdot\rangle_D$. 
Let $\beta\in [0,\pi/2)$ be defined such that
\beq\label{eq:cosbeta}
\cos \beta = \frac{\mathrm{dist}\big(0, W_D(\matrixC)\big)}{\| \matrixC\|_{D}}
\eeq
(observe that $\beta$ is well-defined since the right-hand side of \eqref{eq:cosbeta} is $\leq 1$).
Let $\gamma_\beta$ be defined by 
\beq\label{eq:gamma_beta}
\gamma_\beta:= 2 \sin \left( \frac{\beta}{4-2\beta/\pi}\right),
\eeq
and let $\br_m$ be defined as above.
Then
\beq\label{eq:Elman2}
\frac{\|\br_m\|_{D}}{\|\br_0\|_{D}} \leq \left(2 + \frac{2}{\sqrt{3}}\right)\big(2+ \gamma_\beta\big) \,(\gamma_\beta)^m.
\eeq
\end{theorem}

\bpf
The result with $D=I$ and \eqref{eq:Elman2} replaced by 
\beq\label{eq:Elman}
\frac{\|\br_m\|_{D}}{\|\br_0\|_{D}} \leq \sin^m \beta
\eeq
was proved in \cite{El:82, EiElSc:83}; this result was extended for general Hermitian positive-definite $D$ by an elementary argument in \cite[Theorem 5.1]{GrSpVa:17}. The result \eqref{eq:Elman2} with $D=I$ was proved in \cite[Theorem 2.1]{BeGoTy:06}, and can be extended to general $D$
using the arguments in \cite[Theorem 5.1]{GrSpVa:17}.
\epf

When we apply the estimate \eqref{eq:Elman2} to $B_{\abs, AS}^{-1}$ in \S\ref{sec:53} below, we find that $\beta= \pi/2-\epsilon$, where (for fixed $\delta, \Hcs$) $\epsilon= \epsilon (\wn,\abs)$ is such that $\epsilon\tendo$ as $\wn\tendi$.
It is therefore convenient to specialise the results \eqref{eq:Elman} and \eqref{eq:Elman2} to this particular situation in the following corollary.

\begin{corollary}\label{cor:Elman}
With $C$ a matrix such that $0\notin W_D(C)$, let $\epsilon\in (0,\pi/2]$ be defined such that
\beq\label{eq:cosbeta2}
\sin \epsilon = \frac{\mathrm{dist}\big(0, W_D(\matrixC)\big)}{\| \matrixC\|_{D}}.
\eeq
Then there exists $\cC>0$ (independent of $\epsilon$) such that, for $0<a<1$, 
\beq\label{eq:corE1}
\text{if} \quad m\geq \frac{\cC}{\epsilon}\log\left(\frac{12}{a}\right)\quad\text{ then }  \quad\frac{\|\br_m\|_{D}}{\|\br_0\|_{D}}\leq a.
\eeq
\end{corollary}

\noi That is, choosing $m\gtrsim \epsilon^{-1}$ is sufficient for GMRES to converge in an $\epsilon$-independent way as $\epsilon\tendo$.

\bpf[Proof of Corollary \ref{cor:Elman}]
We first prove that it is sufficient to establish the result for $\epsilon$ sufficiently small. Suppose there exist 
$0<\epsilon_1<\pi/2$ and $\cC_1(\epsilon_1)>0$ (independent of $\epsilon$) such that, for $0<a<1$ and $0<\epsilon<\epsilon_1$,
\beq\label{eq:corE1a}
\text{if} \quad m\geq \frac{\cC_1(\epsilon_1)}{\epsilon}\log\left(\frac{12}{a}\right)\quad\text{ then }  \quad\frac{\|\br_m\|_{D}}{\|\br_0\|_{D}}\leq a.
\eeq
Setting $\beta=\pi/2-\epsilon$ we see that \eqref{eq:cosbeta2} implies \eqref{eq:cosbeta}. Observe that $\gamma_\beta$ in \eqref{eq:Elman2} is a continuous function of $\epsilon$ and decreases from $1$ to $0$ as $\epsilon$ increases from $0$ to $\pi/2$. Hence, if $\epsilon_1\leq \epsilon \leq \pi/2$, then there exists a constant $\cC_2(\epsilon_1)$ such that $\gamma_\beta\leq 1- \cC_2(\epsilon_1)$ and thus 
\beqs
(\gamma_\beta)^m \leq \exp\Big( - m \big|\log \big(1-\cC_2(\epsilon_1)\big)\big|\Big).
\eeqs
Since $(2 + 2/\sqrt{3})(2+ \gamma_\beta)<3(2 + 2/\sqrt{3})<12$, it follows from \eqref{eq:Elman2} that, 
if $\epsilon_1\leq\epsilon\leq \pi/2$,
$\|\br_m\|_{D}/\|\br_0\|_{D}\leq a$ provided that
\beqs
\exp\Big( - m \big|\log \big(1-\cC_2(\epsilon_1)\big)\big|\Big)\leq \frac{a}{12},\quad\text{ i.e. }\quad
m\geq \frac{1}{\big|\log \big(1-\cC_2(\epsilon_1)\big)\big|}\log \left(\frac{12}{a}\right).
\eeqs
Assuming we have proved that \eqref{eq:corE1a} holds for all $0<\epsilon<\epsilon_1$, we then set
\beq\label{eq:C1}
\cC:= \max\bigg[ \cC_1(\epsilon_1) , \frac{\pi}{2|\log (1-\cC_2(\epsilon_1)|}\bigg]
\eeq
and then \eqref{eq:corE1} holds for all $0<\epsilon\leq \pi/2$ (with the first term in the max in \eqref{eq:C1} dealing with  $0<\epsilon<\epsilon_1$ and the second term dealing with  $\epsilon_1\leq \epsilon\leq \pi/2$).

We therefore only need to prove that 
 there exist $\epsilon_1>0$ and $\cC_1(\epsilon_1)>0$ such that for $0<a<1$, \eqref{eq:corE1a} holds for all $0<\epsilon<\epsilon_1$.
If $\beta= \pi/2-\epsilon$, then $\cos\beta=\sin\epsilon = \epsilon+ \cO(\epsilon^3)$ as $\epsilon \tendo$. From the definition of $\gamma_\beta$ in \eqref{eq:gamma_beta}, we have
\beq\label{eq:factor}
\gamma_\beta:=2 \sin \left( \frac{\beta}{4-2\beta/\pi}\right)= 2 \sin \left( \frac{\pi}{6} -\frac{4\epsilon}{9}+\cO(\epsilon^2)\right)= 1 -\frac{4\epsilon}{3\sqrt{3}} + \cO(\epsilon^2)\quad\tas\,\, \epsilon\tendo,
\eeq
and then
\beqs
\log \gamma_\beta = -\frac{4\epsilon}{3\sqrt{3}}  + \cO(\epsilon^2) \quad\tas\,\, \epsilon\tendo,
\eeqs
and so there exist $\epsilon_1>0$ and $\cC_1(\epsilon_1)>0$ such that
\beqs
(\gamma_\beta)^m = \re^{m\log\gamma_\beta}\leq \re^{-m\epsilon/\cC_1} \quad\tfa0< \epsilon\leq \epsilon_1,
\eeqs
and \eqref{eq:corE1a} follows from \eqref{eq:Elman2}.
\epf

\bre
When $\beta= \pi/2-\epsilon$, the convergence factor in the original Elman estimate \eqref{eq:Elman} is
\beqs
\sin\beta = \cos\epsilon 
= 1- \frac{\epsilon^2}{2}+\cO(\epsilon^4);
\eeqs
by comparing this to \eqref{eq:factor} we can see that \eqref{eq:Elman} is indeed a weaker bound.
\ere

\subsection{The main results}\label{sec:53}

With the classical two-level additive Schwarz preconditioner for $A_\abs$ ($\matrixB_{\abs, \text{AS}}^{-1}$
defined by
\eqref{eq:defAS} above) and the inner product $\langle\cdot,\cdot\rangle_{D_\wn}$ defined by \eqref{eq:ip} above, we have the following results.

\begin{theorem}\textbf{\emph{(Bounds on the norm and field of values for left preconditioning)}}\label{thm:final1}

\noi (i) 
\vspace{-3ex}
\begin{align*}
\N{ B_{\abs, AS}^{-1} A_\abs }_{D_\wn}  & \ \lesssim  \ \left(\ksqeps\right) \quad \text{for all} \,\, \Hcs, \Hsub.
\end{align*}
(ii) If $\Omega$ is a convex polyhedron,
then, given $\wn_0>0$, there exists a constant $\cC_1$ such that 
if 
\beq\label{eq:E20rpt}
\max\left\{ (\wn\Hsub)\Csub,\  (\wn \Hcs)  \kseb \Ccs \right\}  \ \leq \  
\cC_1 \bigg(1 + \left(\frac{\Hcs}{\delta}\right)^2\bigg)^{-1}  \left(\frac{\vert \abs \vert }{\wn^2}\right),
\eeq
then
 \begin{align}\nonumber
\frac{
\Big\vert \left\langle \bV, 
B_{\abs, AS}^{-1}A_\abs \bV \right\rangle_{D_\wn} \Big\vert
}{
\left\Vert \bV \right\Vert_{D_\wn}^2
}  & \ \gtrsim   \ 
\bigg(1 + \left(\frac{\Hcs}{\delta}\right)^2\bigg)^{-1} 
\left(\frac{\vert \abs\vert }{\wn^2} \right)^{2}, 
\end{align}
for all $\bV \in \Com^n$ and  for all $\wn\geq \wn_0$. 
\end{theorem}

\bpf
This follows from combining Theorems \ref{thm:boundQ} and \ref{thm:final}, and Lemma \ref{ipnormQ}.
\epf

\noi We make the following remarks:
\bit
\item In stating Part (ii) of Theorem \ref{thm:final1}, we have specified  that $\Omega$ is a convex polyhedron, since the assumption needed for this part, 
Assumption \ref{ass:reg}, holds in this case by  Lemma \ref{lem:reg1}. 
\item 
Observe that, just as in the Helmholtz theory in \cite{GrSpVa:17}, the condition on the coarse mesh diameter $\Hcs$ in \eqref{eq:E20rpt} is more stringent than the condition on the subdomain diameter $\Hsub$; one finds similar discrepancies in criteria in domain decomposition theory for coercive elliptic PDEs; see, e.g., \cite{GrLeSc:07}.
\item The condition on $\Hcs$ in \eqref{eq:E20rpt} is better than the analogous condition obtained for the Helmholtz equation in \cite[Theorem 5.6]{GrSpVa:17}. This is because \cite{GrSpVa:17} considered the interior impedance problem for the Helmholtz equation with absorption, and the contribution from the impedance boundary led to extra restrictions on $\Hcs$.
\eit

\noi Combining Theorem \ref{thm:final1} with the result about GMRES convergence in Corollary \ref{cor:Elman}, we obtain 
\begin{theorem}\textbf{\emph{(GMRES convergence for left preconditioning)}}\label{cor:final2}
Let $\Omega$ be a convex polyhedron. Consider the weighted GMRES method applied to $B_{\abs, AS}^{-1} A_\abs$, where the residual is minimised in
the norm induced by $D_\wn$ (as described in \S\ref{sec:52}).

Given $\wn_0>0$, there exists $\cC>0$, independent of all parameters (i.e.~$\wn, \delta, \Hcs,$ and $\Hsub$), such that, 
given $0<a<1$, if (i) $\wn\geq \wn_0$, (ii) Condition \eqref{eq:E20rpt} holds,
and (iii)
\beq\label{eq:iterations}
m \geq \cC \,
\Cstab
\kseb^3 \bigg(1 + \left(\frac{\Hcs}{\delta}\right)^2\bigg)\log\left(\frac{12}{a}\right),
\quad\text{ then } \quad
\frac{\|\br_m\|_{D_k}}{\|\br_0\|_{D_k}}\leq a.
\eeq
\end{theorem}

\bre[Cases when Theorem \ref{cor:final2} implies optimal GMRES convergence]\label{rem:TheoryCases}
When $\vert \abs\vert  \sim \wn^2$ and $\delta \sim \Hcs$, Condition \eqref{eq:E20rpt} is satisfied when $\Hsub\sim \Hcs \sim \wn^{-1}$.  
\es{Similarly, when $\vert \abs\vert  \sim \wn^2$ and $\delta \sim h$, Condition \eqref{eq:E20rpt} is satisfied when $h\sim \Hsub\sim \Hcs \sim \wn^{-1}$. In both these cases} 
the bound \eqref{eq:iterations} implies GMRES will converge with the number of iterations independent of all these parameters;
these results are illustrated in   \S\ref{sec:num}. 
\ere

When $\vert\abs\vert \ll \wn^2$, Theorem \ref{cor:final2} is not sharp. Indeed, when $\vert \abs\vert  \sim \wn$ and $\delta \sim \Hcs$, Condition \eqref{eq:E20rpt} is satisfied with the impractical mesh widths $\Hsub \sim \wn^{-2}$, $\Hcs\sim \wn^{-3}$, and the bound \eqref{eq:iterations} implies GMRES will converge with the number of iterations growing at most like $\wn^3$ as $\wn\tendi$; in \S\ref{sec:num} we see that in fact the method does deteriorate badly when
$\xi$ is substantially less than $k^2$.

The results of Theorem \ref{cor:final2} are better than the analogous results for the Helmholtz equation obtained in \cite[Corollary 5.7]{GrSpVa:17}; this is because we used the improvement \eqref{eq:Elman2} on the Elman estimate due to \cite{BeGoTy:06}, whereas \cite{GrSpVa:17} only used the original Elman estimate \eqref{eq:Elman}.

\bpf[Proof of Theorem \ref{cor:final2}]
With $\beta$ defined by \eqref{eq:cosbeta} with $C=B_{\abs, AS}^{-1} A_\abs$, Theorem \ref{thm:final1} implies that $\cos\beta \gtrsim \widetilde\epsilon(k,\abs,\delta,\Hcs)$ with
\beqs
\widetilde\epsilon(k, \abs, \delta, \Hcs):=
\bigg(1 + \left(\frac{\Hcs}{\delta}\right)^2\bigg)^{-1} 
\left(\frac{\vert \abs\vert }{\wn^2} \right)^{3}.
\eeqs
If $|\abs|\sim k^2$ and $\delta\sim \Hcs$, then $\widetilde\epsilon(k,\abs, \delta,\Hcs)$ is bounded below by a positive constant, but if either $|\xi|\ll k^2$ or $\delta \ll \Hcs$, then $\widetilde\epsilon(k,\abs, \delta,\Hcs)$ can approach zero.
The closer $\cos\beta$ is to zero, the worse the estimate of Theorem \ref{thm:Elman} is; the worse-case scenario is when $\widetilde\epsilon \tendo$ and $\cos \beta\sim \widetilde\epsilon$. Since $\cos\beta=\sin(\pi/2-\beta) = (\pi/2-\beta)(1+o(1))$, it is sufficient to consider the case when $\beta= \pi/2-\epsilon$ with $\epsilon\tendo$ and  $\epsilon\sim \widetilde\epsilon$. Applying Corollary \ref{cor:Elman}, we obtain the bound on $m$ \eqref{eq:iterations} and the proof is complete.
\epf

Using coercivity of the adjoint
form in Lemma \ref{lem:coer}, we can obtain the following  result about right
preconditioning, however in the inner product induced by
$D_\wn^{-1}$. From this, the analogue of Theorem
\ref{cor:final2}, with $D_\wn$ replaced by $D^{-1}_k$, follows.

\begin{theorem}\textbf{\emph{(Bounds on the field of values for right preconditioning)}}\label{thm:final3}

\noi (i)
\vspace{-3ex}
\begin{align*}
\Vert  A_\abs B_{\abs, AS}^{-1}  \Vert_{D_\wn^{-1}}  & \ \lesssim  \ \left(\ksqeps\right) \quad \text{for all} \,\, \Hcs, \Hsub.
\end{align*}
(ii) With the same assumptions as Part (ii) of Theorem \ref{thm:final1}, given $\wn_0>0$ and provided Condition \eqref{eq:E20rpt} holds, 
 \begin{align*}
\quad\quad
\frac{\Big\vert \Big\langle \bV, 
A_\abs B_{\abs, AS}^{-1} \bV \Big\rangle_{D_\wn^{-1}} \Big\vert}
{\Vert \bV \Vert_{D_\wn^{-1}}^2 }  & \ \gtrsim   \ 
\bigg(1 + \left(\frac{\Hcs}{\delta}\right)^2\bigg)^{-1} 
\left(\frac{|\abs|}{\wn^2} \right)^{2},
\end{align*}
for all $\bV\in \Com^n$ and for all $\wn\geq \wn_0$.
\end{theorem}

\begin{proof}
A straightforward calculation shows that for all $\bV \in \bC^n$ and with $\bW =
D_\wn^{-1} \bV$, we have
$$
\frac{\vert \langle \bV, A_\abs B_{\abs,AS}^{-1} \bV \rangle_{D_\wn^{-1}}\vert} 
{  \langle \bV , \bV\rangle_{D_\wn^{-1}}}\  =\ 
\frac{\vert \langle  (B^*_{\abs,AS})^{-1} A_\abs^* \bW, \bW \rangle_{D_\wn}\vert} 
{  \langle \bW , \bW\rangle_{D_\wn}}  \  =\ 
\frac{\vert \langle \bW,  (B^*_{\abs,AS})^{-1} A_\abs^* \bW\rangle_{D_\wn}\vert} 
{  \langle \bW , \bW\rangle_{D_\wn}} , $$
where    $A^*$, $(B_{\abs,AS}^*)^{-1}$ are the Hermitian transposes of $A,
B_{\abs,AS}^{-1}$ respectively.  
The coercivity of the adjoint form proved in Lemma \ref{lem:coer} then ensures that the estimate in Theorem 
\ref{thm:final1} (ii) also holds for the adjoint matrix and the result
(ii) then follows.  The result (i) is obtained analogously from taking
the adjoint and using Theorem  \ref{thm:final1} (i).
\end{proof} 

\begin{theorem}\textbf{\emph{(GMRES convergence for right preconditioning)}}\label{cor:final4}
The result of Theorem \ref{cor:final2} 
holds when left preconditioning is replaced by right preconditioning (i.e.~$B_{\abs, AS}^{-1} A_\abs$ is replaced by $ A_\abs B_{\abs, AS}^{-1}$ in the statement of the theorem).
\end{theorem}

\section{Numerical experiments}\label{sec:num}

\igg{We now give details of the numerical experiments,  previously  summarised in \S\ref{sec:contributions}. 
  After  some technical preliminaries  in  \S\ref{sec:6.1},   \S\ref{sec:6.2}   illustrates in detail the
  theory of the paper. We use  the Additive Schwarz preconditioner  $B^{-1}_{\xiprec}$ to precondition $A_{\xiprob}$ and we  study various choices of   $\xiprob   =    \xiprec$,  as well as the  dependence  on the  overlap $\delta$. We also see
  that weighted GMRES performs almost identically to standard GMRES.   These experiments motivate various extensions of the additive Schwarz method and in  \S\ref{sec:6.3} we present two tables illustrating
   these, including solving the propagative case $\xiprob = 0$.
  To highlight the importance of absorptive problems in practice,  \S\ref{subsec:medimax}  presents results for a
 heterogeneous practical example where $0\leq \xiprob({\bx}) \lesssim \wn^2$. In \S\ref{subsec:cobra} we solve the 
  well-known  COBRA cavity test case, with $\xiprob=0$, $\xiprec=k$. This shows that adding the absorption into  the preconditioner (as in the theory) allows good parallel scaling.} 

\subsection{Details of the computations and notation}\label{sec:6.1}

All the computations are done in FreeFem++, an open source domain specific language (DSL) specialised for solving BVPs with variational 
methods (\url{http://www.freefem.org/ff++/}). The code was parallelised and run on the French supercomputers Curie   (TGCC) and   OCCIGEN (CINES).  
The discretization is by \Nedelec edge finite elements as in \S\ref{sec:dis} on a mesh of tetrahedra of the domain $\Omega$. We will give results for both order $1$ and order $2$ finite elements. 
The right preconditioned linear system is solved with GMRES without restarts.  

In the theory we assumed that each coarse mesh tetrahedron is a union of fine mesh tetrahedra, but in some numerical experiments we consider non-nested fine/coarse meshes. Nevertheless, if the fine and coarse mesh are not nested, one has to be careful when computing the ``restriction'' matrix~\eqref{eq:restriction} for edge finite elements. 
Indeed, if a fine mesh tetrahedron crosses a coarse mesh tetrahedron, then the computation of matrix~\eqref{eq:restriction} requires the integration of \emph{non-smooth} functions $\bw^H_{p}$ along the supports (edges, faces) of the degrees of freedom on the fine mesh. Hence for these integrals we use a Gaussian quadrature rule with more points than that normally used when the meshes are nested (this is described further below).

In our implementation of GMRES, we use a random initial guess, aiming to ensure  
that all frequencies are present in the error.  The stopping criterion  
is based on a reduction of  the relative residual by  $10^{-6}$ and the 
maximum number of iterations allowed is $200$. 
\mb{Apart from in  \S \ref{sec:6.2}, 
  throughout the paper we use standard GMRES.} 
To apply the preconditioner, the local problems in each subdomain and 
the coarse space problem are each  solved with a direct solver 
(in this case MUMPS \cite{amestoy:2001:fully}).

In the experiments in \S\ref{sec:6.2}, \S\ref{sec:6.3}, we solve the PDE \eqref{eq:2} 
in the unit cube  $\Omega = (0,1)^3$.  
The right-hand side function $\mathbf{F}$ in \eqref{eq:2} is given by the point source 
\[\mathbf{F}=[f,f,f], \quad \text{where} \quad f = -\exp(-400((x-{0.5})^2+(y-{0.5})^2+(z-{0.5})^2)).\]

The total number of degrees of freedom in the fine grid 
problem is denoted by $n$.
The fine mesh diameter, $h$, is either chosen as $h \sim k^{-3/2}$ (generally  believed to remove the pollution effect, by analogy with Helmholtz problems; see, e.g., the review in \cite[Pages 182--183]{GrLoMeSp:15}) or with a fixed number $g$ of  grid-points per wavelength, i.e.~$h \neweq  2 \pi / (g  \wn)$, where the notation $\neweq$ means that $h$ is chosen as close to $2\pi/ (g\wn)$ as possible, whilst still ensuring that we have a mesh on the domain. We use the following parameters to describe the {domain decomposition method}: 
\begin{equation*}
\left\{ 
\begin{array}{ll}
\Hsub = \text{maximum subdomain diameter {(\emph{without} considering the overlap),}}  \\
\Hcs = \text{maximum coarse mesh  diameter,} \\
{N} = \Nsub = \text{number of subdomains,}\\
\ncs = \text{number of  coarse mesh degrees of freedom.} 
\end{array} 
\right.   
\end{equation*}
We define the parameters $(\alpha, \alpha')$ such that 
\begin{equation} 
\label{eq:alphas}
\Hsub \sim \wn^{-\alpha} \quad \text{and} \quad \Hcs \sim \wn^{-\alpha'}.  
\end{equation} 
 In the tables we use $\#$ to denote the number of iterations for any given method. 
 For some experiments we also report the computation times in seconds. 
To set up the preconditioners, we start by subdividing the domain $\Omega$ into $N$ non-overlapping simply-connected subdomains  
 $\{\Omega_\ell^0: \ell = 1, \ldots, N\}$,  with the property that each 
subdomain  is a union of fine mesh tetrahedra.   This could be done by hand (if the fine mesh is highly structured); 
or more generally it could 
be done by applying  graph partitioning software (e.g.~METIS~\cite{Ka:98:metis}) to 
the graph determined by the elements of the   fine mesh.  (Both methods are used below.) We then introduce  
extended overlapping subdomains  $\Omega_\ell^p$, $p = 1,2, \ldots$,  defined recursively by requiring    $\Omega_\ell^{p+1}$ 
to be the union of $\Omega_\ell^{p}$ with all the fine mesh tetrahedra  touching it.
As $p$ increases so does the overlap. We call the case $p=2$ \emph{minimal overlap} ($\delta \sim h$).
The term {\em generous overlap} refers to the case $\delta \sim \Hsub$.

\subsection{Experiments illustrating the theory}
\label{sec:6.2}

\igg{Here {we} solve the system arising from the PDE \eqref{eq:2},  
with the PEC  boundary condition \eqref{eq:PECcond}.
In the first  experiment (Table \ref{tab:1}) we set $\xiprob =\xiprec=k^2$ and use generous overlap ($\delta \sim \Hsub$). {We choose $h \sim \wn^{-3/2}$ and} we study the effect of varying   $\alpha$ and $ \alpha'$ {in \eqref{eq:alphas}}. 
The theory tells us that in this case the two-level Additive Schwarz preconditioner~\eqref{eq:defAS} will be robust for $\alpha = \alpha' = 1$ (cf.~Remark~\ref{rem:TheoryCases}, first case). 
\mb{Table \ref{tab:1} gives the iteration numbers for the two-level AS preconditioner, with in parentheses the corresponding one-level preconditioner, which is defined by \eqref{eq:defAS} but omitting the term involving $(A_{\abs}^0)^{-1}$}.
The choice $\alpha = \alpha'$ means  that the subdomains and the coarse mesh are nested. As we decrease $\alpha$, {resp.~$\alpha'$}, we do more work {on each subdomain} {(because subdomains become fewer on a given fine mesh problem)}, resp.~less work in the coarse grid. We study the effect of reducing $\alpha= \alpha'$ in the second and third columns of Table \ref{tab:1}. 
In the third and fourth columns we see the effect of doing more work on the coarse grid and keeping the subdomains fixed.   
Because our parallel code solves each subdomain problem on an individual  
core, which is the natural and efficient implementation of domain decomposition methods, and $\Hsub \sim \wn^{-\alpha}$, a large number of cores is required when $\alpha$ is close to   $1$ {(e.g.~for $\alpha=1$, $\wn=25$, this would require $15625$ cores)}.   Thus, smaller values of $\wn$ are presented in the first column of 
Table \ref{tab:1}.   In the other columns we can go higher with $\wn$ since there are fewer subdomains, but for $k=30$ the subdomains, with generous overlap, become too big and memory issues appear.} 
\igg{Note that, in general, if $\alpha \ne \alpha'$ then the fine and coarse mesh are not nested, and so, recalling the discussion in \S\ref{sec:6.1}, we need a more accurate quadrature formula when computing matrix~\eqref{eq:restriction}: here we used $12$ Gaussian quadrature points on the edge.
The method appears to be robust for all the considered choices of $\alpha$ and $\alpha'$, and the coarse grid is less useful for smaller $\alpha$. }

\begin{table}
\mb{
\begin{tabular}{|c||c|c|c|c|}
\hline 
$\wn$ & $\alpha =1 = \alpha'$& $\alpha =0.9 = \alpha'$ & $\alpha =0.8 = \alpha'$ & $\alpha =0.8, \alpha' = 1$ \\
\hline
10 & 53(57) & 40(42) & 37(38) & 37(38) \\
15 & 59(63) & 46(48) & 37(37) & 37(37) \\
20 & 63(67) & 46(47) & 37(37) & 36(37) \\
25 & - & 48(50) & 38(38) & 38(38) \\
\hline
\end{tabular}
\caption{PEC  boundary condition on $\partial \Omega$, $\xiprob = \xiprec = \wn^2$, $\delta \sim \Hsub$, $h \sim \wn^{-3/2}$: 
iteration numbers for the two-level(one-level) AS preconditioner.    
} \label{tab:1}
}
\end{table}

\begin{table}
\mb{
\begin{tabular}{|c||c|c|c|c|c|}
\hline 
$\wn$ & $\alpha =1 = \alpha'$& $\alpha =0.9 = \alpha'$ & $\alpha =0.8 = \alpha'$ & $\alpha =0.8, \alpha' = 1$ & $\alpha =0.6, \alpha' = 1$\\
\hline
10 & 76(96) & 55(61) & 52(57) & 47(57) & 40(48)\\
15 & 82(105) & 59(73) & 57(65) & 51(65) & 44(55)\\
20 & 58(90) & 61(80) & 60(67) & 52(67) & 44(58)\\
25 & - & 61(89) & 64(75) & 54(75) & 44(60) \\
30 & - & - & 66(80) & 53(80) & 44(60) \\
\hline
\end{tabular}   
\caption{Repeat of Table \ref{tab:1} but  with $\delta \sim h $.}
\label{tab:minOverlap}
}
\end{table}

\begin{table}
\mb{
\begin{tabular}{|c||c|c|c|c|c|}
\hline 
$\wn$ & $\alpha =1 = \alpha'$& $\alpha =0.9 = \alpha'$ & $\alpha =0.8 = \alpha'$ & $\alpha =0.8, \alpha' = 1$ \\
\hline
10 & $>$200($>$200) & 134(145) & 96(116) & 92(116) \\ 
15 & * & $>$200($>$200) & 190($>$200) & 175($>$200) \\ 
20 & * & * & * & * \\
25 & - & * & * & * \\
\hline
\end{tabular}
\caption{Repeat of Table \ref{tab:1} but  with $\xiprob = \xiprec = \wn$.} 
\label{tab:lowAbs}
}
\end{table}

\begin{table}
\mb{
\begin{tabular}{|c||c|c|c|c|}
\hline 
$\wn$ & $\alpha =1 = \alpha'$& $\alpha =0.9 = \alpha'$ & $\alpha =0.8 = \alpha'$ & $\alpha =0.8, \alpha' = 1$ \\
\hline
10 & 55(59) & 41(43) & 38(39) & 39(39) \\ 
15 & 61(65) & 47(49) & 38(39) & 38(39) \\ 
20 & 65(70) & 47(49) & 38(38) & 38(38) \\ 
25 & - &  50(52) & 39(40) & 39(40) \\ 
\hline
\end{tabular}
\caption{Repeat of Table \ref{tab:1} but  using weighted  instead of standard GMRES.} \label{tab:weighted}
}
\end{table} 

\igg{Since generous overlap requires too much communication in parallel implementations, we now investigate the performance of (more parallel efficient) minimal overlap methods. 
Table \ref{tab:minOverlap} shows the results for the cases considered in 
 Table \ref{tab:1}, but with $\delta \sim h$; moreover, Table \ref{tab:minOverlap} shows also the case with $\alpha=0.6$, $\alpha'=1$, since with  $\delta \sim h$ subdomains remain small enough for the local problems to be solved on a single core.  
The number of iterations increases, but the method still performs well, with  the number of iterations growing mildly with $\wn$; for $\alpha=0.6$, $\alpha'=1$, robustness appears to be fully achieved.}     
\igg{We also observe that the coarse grid solve has a larger effect on the iteration count  
compared to the generous overlap case. }

\mb{When solving systems arising from the propagative Maxwell equations $\xiprob = 0$ we know by analogy with 
the Helmholtz case \cite{GrSpVa:17} that the preconditioner should be based on a smaller $\xiprec$, typically $\xiprec = k$ (or less) is a good  choice. In Table \ref{tab:lowAbs} we repeat Table \ref{tab:1} with {$\xiprob = \xiprec = \wn$}. 
This table shows that we cannot expect a method which uses PEC boundary conditions on subdomains to provide a robust preconditioner for the propagative case. A similar observation for Helmholtz problems is given in \cite{GrSpVa:17}.
}

\mb{We recall that the theory  above is for left preconditioning in the norm induced by $D_k$, whereas in  Tables \ref{tab:1}--\ref{tab:lowAbs} above we have used standard GMRES. For comparison in Table \ref{tab:weighted} we repeat Table~\ref{tab:1} using weighted GMRES  with weight $D_k$. As we see this makes little difference. 
Similar observations for Helmholtz problems were made in \cite{GrSpVa:17}.
}

\begin{table} 
\mb{
\begin{tabular}{|c|c|c|c||c|c|c|c|}  
\hline
\multicolumn{4}{|c||}{} & \multicolumn{2}{c|}{$\Hsub = 2/(5 f)$} &  \multicolumn{2}{c|}{$\Hsub = 1/(2 f)$} \\
\hline 
$f$ & $\wn$ & $n$ & $\ncs$ & $N_\text{sub}$ & \#AS  &   $N_\text{sub}$ & \#AS\\
\hline 
2 & 12.6 & 2.5 $\times 10^6$ & 2.9 $\times 10^3$ & 125   & 38(76)   & 64     & 39(72)\\ 
3 & 18.8 & 8.3 $\times 10^6$ & 9.3 $\times 10^3$ &          &               & 216  & 39(101)\\
4 & 25.1 & 2.0 $\times 10^7$ & 2.1 $\times 10^4$ & 1000 & 39(137) & 512   & 39(129)\\ 
5 & 31.4 & 3.8 $\times 10^7$ & 4.1 $\times 10^4$ & 		& 		& 1000 & 39(157)\\ 
6 & 37.7 & 6.6 $\times 10^7$ & 7.0 $\times 10^4$ & 3375 & 58(195) & 1728 & 57(183)\\ 
7 & 44.0 & 1.0 $\times 10^8$ & 1.1 $\times 10^5$ & & & 2744 & 60($>$200) \\ 
8 & 50.3 & 1.6 $\times 10^8$ & 1.6 $\times 10^5$ & 8000 & 59($>$200) & 4096 & 61($>$200)\\ 
\hline
\end{tabular}
\caption{PEC  boundary condition on $\partial \Omega$, $\xiprob = \xiprec = \wn^2$, $\delta \sim h$, $h = 1/(20 f)$, $\Hcs = 1/(2 f)$, where $f=2\pi \wn$:  results for the two-level(one-level) AS preconditioner.} 
\label{tab:theoryDeltamin}
}
\end{table}  

\mb{Finally, in Table~\ref{tab:theoryDeltamin} we illustrate the second case of Remark~\ref{rem:TheoryCases}, which proves  robustness even in the case of minimal overlap, when the fine grid and the coarse grid each have a fixed number of grid points per wavelength. 
Here we set $\xiprob = \xiprec = \wn^2$, $\delta \sim h$, and we use order 2 \Nedelec finite elements.
In order to have ``round'' numbers with respect to the wavelength and then actual cubes as subdomains, we vary the frequency $f$, with $\wn = 2\pi f$, and set $h = 1/(20 f)$, $\Hcs = 1/(4 f)$, $\Hsub = 2/(5 f)$ (with $f$ divisible by $2$) or $\Hsub = 1/(2 f)$.  
As expected from the theory, the two-level preconditioner performs very well, especially in comparison to the one-level version. 
}

\subsection{Alternative preconditioners} 
\label{sec:6.3}
In order to properly take into account the overlap between subdomains, where unknowns are repeated, the Additive Schwarz preconditioner~\eqref{eq:defAS} is classically modified using a discrete \emph{partition of unity}. More precisely, diagonal matrices $(D^\ell)_{\ell=1}^N$ are constructed such that 
\begin{equation}
\label{eq:algebPartUnity}
\sum_{\ell = 1}^{N} (R^\ell)^T D^\ell R^\ell = I
\end{equation}
and then the two-level \emph{Restricted Additive Schwarz} preconditioner is defined by
\begin{equation}
\label{eq:RAS2}
\matrixB_{\abs, \text{RAS},2}^{-1}: =  \sum_{\ell=1}^N (R^\ell)^T D^\ell (A_{\abs}^\ell)^{-1} R^\ell +(R^0)^T ({A_{\abs}^0})^{-1} R^0.
\end{equation}
Note that the construction of the partition of unity matrices $D^\ell$ is intricate, especially for edge finite elements: here we follow the construction in \cite{paperPP:2016}. 
In \eqref{eq:RAS2} the coarse correction matrix $G := (R^0)^T ({A_{\abs}^0})^{-1} R^0$ is combined in an additive way with \mb{the \emph{one-level} {Restricted Additive Schwarz} preconditioner}
\begin{equation}
\mb{\matrixB_{\abs, \text{RAS},1}^{-1}: =  \sum_{\ell=1}^N (R^\ell)^T D^\ell (A_{\abs}^\ell)^{-1} R^\ell.}
\end{equation}
The two-level \emph{hybrid} version of RAS is obtained using instead the BNN (Balancing Neumann Neumann) coarse correction formula:
\begin{equation}
\label{eq:HRAS}
B^{-1}_{\abs,\text{HRAS}}:= (I - G A_{\abs}) \matrixB_{\abs, \text{RAS},1}^{-1} (I - A_{\abs} G) + G,
\end{equation}
\mb{while the \emph{Adapted Deflation (Variant 1)} coarse correction formula gives this other two-level version of RAS (for more details on  \eqref{eq:RAS2}--\eqref{eq:ADEF1ras} and other related methods, see \cite{Tang2009}):}
\begin{equation}
\label{eq:ADEF1ras}
\mb{B^{-1}_{\abs,\text{ADEF1-RAS}}:= \matrixB_{\abs, \text{RAS},1}^{-1} (I - A_{\abs} G) + G.}
\end{equation}
Moreover, as discussed in the introduction, better boundary conditions can be used on the subdomains: \mb{$B_{\abs,\text{ImpRAS},1}^{-1}$ is defined like  $B_{\abs,\text{RAS},1}^{-1}$,} but the solves with $A_{\abs}^\ell$, $\ell=1,\dots,N$, are replaced by solves with matrices corresponding to the Maxwell problem on $\Omega_\ell$ with homogeneous \emph{impedance} boundary condition on $\partial \Omega_\ell \backslash \po$.  
\mb{As in~\eqref{eq:HRAS}, resp.~\eqref{eq:ADEF1ras}, one can define the two-level hybrid version $B^{-1}_{\abs,\text{ImpHRAS}}$, resp.~the Adapted Deflation version $B^{-1}_{\abs,\text{ADEF1-ImpRAS}}$ of ImpRAS.}

 \begin{table}
\mb{
 \begin{tabular}{|c|c|c|c||c|c|c||c|c|c|}
 \hline 
 \multicolumn{4}{|c||}{} & \multicolumn{3}{c||}{$\delta \sim \Hsub$} & \multicolumn{3}{c|}{ $\delta \sim h$} \\ 
 \hline 
 $\wn$ &  $n$ & $\Nsub$ & $\ncs$ & $\#$RAS & $\#$HRAS &  $\#$ADEF1 & $\#$RAS & $\#$HRAS &  $\#$ADEF1\\
\hline 
10 & 3.4 $\times 10^5$ & 216   &  7.9 $\times 10^3$  & 20(21) & 17 & 15 & 27(39) & 19 & 19 \\ 
15 & 1.9 $\times 10^6$ & 512   &  2.6 $\times 10^4$  & 20(22) & 17 & 15 & 31(46) & 20 & 20 \\ 
20 & 5.2 $\times 10^6$ & 1000 &  6.0 $\times 10^4$  & 20(22) & 17 & 15 & 31(49) & 20 & 20 \\  
25 & 1.5 $\times 10^7$ & 2197 &  1.2 $\times 10^5$  & 22(24) & 18 & 16 & 34(55) & 22 & 22 \\ 
30 & 4.1 $\times 10^7$ & 3375 &  2.0 $\times 10^5$  &  - &  - &  -  & 33(59) & 21 &  21 \\ 
\hline
\end{tabular}  
 \caption{PEC  boundary condition on $\partial \Omega$, $\xiprob = \xiprec = \wn^2$, $h \sim \wn^{-3/2}$, $\alpha = 0.8$, $\alpha' = 1$: 
results for the two-level(one-level) RAS preconditioners.}
 \label{tab:rasHras}  
}
 \end{table} 
\mb{First, in Table~\ref{tab:rasHras} we repeat the same experiment of Tables~\ref{tab:1}--\ref{tab:minOverlap} using the RAS, HRAS, ADEF1-RAS preconditioners, that is we consider PEC  boundary condition on $\partial \Omega$, $\xiprob = \xiprec = \wn^2$, $h \sim \wn^{-3/2}$, $\delta \sim \Hsub$ or $\delta \sim h$. 
We set $\Hsub$ and $\Hcs$ as in \eqref{eq:alphas}, with $\alpha=0.8$, $\alpha'=1$.  }
\mb{Comparing Table~\ref{tab:rasHras} with Table~\ref{tab:1}, we see that RAS is superior to AS as expected; HRAS and ADEF1-RAS are superior to both additive methods.} 

In the next experiments we solve the problem with $\xiprob=0$ and impedance boundary condition \eqref{eq:imp} on $\partial \Omega$. We use the preconditioner ImpHRAS and minimal overlap.
\begin{table}
\begin{center}
\begin{tabular}{|c|c|c||c|c|c|c|c|}
\hline
\multicolumn{3}{|c||}{} & \multicolumn{5}{c|}{$\alpha=0.6$, $\alpha'=0.9$}\tabularnewline
\hline 
$\wn$ & $n$ & $N_\text{sub}$ & \#2-level & $\ncs$ & Time & \#1-level & Time\tabularnewline
\hline 
10 & $2.6 \times 10^5$ & 27 & 20 & $2.9 \times 10^3$ & 17.1(1.8) & 37 & 13.7(2.6)\tabularnewline
15 & $1.5 \times 10^6$ & 125 & 26 & $1.0\times 10^4$ & 25.4(3.9) & 71 & 26.5(9.1) \tabularnewline
20 & $5.2 \times 10^6$ & 216 & 29 & $2.1\times 10^4$ & 53.0(9.0) & 95 & 60.8(25.9)\tabularnewline
25 & $1.4\times 10^7$ & 216 & 33 & $4.4\times 10^4$ & 145.0(29.6) & 107 & 189.3(90.5) \tabularnewline 
30 & $3.3\times 10^7$ & 343 & 39 & $6.9\times 10^4$ & 387.9(132.7) & 134 & 669.4(444.7)\tabularnewline 
\hline
\multicolumn{3}{|c||}{} & \multicolumn{5}{c|}{$\alpha=0.7$, $\alpha'=0.8$}\tabularnewline
\hline 
$\wn$ & $n$ & $N_\text{sub}$ & \#2-level & $\ncs$ & Time & \#1-level & Time\tabularnewline
\hline 
10 & $3.1\times 10^5$ & 125 & 28 & $1.9\times 10^3$ & 8.2(1.2) & 58 & 7.7(2.0)\tabularnewline
15 & $1.5\times 10^6$ & 216 & 39 & $4.2\times 10^3$ & 19.3(3.7) & 82 & 19.7(7.2) \tabularnewline
20 & $6.3\times 10^6$ & 512 & 58 & $7.9\times 10^3$ & 42.6(10.0) & 124 & 49.1(21.0) \tabularnewline
25 & $1.4\times 10^7$ & 729 & 60 & $1.7\times 10^4$ & 75.1(17.8) & 150 & 98.2(43.0)\tabularnewline 
30 & $3.5\times 10^7$ & 1000 & 81 & $2.6\times 10^4$ & 223.5(86.7) & 181 & 320.7(199.0)\tabularnewline 
\hline
\end{tabular}
\caption{ ImpHRAS for  \Nedelec order 1 elements,  $h \sim \wn^{-3/2}$, impedance boundary condition on $\partial \Omega$, $\xiprob = 0$, $\xiprec = 0$, $\delta \sim h$. In the Time columns we report the total time (the execution time for GMRES) in seconds. 
} 
\label{tab:EpsPrec0}
\end{center}
\end{table}
{In Table~\ref{tab:EpsPrec0} we still discretize the problem with order 1 edge elements and take $h \sim \wn^{-3/2}$, $\Hsub$ and $\Hcs$ as in \eqref{eq:alphas}, for two choices of $\alpha, \alpha'$.}  
These methods are close to being load balanced in the sense that 
the coarse grid and subdomain problem sizes are  very similar  
when $\alpha + \alpha' = 3/2$. 
As in \S\ref{sec:6.2} we recall that if $\alpha \ne \alpha'$ the fine and coarse mesh are not nested, and again we used $12$ Gaussian quadrature points on the edge when computing the matrix \eqref{eq:restriction}. 
Out of the methods tested, the 2-level method with $(\alpha, \alpha')  = (0.6, 0.9)$ gives the best iteration count, but is more expensive. 
The method $(\alpha, \alpha')  = (0.7, 0.8)$  is faster in time  
but its iteration count grows more quickly, so its   advantage will  
diminish as $\wn$ increases further. 
For $(\alpha, \alpha') = (0.6,0.9)$ the coarse grid size grows with 
$\mathcal{O}(n^{{0.64}}  )$ while the time grows with $\mathcal{O}(n^{{0.65}}  )$. 
 For $(\alpha, \alpha') = (0.7,0.8)$ the rates are  $\mathcal{O}(n^{{0.54}}  )$  and $\mathcal{O}(n^{{0.69}})$.
\mb{Note that here we switch off the absorption also in the preconditioner, i.e., we choose $\xiprec = 0$. 
The iteration counts are almost identical to the ones with $\xiprec=\wn$, reported in the preliminary paper \cite[Table 4]{BoDoGrSpTo:2017:DD24Max}.} 
This observation raises very interesting open theoretical questions regarding the performance of methods based on local impedance solves without using absorption; these questions are also discussed in \cite{GrSpZo:17}. 
Despite absorption having little effect here, it has a beneficial effect if the coarse grid problem is solved by an inner iteration - see our discussion of the COBRA cavity problem in \S\ref{subsec:cobra}.
 
\begin{table}
\begin{center}
\begin{tabular}{|c|c|c||c|c|c|c|c|}  
\hline
\multicolumn{3}{|c||}{} & \multicolumn{5}{c|}{$g=20$, $\alpha=0.6$, $g_\text{cs} = 2$}\tabularnewline
\hline 
$\wn$ & $n$ & $N_\text{sub}$ & \#2-level & $\ncs$ & Time & \#1-level & Time\tabularnewline 
\hline 
10 &  1.1 $\times 10^6$ & 343 & 53 &1.3 $\times 10^3$ &27.6(6.9) & 116 & 29.5(14.1)\\ 
20 &  8.3 $\times 10^6$& 1728 & 55 & 9.3 $\times 10^3$ & 44.8(12.1)& $>$ 200 & 73.8(47.5)\\ 
30 &  2.8 $\times 10^7$& 3375 & 52 &3.0  $\times 10^4$&83.4(21.7)& - &-  \\ 
40 & 6.6 $\times 10^7$& 5832 & 59 & 7.0 $\times 10^4$ & 164.9(45.2) & - & - \\ 
\hline
\multicolumn{3}{|c||}{} & \multicolumn{5}{c|}{$g=20$, $\alpha=0.5$, $g_\text{cs} = 2$}\tabularnewline
\hline 
$\wn$ & $n$ & $N_\text{sub}$ & \#2-level & $\ncs$ & Time & \#1-level & Time\tabularnewline 
\hline 
10 &  1.1 $\times 10^6$& 125 & 38 & 1.3 $\times 10^3$&37.7(7.5) &80 & 36.6(14.8)\\ 
20 & 8.3 $\times 10^6$& 343 &36 & 9.3 $\times 10^3$&85.8(18.9) &123 & 161.7(72.1)\\ 
30  & 2.8 $\times 10^7$& 729 & 41 &3.0  $\times 10^4$& 155.7(39.8)& 162& 267.3(174.5)\\ 
40 &  6.6 $\times 10^7$& 1331 & 51 & 7.0 $\times 10^4$ & 272.3(77.6) & $>$ 200& 453.8(305.2)\\ 
\hline
\end{tabular}
\caption{ ImpHRAS for  \Nedelec order 2 elements,  $h \neweq 2\pi/ (g \wn)$, impedance boundary condition on $\partial \Omega$, $\xiprob = 0$, $\xiprec = \wn$, $H_\text{sub}  \neweq (g\wn/(2\pi))^{-\alpha}$, $\Hcs  \neweq 2\pi/ (g_\text{cs} \wn)$, $\delta \sim h$, irregular subdomains built with METIS. In the Time columns we report the total time (the execution time for GMRES) in seconds. 
} 
\label{tab:6bis}
\end{center}
\end{table}
 
In Table \ref{tab:6bis} we repeat Table~\ref{tab:EpsPrec0} (with $\xiprec =\wn$) {but changing the discretization: we consider} \Nedelec order 2 elements, $h \neweq 2\pi/ (g \wn)$, $H_\text{sub}  \neweq (g\wn/(2\pi))^{-\alpha}$, $\Hcs  \neweq 2\pi/ (g_\text{cs} \wn)$, $g_\text{cs}=2 < g=20$ (i.e. 20 grid points per wavelength for the fine mesh and only 2 grid points per wavelength for the coarse mesh). Here the partition into subdomains is irregular and built with METIS. 
We can see that the two-level preconditioner combined with a high-order discretization and a fixed number of grid-points per wavelength performs even better than in the previous test cases, especially compared to the one-level preconditioner. The coarse grid size grows with $\mathcal{O}(n)$, but with a much smaller constant compared to the fine grid size, while the time grows with $\mathcal{O}(n^{{0.42}})$ for $\alpha = 0.6$ and with $\mathcal{O}(n^{{0.47}})$ for $\alpha = 0.5$.

 \subsection{A highly-heterogeneous practical example} 
 \label{subsec:medimax}

\begin{figure}
\centering
{\includegraphics[width=0.34\textwidth]{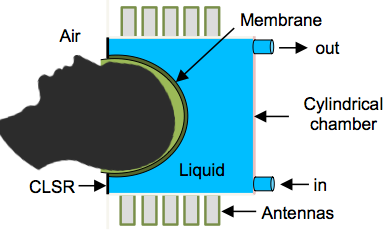}}
{\includegraphics[width=0.23\textwidth]{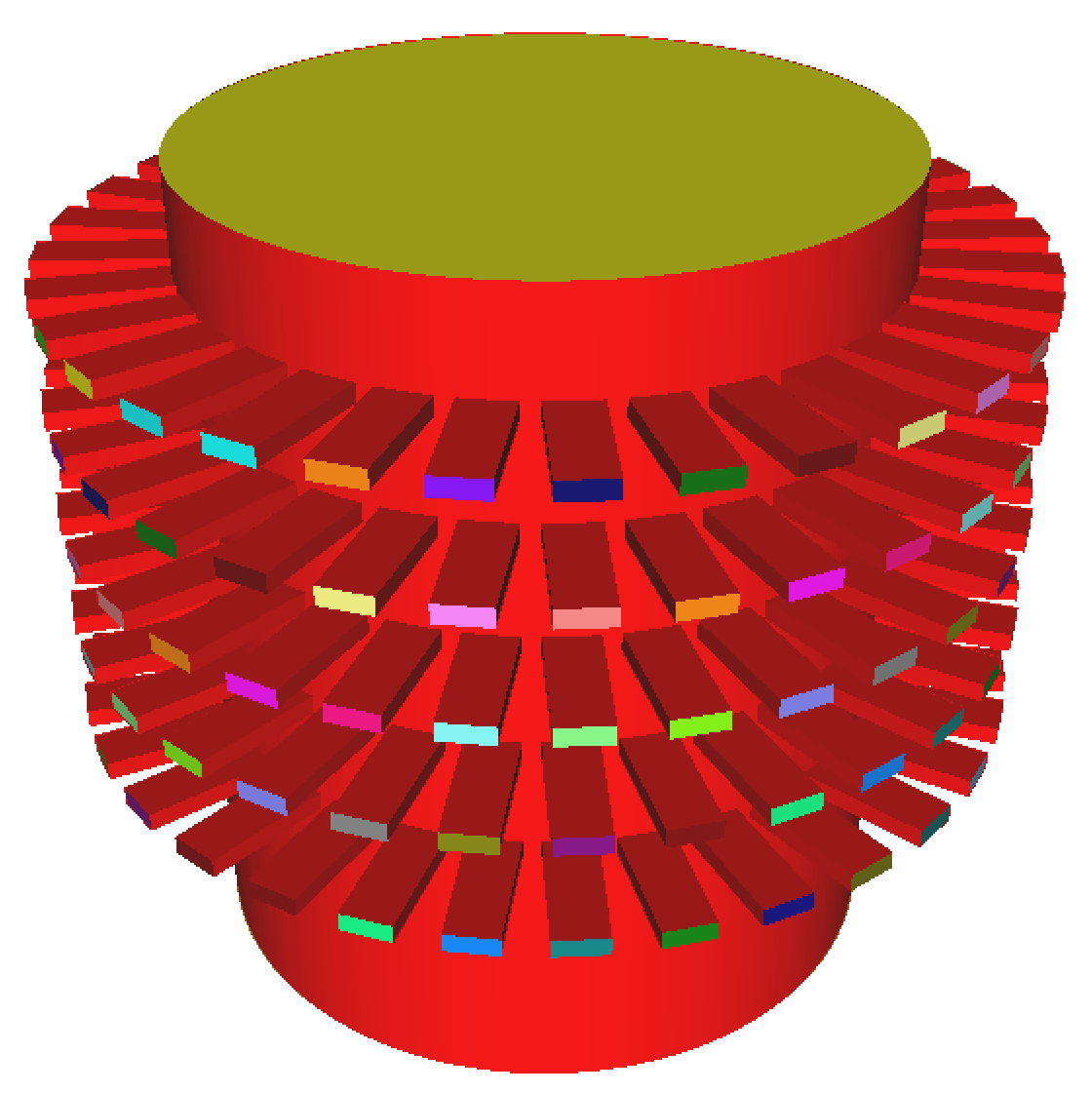}}
\includegraphics[width=0.41\textwidth]{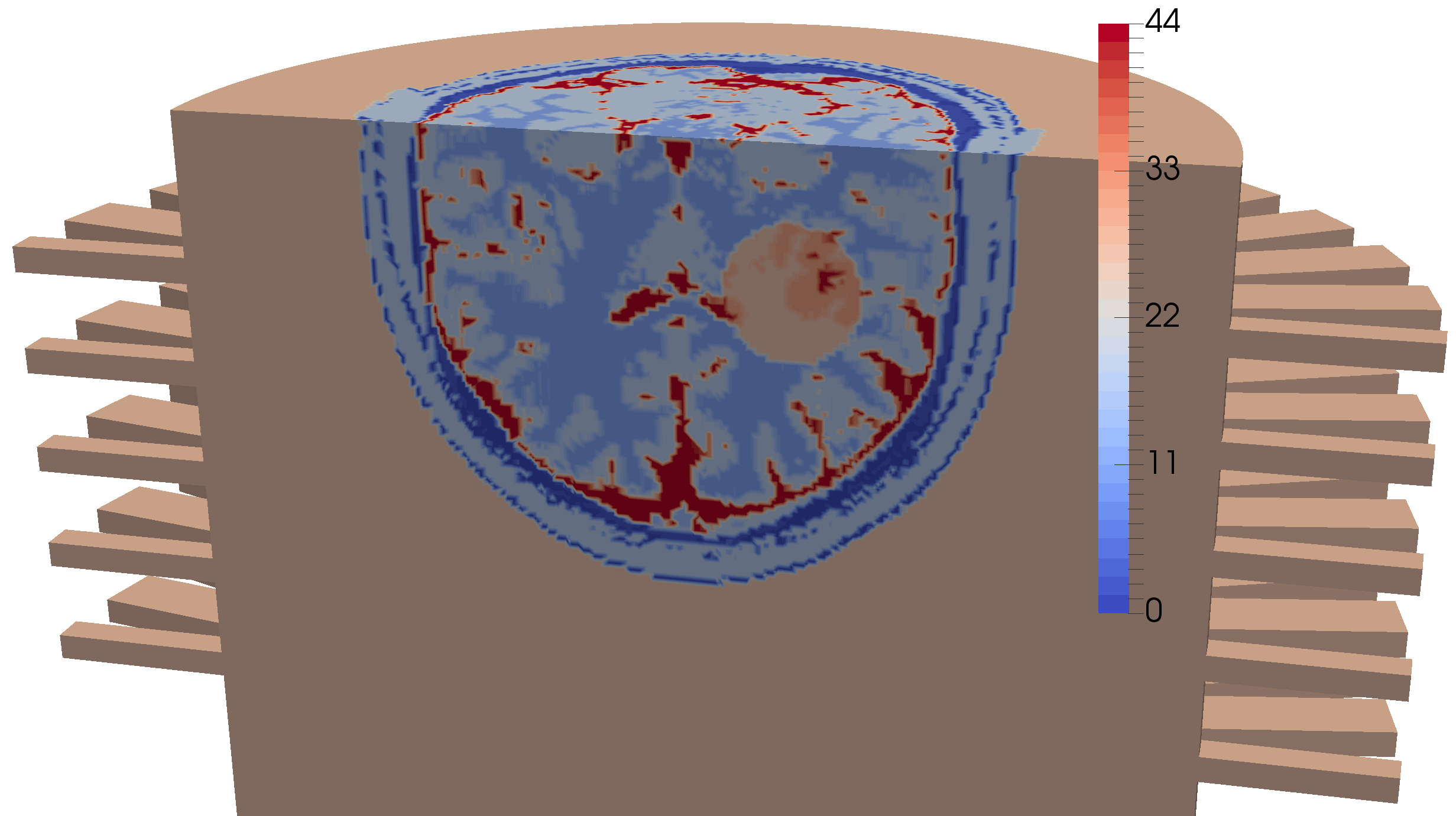} 
\caption{{The microwave imaging system prototype developed by EMTensor GmbH, the computational domain, and the imaginary part of the relative complex permittivity of a virtual head model immersed in the imaging chamber, with a simulated ellipsoid-shaped stroke.}} 
\label{fig:medimax}
\end{figure}

 \begin{table}
 \begin{center}
 \begin{tabular}{|c||c|c|c|c|}  
 \hline
  & \#2-level & Time & \#1-level & Time \\
 \hline
 {homogeneous} liquid & 28 & 63.4(8.6) & 30 & 53.1(6.4) \\
 head model & 28 & 64.1(9.2) & 32 & 53.4(6.9) \\
 {non-conductive} cylinder & 29 & 62.3(9.4) & 125 & 83.5(38.2) \\
 \hline
 \end{tabular}
 \caption{ImpHRAS for the microwave imaging system problem for three different material configurations inside the imaging chamber. In the Time columns we report the total time (the execution time for GMRES) in seconds. 
 } 
 \label{tab:medimax}
 \end{center}
 \end{table}

{As an example of a practical problem with $\xiprob > 0$,} we consider the modeling of a microwave imaging system, for the detection and monitoring of brain strokes. The prototype in Figure~\ref{fig:medimax} was developed by the company EMTensor GmbH and studied in the framework of the ANR project MEDIMAX \cite{paperPP:2016}. 
 In the full application, the data acquired with this device are used as input for an inverse problem associated with the time-harmonic Maxwell's equations~\eqref{eq:new1}, which makes it possible to estimate the complex electric permittivity $\varepsilon_{\sigma} := \varepsilon + \ri\, {\sigma}/{\omega}$ of the brain tissues of a patient affected by a stroke (observe that, with this definition, the coefficient $\mu(\eps \omega^2 +\ri \sigma\omega)$ in \eqref{eq:new1} can be rewritten as $\omega^2 \mu \varepsilon_{\sigma}$).
 Indeed, a stroke results in a variation of the complex electric permittivity inside a region of the brain, thus it can be detected and monitored by clinicians thanks to an image of the brain displaying the values of this property. 
 The absorption in this problem is given naturally by the non-null conductivity of brain tissues: the imaginary part of $\varepsilon_{\sigma}$ is typically of the same order as the real part, corresponding to $\xiprob\sim k^2$ in our notation. 
 
Here we solve the forward problem {(at frequency $1$\,GHz)} on the domain shown in Figure~\ref{fig:medimax} (center), discretized with order 1 edge finite elements using $40$ grid-points per wavelength, resulting in a linear system of size $n\approx1.6 \times 10^7$. 
 We take 729 subdomains (with minimal overlap) and for the two-level preconditioner a coarse problem of size $\ncs \approx 3.8 \times 10^4$.
 We test the two-level ImpHRAS preconditioner and the corresponding one-level version for three different material configurations: 
 the imaging chamber filled just with a homogeneous matching liquid, the virtual head model of Figure~\ref{fig:medimax}, and a plastic-filled cylinder immersed in the matching liquid; 
 the last test case is the most difficult because plastic is a non-conductive material. 
 In contrast to the previous experiments, we use the zero vector as the initial guess for GMRES because this gives lower iteration counts than a random initial guess.
 In Table~\ref{tab:medimax} we see that the performance of the one-level method deteriorates badly for the non-conductive cylinder, but the performance of the two-level method is uniform across all three cases, i.e.~it appears robust with respect to the type of heterogeneity.

\subsection{Electromagnetic scattering from a COBRA cavity} 
\label{subsec:cobra}

We now consider electromagnetic scattering from the COBRA cavity, which was designed and measured by EADS Aerospatiale Matra Missiles for Workshop EM-JINA 98 (see the description  in~\cite{liu2003scattering,jin2015finite}) and investigated in the framework of domain decomposition methods by, e.g.,~\cite{DoGaLaLePe:15}.

The cavity has PEC boundary conditions on its walls and we truncate the infinite domain of the scattering problem by a box enclosing the cavity. The sides of the box are placed 4 wavelengths away from the cavity in each direction, and the first order absorbing boundary condition~\eqref{eq:imp} is imposed on each side.
We consider a plane wave normally incident upon the cavity aperture for two frequencies, $f = 10$\,GHz and $f = 16$\,GHz, corresponding to $k=209$\,m$^{-1}$ and $k=335$\,m$^{-1}$ respectively.
The two problems are discretized with order $2$ edge elements using $10$ points per wavelength, resulting in $107$ million degrees of freedom at $10$\,GHz and $198$ million at $16$\,GHz. The mesh for the coarse problem corresponds to a discretization with $3.33$ points per wavelength.

\begin{figure}[h!]
\centering
\begin{minipage}[c]{.4\linewidth}
{\includegraphics[width=1\textwidth]{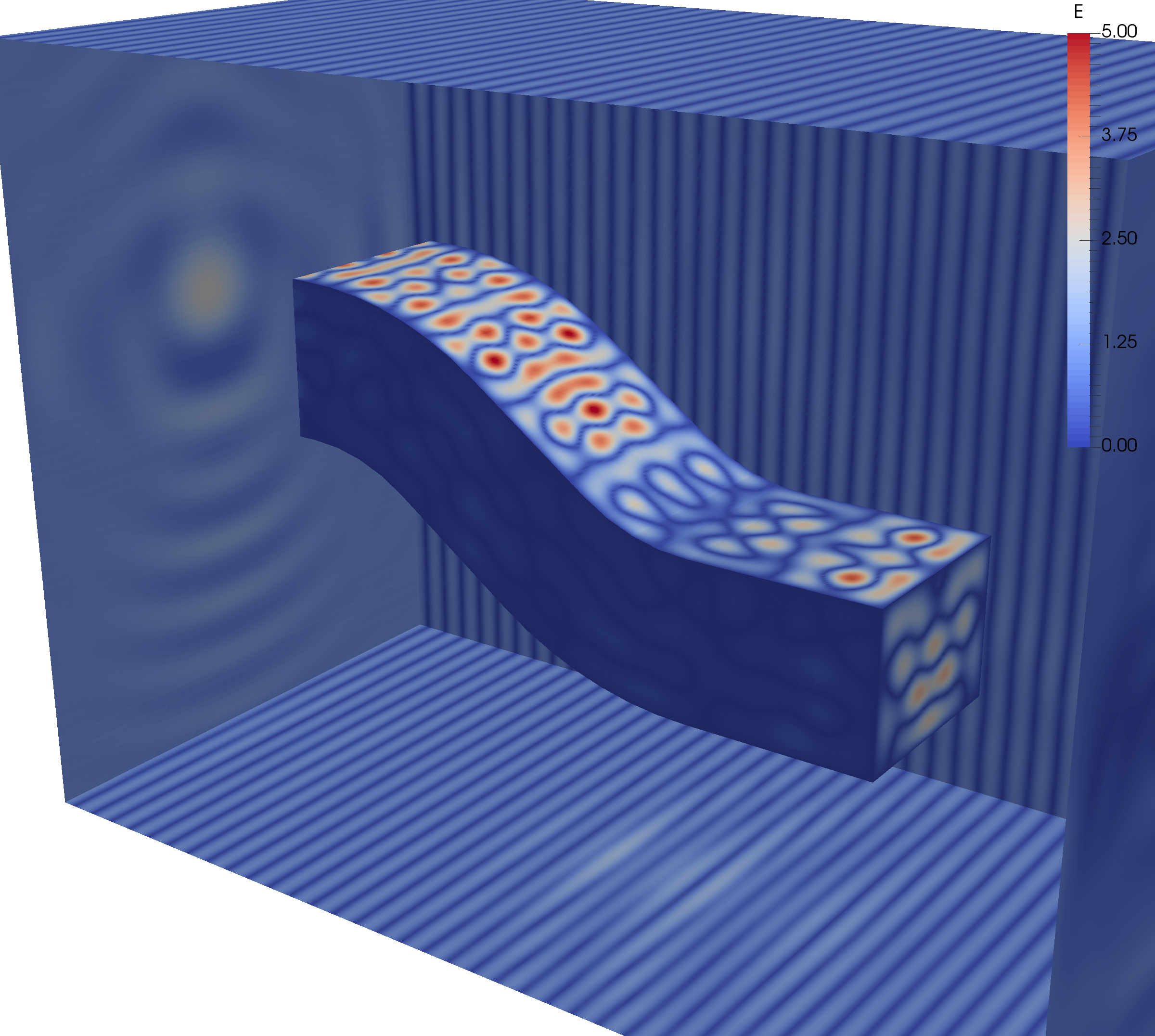}}
\end{minipage}
\hspace{.01\linewidth}
\begin{minipage}[c]{.5035\linewidth}
{\includegraphics[width=1\textwidth,trim=0cm 9cm 0cm 7cm,clip]{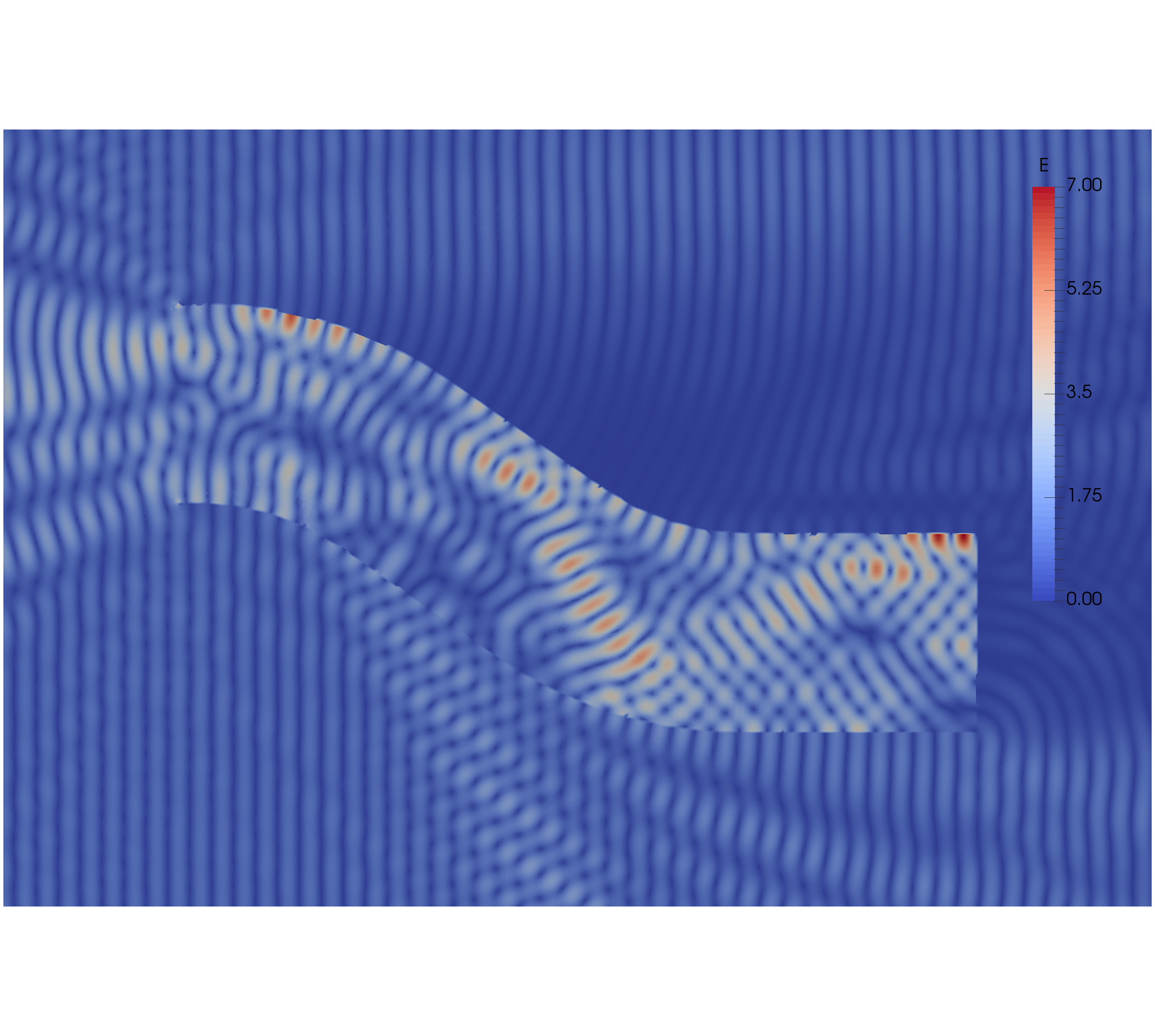}} 
\end{minipage}
\caption{Scattering from the COBRA cavity of a plane wave incident upon the cavity aperture at $f = 10$\,GHz (left) and $f = 16$\,GHz (right): magnitude of the electric field.} 
\label{fig:cobra}
\end{figure}

For such large simulations, the coarse problem becomes too large to be solved with a direct solver. We then need to use an iterative GMRES algorithm in order to solve the coarse problem at each application of the preconditioner. Following the ideas of~\cite[\S6]{GrSpVa:17}, we use the one-level ImpRAS preconditioner for the inner coarse solve. We  chose $\xiprec  =  \wn$ as a good balance between the convergence of the fine (outer) and coarse (inner) solves, since the effectiveness of the one-level preconditioner for the coarse solve improves as the added absorption becomes larger. We chose a relative tolerance of $10^{-1}$ for the inner GMRES coarse solve, which works surprisingly well, leading to only a few extra outer iterations in order to reach convergence.

\begin{table}[h!]
\begin{center}
\begin{tabular}{|c|c|c||c|c|c|c|c|c|c|}
\hline
\multicolumn{3}{|c||}{} & \multicolumn{7}{c|}{$g=10$, \quad $g_\text{cs} = 3.33$}\tabularnewline
\hline 
\multicolumn{1}{|c|}{} & \multicolumn{1}{c|}{} & \multicolumn{1}{c||}{} &\multicolumn{1}{c|}{} & \multicolumn{1}{c|}{} & \multicolumn{1}{c|}{Total \#} & \multicolumn{4}{c|}{Total times (seconds)}\tabularnewline \cline{7-10}
$f$ & $n$ & $N_\text{sub}$ & \# it & $\ncs$ & inner it & Total & Setup & GMRES & inner GMRES \tabularnewline
\hline
$10$GHz & $1.07 \times 10^8$ &1536 & 32 & $4.0 \times 10^6$ & 1527 & 515.8 & 383.2 & 132.6 & 61.8\\
$10$GHz & $1.07 \times 10^8$ & 3072 & 33 & $4.0 \times 10^6$ & 2083 & 285.0 & 201.6 & 83.4 & 40.6\\
\hline
$16$GHz & $1.98 \times 10^8$ & 3072 & 43 & $7.4 \times 10^6$ & 3610 & 549.2 & 336.8 & 212.4 & 118.6\\
$16$GHz & $1.98 \times 10^8$ & 6144 & 46 & $7.4 \times 10^6$ & 4744 & 363.0 & 210.5 & 152.5 & 96.8\\
\hline
\end{tabular}
\caption{A-DEF1-ImpRAS for a COBRA cavity, $\xiprob = 0$, $\xiprec = \wn$.}
\label{tab:cobra}
\end{center}
\end{table}

The results are displayed in Table~\ref{tab:cobra}, and Figure~\ref{fig:cobra} shows the computed scattered field. Experiments are performed with two different numbers of subdomains for both frequencies and indicate scalability of the algorithm. For $f = 10$\,GHz, we obtain a total time speedup of $1.81$ going from $1536$ to $3072$ cores. However, we observe weaker performance going to $16$\,GHz, with a speedup of $1.51$ from $3072$ to $6144$ cores, even though the number of outer GMRES iterations only increases from $43$ to $46$. This can be explained by the growth of the total number of inner coarse iterations: the total coarse solve time represents $11.8\%$ of the total time for $f = 10$\,GHz and 1536 subdomains, but goes up to $26.7\%$ of the total time for $f = 16$\,GHz and 6144 subdomains.

\smallskip
\noindent 
{\bf Acknowledgements.} 
For useful discussions, we thank Xavier Claeys (Universit\'e Pierre et Marie Curie), Jay Gopalakrishnan (Portland State University), Markus Melenk (TU Wien), Andrea Moiola (University of Reading), Peter Monk (University of Delaware), and Clemens Pechstein (Computer Simulation Technology (CST), Dahmstadt),
with the discussions with Moiola and Monk taking place at the LMS--EPSRC Durham Symposium on ``Mathematical and Computational Aspects of Maxwell's Equations" in July 2016.

This work was granted access to the HPC resources of TGCC at 
CEA and CINES
under the allocations 2016-067730 and 2017-ann7330 made by 
GENCI.
MB, VD, and PT gratefully acknowledge support from the French National Research Agency (ANR), project MEDIMAX, ANR-13-MONU-0012;
EAS gratefully acknowledges support from the EPSRC 
grant EP/R005591/1.

\begin{appendix}

\section{Proof of Lemma \ref{lem:reg1} ($\wn$- and $\abs$-explicit ${\bf H}^1$-regularity for $\Omega$ a convex polyhedron)}\label{app:2}

\bpf[Proof of Lemma \ref{lem:reg1}]
Our proof follows the proof of  \cite[Theorem 5.5.5]{Mo:11} (which in turn follows the reasoning in \cite{CoDaNi:10}).  
Let $\bE$ be the solution of the BVP \eqref{eq:BVP}; essentially identical arguments apply for the solution of the adjoint BVP \eqref{eq:BVPadj}.

We actually prove the more general result that, given $\bF\in {\bf L}^2(\Omega)\cap \Hdiv$ (i.e.~not necessarily with $\dive \bF=0$) then 
for $\abs\in\mathbb{R}\setminus\{0\}$ and $\wn>0$, the solution of the BVP \eqref{eq:BVP} is such that 
$\curl \bE\in {\bf H}^1(\Omega)$ and $\bE \in {\bf H}^1(\Omega)$,  
and furthermore, if $\abs$ satisfies \eqref{eq:alpha}, then, 
given $\wn_0>0$, 
\beq\label{eq:reg_result_app}
\N{\curl \bE}_{{\bf H}^1(\Omega)} + \wn\N{\bE}_{{\bf H}^1(\Omega)} \lesssim \Crego \wn\bigg( \big\|\curl \bE\big\|_{{\bf L}^2(\Omega)} + \wn\big\|\bE\big\|_{{\bf L}^2(\Omega)}\bigg) + \Cregt\N{\bF}_{{\bf L}^2(\Omega)}
+\frac{1}{\wn}\N{
\dive \bF
}_{\LtO},
\eeq
for all $\wn\geq\wn_0$.
Since $\bE\in \HocO$, we can use the regular decomposition (see, e.g., \cite[Lemma 2.4]{Hi:02}) to write $\bE$ as 
\beq\label{eq:reg0}
\bE= \bPhi + \grad \psi
\eeq
where $\bPhi\in {\bf H}^1(\Omega)$ and $\psi\in H_0^1(\Omega)$, with
\beq\label{eq:reg1}
\N{\bPhi}_{{\bf H}^1(\Omega)}+ \N{\psi}_{H^1(\Omega)} \lesssim \N{\bE}_{\HcO},
\eeq
where 
\beqs
\N{\bE}_{\HcO}:= \big(\N{\curl \bE}_{\bLtO}^2 + \N{\bE}_\bLtO^2\big)^{1/2}
\eeqs
(i.e.~$\|\cdot\|_{\weight}$ defined by \eqref{eq:weight} but without the $\wn^2$ weighting).
From \eqref{eq:reg0} we have
\beq\label{eq:reg1a}
\N{\bE}_{{\bf H}^1(\Omega)}\leq \N{\bPhi}_{H^1(\Omega)}+ \N{\psi}_{H^2(\Omega)},
\eeq
and so we need to obtain information about the $H^2$ norm of $\psi$. We first claim that $\Delta \psi\in L^2(\Omega)$. Indeed, from \eqref{eq:reg0} we have
\beq\label{eq:reg1b}
\Delta \psi = \dive\bE - \dive\bPhi.
\eeq
The PDE in \eqref{eq:BVP} 
implies that 
\beq\label{eq:reg1c}
\dive\bE= -\frac{1}{\wn^2+\ri \abs}\dive\bF,
\eeq
and thus $\dive \bE\in L^2(\Omega)$ since $\bF\in\Hdiv$.
The definition of $\bPhi$ implies that $\dive \bPhi\in L^2(\Omega)$, and so 
$\Delta \psi\in L^2(\Omega)$ follows from \eqref{eq:reg1b}.

The $H^2$ regularity theory for the Laplacian states that, if $u\in H_0^1(\Omega)$ with $\Delta u\in L^2(\Omega)$, then
\beq\label{eq:reg2}
\N{u}_{H^2(\Omega)} \lesssim \N{\Delta u}_{L^2(\Omega)} + \N{u}_{H^1(\Omega)},
\eeq
with this estimate holding \emph{both} when $\Omega$ is $C^{1,1}$ \cite[Theorem 2.3.3.2 (Page 106) and Theorem 2.4.2.5 (Page 124)]{Gr:85} \emph{and} when $\Omega$ is a convex polyhedron \cite[Equation 8.2.2, Page 357]{Gr:85}, \cite[Corollary 18.18]{Da:88a}. 
Using \eqref{eq:reg1}, \eqref{eq:reg1b}, and \eqref{eq:reg1c} in \eqref{eq:reg2} (with $u=\psi$), and recalling that $|\abs|\lesssim \wn^2$, 
we find 
\beqs
\|\psi\|_{H^2(\Omega)}\lesssim \N{\bE}_{\HcO} + \frac{1}{\wn^2} \N{\dive\bF}_{L^2(\Omega)},
\eeqs
and then using this, with \eqref{eq:reg1} again, in \eqref{eq:reg1a}, we find 
\beq\label{eq:reg2a}
\N{\bE}_{{\bf H}^1(\Omega)} \lesssim \N{\bE}_{\HcO}+ \frac{1}{\wn^2} \N{\dive\bF}_{L^2(\Omega)}.
\eeq

We now seek control of the ${\bf H}^1$ norm of $\curl \bE$. First, we claim that $\curl \bE \in \HcO$. Indeed, since $\bE \in \HocO$, we have that $\bE\in {\bf L}^2(\Omega)$, and the PDE in \eqref{eq:BVP} and the fact that $\bF\in {\bf L}^2(\Omega)$, then imply that $\curl(\curl \bE)\in {\bf L}^2(\Omega)$. Using the PDE and recalling the bound on $\abs$ \eqref{eq:alpha}, we then obtain
\beqs
\N{\curl(\curl \bE)}_{{\bf L}^2(\Omega)} \lesssim \wn^2 \N{\bE}_{{\bf L}^2(\Omega)} + \N{\bF}_{{\bf L}^2(\Omega)},
\eeqs
and thus 
\beq\label{eq:reg3}
\N{\curl \bE}_{\HcO} \lesssim\N{\curl \bE}_{{\bf L}^2(\Omega)} +\wn^2 \N{\bE}_{{\bf L}^2(\Omega)} +\N{\bF}_{{\bf L}^2(\Omega)}.
\eeq
We now use the other variant of the regular decomposition \cite[Lemma 2.4]{Hi:02} (recalling that, since $\Omega$ is simply connected, its first Betti number is zero) to write $\curl \bE\in \HcO$ as 
\beq\label{eq:reg3a}
\curl \bE = \bPsi^0 + \grad \phi,
\eeq
where $\bPsi^0\in H^1(\Omega)$ with $\dive\bPsi^0 =0$ and $\phi\in H^1(\Omega)$, and 
\beq\label{eq:reg4}
\N{\bPsi^0}_{{\bf H}^1(\Omega)} + \N{\phi}_{H^1(\Omega)} \lesssim \N{\curl \bE}_{\HcO}.
\eeq
From \eqref{eq:reg3a} we have
\beq\label{eq:reg5}
\N{\curl\bE}_{{\bf H}^1(\Omega)} \leq \N{\bPsi^0}_{{\bf H}^1(\Omega)} + \N{\phi}_{H^2(\Omega)},
\eeq
and, analogous to before, we need to obtain information about the $H^2$ norm of $\phi$.

Since $\dive(\curl\bE)=0$ and $\dive \bPsi^0=0$ we have that $\Delta \phi=0$. Furthermore, since 
$\bn\cdot( \curl \bv )= \dive_T(\bv\times \bn)$
\cite[Equation 6.38]{CoKr:98}, where $\dive_T$ is the surface divergence (see, e.g., \cite[\S3.4]{Mo:03}), we have $\bn \cdot (\curl \bE )=0$ on $\po$, and thus \eqref{eq:reg3a} implies that
\beqs
\partial_n \phi= -\bn \cdot \bPsi^0 \quad\ton\po.
\eeqs
We now need some information about the regularity of $\partial_n \phi$. When $\po$ is $C^{1,1}$, $\bn$ is $C^{0,1}$ and then $\bn \cdot\bPsi^0 \in H^{1/2}(\po)$ 
with 
\beqs
\N{\bn\cdot\bPsi^0}_{H^{1/2}(\po)} \lesssim \N{\bPsi^0}_{{\bf H}^{1/2}(\po)}
\eeqs
(where the omitted constant depends on $\po$) by either \cite[Theorem 1.4.1.1, page 21]{Gr:85} (with $p=2,\, s=1/2,\,k=0,$ and $\alpha=1$) or \cite[Page 140, \S2.8.2]{Tr:83}. 
Thus 
\beq\label{eq:reg7}
\N{\partial_n \phi}_{H^{1/2}(\po)}\lesssim \N{\bPsi^0}_{{\bf H}^1(\Omega)}
\eeq
by standard trace results (see, e.g., \cite[Theorem 3.37]{Mc:00}.

When $\Omega$ is a polyhedron, $\bn$ is piecewise $C^\infty$, and by the previous argument 
\beq\label{eq:reg7a}
\N{\partial_n \phi}_{H^{1/2}_{\rm pw}(\po)}\lesssim \N{\bPsi^0}_{{\bf H}^1(\Omega)},
\eeq
where now the relevant trace result is given by, e.g., \cite[Corollary 4.3]{BeDaMa:07}.
The $H^2$-regularity theory for the Laplacian states that, when $\Omega$ is $C^{1,1}$ and $u\in H^1(\Omega)$ with $\Delta u\in L^2(\Omega)$ and $\partial_n u \in H^{1/2}(\po)$, then
\beq\label{eq:reg8}
\N{u}_{H^2(\Omega)} \lesssim \N{\Delta u}_{L^2(\Omega)} + \N{u}_{H^1(\Omega)}+ \N{\partial_n u}_{H^{1/2}(\po)}
\eeq
 \cite[Theorem 2.3.3.2, Page 106]{Gr:85}.  
When $\Omega$ is a convex polyhedron, \eqref{eq:reg8} holds with the requirement $\partial_n u \in H^{1/2}(\po)$ replaced by $\partial_n u \in H^{1/2}_{\rm pw}(\po)$ and 
$\|\partial_n u\|_{H^{1/2}(\po)}$ on the right-hand side replaced by $\|\partial_n u\|_{H^{1/2}_{\rm pw}(\po)}$; this follows from \cite[Corollary 3.12]{Da:92} (see also \cite[Theorem 4]{Da:08}).

Using \eqref{eq:reg7}/\eqref{eq:reg7a} in \eqref{eq:reg8} with $u=\phi$, we find 
\beqs
\|\phi\|_{H^2(\Omega)}\lesssim \N{\Psi^0}_{{\bf H}^1(\Omega)}+ \N{\phi}_{H^1(\Omega)},
\eeqs
and then using \eqref{eq:reg4} we find $\|\phi\|_{H^2(\Omega)}\lesssim \N{\curl \bE}_{\HcO}$.
Using this, with \eqref{eq:reg4} again, in \eqref{eq:reg5}, we find 
\beqs
\N{\curl \bE}_{{\bf H}^1(\Omega)}\lesssim  \N{\curl \bE}_{\HcO}.
\eeqs
Combining this last inequality with \eqref{eq:reg3} and \eqref{eq:reg2a}, 
we obtain
\begin{align*}
\N{\curl \bE}_{{\bf H}^1(\Omega)}+ \wn \N{\bE}_{{\bf H}^1(\Omega)}
\lesssim \left( 1 + \wn\right) \N{\curl \bE}_\bLtO  
+ \wn^2 \N{\bE}_\bLtO +\N{\bF}_\bLtO+ \frac{1}{\wn} \N{\dive\bF}_{L^2(\Omega)}.
\end{align*}
the result \eqref{eq:reg_result_app} follows.
\epf

\section{The relative error,  at the level of PDEs, in approximating by the solution with absorption}
\label{app:1}

Recall that in \S\ref{sec:analyse} we split the analysis of preconditioning with absorption into achieving the two properties (i) 
$\matrixAepsinv$ is  a good preconditioner for $\matrixA$ 
and (ii) $\matrixBepsinv$ is  a good
preconditioner for $\matrixAeps$.
We now prove the PDE analogue of Property (i) for the interior impedance problem in heterogeneous media, i.e.~the result that, as $\omega\tendi$, the relative error between the solutions of the heterogeneous versions of \eqref{eq:1} and \eqref{eq:2}, supplemented with impedance boundary conditions, 
can be bounded to a prescribed accuracy (independent of $\omega$) if $|\abs|/\omega$ is sufficiently small.
The Helmholtz analogue of this result (in homogeneous media) is proved in \cite[Theorem 6.1]{GaGrSp:15}.

\begin{assumption}\label{ass:A1}
$\Omega$ is a bounded Lipschitz domain, $\eps, \mu\in L^{\infty}(\Omega)$, and the solution  $\bE$ of the interior impedance problem
\beqs
\curl\left(\frac{1}{\mu}
\curl \bE\right) - \omega^2 \eps \,\bE= \bF \,\,\tin\Omega
\,\,
\tand 
\,\,
\left(\frac{1}{\mu}
\curl \bE\right)\times \bn - \ri \omega (\bn \times\bE)\times \bn = \bze \,\,\ton\partial\Omega,
\eeqs
exists and satisfies the bound
\beq\label{eq:bound}
\N{\bE}_{\curlt, \omega}
\leq C_1(\eps,\mu)\N{\bF}_{{\bf L}^2(\Omega)}
\eeq
for all $\omega\geq \omega_0$, for some $\omega_0>0$.
\end{assumption}

\bre\textbf{\emph{(When is Assumption \ref{ass:A1} satisfied?)}}
When $\Omega$ is star-shaped with respect to a ball
and either 
\ben
\item $\eps= \mu = 1$ \cite[Theorem 5.4.5]{Mo:11}/\cite[Theorem 3.3]{HiMoPe:11}, or 
\item $\eps,\mu$ satisfy certain monotonicity conditions (which allow $\eps, \mu$ to jump), with these  conditions guaranteeing nontrapping of rays \cite{MoSp:17}.
\een
\ere

\begin{theorem}
Assume that $\eps, \mu \in L^\infty
(\Omega)$ are such that the solution  $\bE_\abs$ of 
\beqs
\curl\left(\frac{1}{\mu}
\curl \bE_\abs\right) - (\omega^2+ \ri \abs) \eps \,\bE_\abs= \bF \,\,\tin\Omega
\,\,
\tand
\,\,
\left(\frac{1}{\mu}
\curl \bE_\abs\right)\times \bn - \ri \omega (\bn \times\bE_\abs)\times \bn = \bg\,\,\ton\partial\Omega
\eeqs
exists for all $\abs\geq 0$; let $\bE:=\bE_0$. If Assumption \ref{ass:A1} holds, then there exists $C_2(\eps,\mu), C_3(\eps,\mu)>0$  and $\omega_0>0$ such that if 
\beqs
\frac{\abs}{\omega}\leq C_2(\eps,\mu) \quad\tfa \omega\geq \omega_0, \text{ then } \quad
\frac{
\N{\bE-\bE_\abs}_{\curlt,\omega}
}{
\N{\bE}_{\curlt,\omega}
}
\leq C_3(\eps,\mu) \frac{\abs}{\omega} \quad\tfa \omega\geq \omega_0.
\eeqs
\end{theorem}

\bpf
Subtracting the equations satisfied by $\bE$ and $\bE_\abs$, we find that $\bE-\bE_\abs$ satisfies the BVP
\beqs
\curl\left(\frac{1}{\mu}
\curl \big(\bE-\bE_\abs\big)\right) - \omega^2 \eps \,\big(\bE-\bE_\abs\big)= -\ri \abs \eps\bE_\abs \quad\tin\Omega
\eeqs
and
\beqs
\left(\frac{1}{\mu}
\curl \big(\bE-\bE_\abs\big)\right)\times \bn - \ri \omega \big(\bn \times(\bE-\bE_\abs)\big)\times \bn =\bze \quad\ton\po.
\eeqs
Applying the bound \eqref{eq:bound} and then using the triangle inequality, we have
\begin{align*}
\N{\bE-\bE_\abs}_{\curlt,\omega}&\leq C_1(\eps,\mu) \N{\eps}_{L^\infty(\Omega)}\abs \N{\bE_\abs}_{{\bf L}^2(\Omega)} \leq C_1(\eps,\mu)  \N{\eps}_{L^\infty(\Omega)}\abs \left(  \N{\bE-\bE_\abs}_{{\bf L}^2(\Omega)} + \N{\bE}_{{\bf L}^2(\Omega)}\right),\\
& \leq C_1(\eps,\mu)  \N{\eps}_{L^\infty(\Omega)} \frac{\abs}{\omega} \left(  \N{\bE-\bE_\abs}_{\curlt,\omega} + \N{\bE}_{\curlt, \omega}\right).
\end{align*}
Letting
\beqs
C_2(\eps,\mu):= \frac{1}{2C_1(\eps,\mu) \N{\eps}_{L^\infty(\Omega)}} \quad\tand\quad C_3(\eps,\mu):= 2C_1(\eps,\mu) \N{\eps}_{L^\infty(\Omega)},
\eeqs
the result follows.
\epf

\end{appendix}

\footnotesize{

}

\end{document}